\documentclass[a4paper,12pt]{amsart}
\usepackage{amssymb,amsthm,amsmath,color,url}
\usepackage[margin=3cm]{geometry}

\usepackage{graphicx}
\usepackage{stmaryrd}
\usepackage[T1]{fontenc}
\usepackage[english]{babel}
\usepackage{comment}
\usepackage{hyperref}
\usepackage{ifthen}
\hypersetup{
  colorlinks   = true,
  linkcolor = red!70!black,
     citecolor    = green!50!black
}
\usepackage[noabbrev,capitalise]{cleveref}
\crefname{subsection}{Subsection}{Subsections}
\crefname{claim}{Claim}{Claims}
\crefname{problem}{Problem}{Problems}

\makeatletter
\def\namedlabel#1#2{\begingroup
   \def\@currentlabel{#2}%
   \label{#1}\endgroup
}
\makeatother

\usepackage{thmtools}
\usepackage{thm-restate}
\usepackage{enumitem}
\usepackage{caption,subcaption}

\usepackage{tikz}
\usetikzlibrary{shadows}
\usetikzlibrary{backgrounds}
\usetikzlibrary{arrows}
\usetikzlibrary{patterns}
\usetikzlibrary{calc}
\usetikzlibrary{shapes}
\usetikzlibrary{positioning}
\usetikzlibrary{math}
\usetikzlibrary{external}
\usetikzlibrary{decorations.pathreplacing}
\declaretheorem[name=Theorem, numberwithin=section]{theorem}
\declaretheorem[name=Lemma, sibling=theorem]{lemma}
\declaretheorem[name=Proposition, sibling=theorem]{proposition}

\declaretheorem[name=Corollary, sibling=theorem]{corollary}
\declaretheorem[name=Conjecture, sibling=theorem]{conjecture}
\declaretheorem[name=Problem, sibling=theorem]{problem}
\declaretheorem[name=Claim, sibling=theorem]{claim}
\declaretheorem[name=Claim, numbered=no]{claim*}

\newenvironment{proofofclaim}{\noindent{\emph{Proof of the Claim}:}}{\hfill$\Diamond$\medskip}
\declaretheorem[name=Remark, style=remark, sibling=theorem]{remark}
\declaretheorem[name=Example, style=remark, sibling=theorem]{example}

\def\cqedsymbol{\ifmmode$\lrcorner$\else{\unskip\nobreak\hfil
\penalty50\hskip1em\null\nobreak\hfil$\lrcorner$
\parfillskip=0pt\finalhyphendemerits=0\endgraf}\fi} 
\newcommand*\mean[1]{\overline{#1}}

\newcommand{\Cay}{\mathrm{Cay}}
\newcommand{\Aut}{\mathrm{Aut}}
\newcommand{\GL}{G^{\backslash L /}}
\newcommand{\RL}{R^{\backslash L /}}

\newcommand{\fc}{\mathrm{fc}}
\newcommand{\TL}{\mathcal{T}^{\backslash L /}}
\newcommand{\Tnd}{\mathcal{T}_{\mathrm{nd}}}
\newcommand{\Stab}{\mathrm{Stab}}

\newcommand*\torso[1]{\llbracket #1 \rrbracket}
\newcommand{\Sep}{(Y,S,Z)}
\newcommand{\Kinf}{K_{\infty}}
\newcommand{\Sepp}{(Y',S',Z')}
\newcommand{\Seppp}{(Y'',S'',Z'')}
\newcommand*\Sepi[1]{(Y_{#1},S_{#1},Z_{#1})}
\newcommand*\Seppi[1]{(Y'_{#1},S'_{#1},Z'_{#1})}
\newcommand*\Sepppi[1]{(Y''_{#1},S''_{#1},Z''_{#1})}
\newcommand{\Pres}{\langle S | R \rangle}
\newcommand*\sg[1]{\{ #1 \}}
\newcommand*\Contr[1]{^{\backslash #1 /}}
\newcommand{\GM}{G\Contr{M}}

\DeclareMathOperator{\separ}{Sep}

\newcommand*\Expand[1]{_{/ #1 \backslash}}
\newcommand*\Exnd[1]{E^\times_{\mathrm{nd}}(#1)}

\def\H{\mathcal{H}} 

\def\M{\mathcal{M}} 
\newcommand{\NN}{\mathbb{N}} 
\def\T{\mathcal{T}} 
\def\S{\mathcal{S}} 

\let\le\leqslant
\let\ge\geqslant
\let\leq\leqslant
\let\geq\geqslant

\begin{document}
	
	\title[The structure of quasi-transitive graphs avoiding a
          minor]{The structure of quasi-transitive graphs avoiding a
          minor with applications to the domino problem}

\author[L.~Esperet]{Louis Esperet}
\address[L.~Esperet]{Univ.\ Grenoble Alpes, CNRS, Laboratoire G-SCOP,
  Grenoble, France}
\email{louis.esperet@grenoble-inp.fr}

\author[U.~Giocanti]{Ugo Giocanti}
\address[U.~Giocanti]{Univ.\ Grenoble Alpes, CNRS, Laboratoire G-SCOP,
  Grenoble, France}
\email{ugo.giocanti@grenoble-inp.fr}

\author[C.~Legrand-Duchesne]{Cl\'ement Legrand-Duchesne}
\address[C.~Legrand-Duchesne]{CNRS, LaBRI, Université de Bordeaux, Bordeaux, France}
\email{clement.legrand@u-bordeaux.fr}

\thanks{All authors are partially supported by the French ANR Project
  GrR (ANR-18-CE40-0032). L.\ Esperet and U. Giocanti are partially supported by the French ANR Project TWIN-WIDTH
  (ANR-21-CE48-0014-01), and by LabEx
  PERSYVAL-lab (ANR-11-LABX-0025).}

        \begin{abstract}
       An infinite graph is quasi-transitive if its vertex set has finitely many orbits under the action of its automorphism group. In this paper we obtain  a structure
       theorem for locally  finite
       quasi-transitive graphs avoiding a minor, which is reminiscent
       of the Robertson-Seymour Graph Minor Structure Theorem. We
       prove that every locally finite quasi-transitive graph $G$ avoiding a
       minor has a
       tree-decomposition whose torsos are finite or planar; moreover
       the tree-decomposition is canonical, i.e.\ invariant under the action of the automorphism
       group of $G$. As applications of this result, we prove the following.
       \begin{itemize}
          \item Every locally finite quasi-transitive graph attains its
           Hadwiger number, that is, if such a graph contains
           arbitrarily large clique minors, then it contains an infinite
           clique minor. This extends a result of Thomassen (1992) who
           proved it in the (quasi-)4-connected case and suggested
           that this assumption could be omitted. In particular, this shows that a Cayley graph excludes a finite minor if and only if it avoids the countable clique as a minor.
       \item Locally finite quasi-transitive graphs avoiding a minor
         are accessible (in
         the sense of Thomassen and Woess), which extends known results
         on planar graphs to any proper minor-closed family.
       \item Minor-excluded finitely generated groups are accessible
         (in the group-theoretic sense) and finitely presented, which
         extends classical results on planar groups.
       \item The domino problem is decidable in a minor-excluded
         finitely generated group if and only if the group is
         virtually free, which proves the minor-excluded  case of a
         conjecture of Ballier and Stein (2018).
       \end{itemize}
     \end{abstract}
        
	\maketitle

        \section{Introduction}

\subsection{A structure theorem}
        
        A central result in modern graph theory is the Graph
        Minor Structure Theorem of Robertson and Seymour
        \cite{RS-XVI}, later extended to infinite graphs by Diestel
        and Thomas
        \cite{DT99}. This theorem states that any graph $G$ avoiding a fixed
        minor has a \emph{tree-decomposition}, such that each piece of the
        decomposition, called a \emph{torso}, is close to being embeddable on a surface of
        bounded genus (the notions of
        tree-decomposition and torso will be defined  in the next
        section). A natural question is the following: if the
        graph $G$ has non trivial symmetries, can we make these
        symmetries apparent in the tree-decomposition? In other
        words, do graphs avoiding a fixed minor have a tree-decomposition as
        above, but with the additional constraint that the
        decomposition is \emph{canonical}, i.e., invariant under the action of the
        automorphism group of $G$? In this paper we answer this
        question positively for infinite, locally finite graphs $G$ that
        are \emph{quasi-transitive}, i.e., the vertex set of $G$ has finitely many orbits under the action of the automorphism group of $G$. This additional restriction, which is a way of saying that the graph $G$ is highly symmetric, has the
        advantage of making the structure theorem much cleaner:
        instead of being almost embeddable on a surface of bounded
        genus, each torso of the tree-decomposition is now simply
        finite or planar.

        \begin{theorem}[see \cref{thm: mainCTTD}]
        \label{intro:main}
          Every locally finite quasi-transitive graph avoiding the countable clique as a minor
          has a canonical tree-decomposition whose torsos are finite
          or planar.
        \end{theorem}

        The tree-decomposition in Theorem \ref{intro:main} will be obtained by refining the tree-decomposition obtained in the following more detailed version of the result, which might be useful for applications.
        
\begin{theorem}[see \cref{thm: main2}]
\label{intro:main2}
          Every locally finite quasi-transitive graph $G$ avoiding the countable clique as a minor
          has a canonical tree-decomposition with adhesion at most 3 in which each torso is a minor of $G$, and is planar or has bounded treewidth. 
        \end{theorem}

        Interestingly, the proof does not use the original structure
        theorem of Robertson and Seymour \cite{RS-XVI} or its extension to infinite
        graphs by Diestel and Thomas \cite{DT99}. Instead, we rely mainly on a series
        of results and tools introduced by Grohe \cite{Gro16}  to study decompositions
        of finite 3-connected graphs into quasi-4-connected components, together
        with a result of Thomassen \cite{Thomassen92} on  locally finite quasi-4-connected
        graphs. The main technical contribution of our work consists in  extending the results of Grohe  to infinite, locally finite graphs and in addition, making sure that the decompositions we obtain are canonical  (in a certain weak sense). Our proof crucially relies on a recent result of Carmesin, Hamann, and Miraftab \cite{CTTD}, which shows that there exists a canonical tree-decomposition that distinguishes all tangles of a given order (in our case, of order 4). 

        Thomassen proved that if a locally finite quasi-transitive graph has only one end, then this end must be thick \cite[Proposition 5.6]{Thomassen92} (see below for a definition of thick end). At some point of our proof, we also need to show the stronger result (see \cref{prop: thin-ends}), of independent interest, that for any $k\ge 1$,  a locally finite quasi-transitive graph cannot have only one end of degree $k$.

        \medskip

        We now discuss some applications of \cref{intro:main}.

        \subsection{Hadwiger number}

        As a consequence of \cref{intro:main}, we obtain a result on the Hadwiger number of locally finite
        quasi-transitive graphs. The \emph{Hadwiger number} of a graph
        $G$ is the supremum of the sizes of all finite complete minors in $G$. We
        say that a graph $G$ \emph{attains its Hadwiger number} if the
        supremum above is attained, that is if it is either finite, or
        $G$ contains an infinite clique minor. Thomassen
        \cite{Thomassen92} proved that every locally finite
        quasi-transitive 4-connected graph attains its Hadwiger
        number, and suggested that the 4-connectedness assumption
        might be unnecessary. We prove that this is indeed the case.

        \begin{theorem}[see \cref{thm: had}]
        \label{intro:had}
          Every locally finite
        quasi-transitive graph attains its Hadwiger
        number.
      \end{theorem}

      We will indeed prove a  stronger statement, namely that every locally finite quasi-transitive graph avoiding the countable clique as a minor also avoids a finite graph with crossing number 1 as a minor.

      \subsection{Accessibility in graphs}

We now introduce the notion of accessibility in graphs considered by Thomassen
and Woess \cite{TW}. To distinguish it from the related notion in
groups (see below), we
will call it vertex-accessibility in the remainder of the paper. A
\emph{ray} in an infinite graph $G$ is an infinite one-way path in
$G$. Two rays of $G$ are
\emph{equivalent} if there are infinitely many disjoint paths between them in
$G$ (note that this is indeed an equivalence relation).
An \emph{end} of $G$ is an
equivalence class of rays in $G$. When there is a finite set $X$ of vertices
of $G$, two distinct components $C_1,C_2$ of $G- X$, and two distinct ends
$\omega_1,\omega_2$ of $G$ such that for each $i=1,2$, all but finitely many vertices of all (equivalently any) rays of $\omega_i$ are in $C_i$, we say that $X$
\emph{separates} $\omega_1$ and $\omega_2$.  A graph $G$ is
\emph{vertex-accessible} if there is an integer $k$ such that for any
two distinct ends $\omega_1,\omega_2$ in $G$, there is a set of at most $k$
vertices that separates $\omega_1$ and $\omega_2$.

\medskip

It was proved by Dunwoody \cite{Dunwoody07} (see also \cite{HamannPlanar,HamannAccessibility} for a more combinatorial approach) that
locally finite quasi-transitive
planar graphs are vertex-accessible. Here we extend the result to locally finite quasi-transitive graphs
excluding the countable clique $K_\infty$ (and not necessarily $K_5$ and $K_{3,3}$) as a minor, and in particular to locally finite quasi-transitive graphs from any
proper minor-closed family.

   \begin{theorem}[see \cref{thm: access}]
   \label{intro:vacc}
          Every locally finite
        quasi-transitive $K_\infty$-minor-free graph  is vertex-accessible.
      \end{theorem}

      \subsection{Accessibility in groups}

      The notion of vertex-accessibility introduced above is related to the notion of accessibility in
groups. Given a finitely generated group $\Gamma$, and a finite set of
generators $S$, the \emph{Cayley graph} of $\Gamma$ with respect to
the set of generators $S$ is the edge-labeled graph 
$\Cay(\Gamma,S)$ whose vertex set is the set of elements of $\Gamma$ and where for every two elements $g,h\in \Gamma$ we put an arc $(g,h)$ labeled with $a\in S$ when $h=a\cdot g$. Cayley graphs have to be seen as highly symmetric graphs; in particular they are transitive: the right action of the group $\Gamma$ onto itself can be easily seen to induce a transitive group action on $\Cay(\Gamma,S)$.
It is known that the number of ends of a Cayley
graph of a finitely generated group does not depend of the choice of
generators, so we can talk about the number of ends of a finitely generated
group. A classical theorem of Stallings \cite{Sta72} states that if a
finitely generated group $\Gamma$ has more than one end, it can be
split as a non-trivial free product with finite amalgamation, or as an
HNN-extension over a finite subgroup. If any group produced by the
splitting still has more than one end we can keep splitting it using
Stallings theorem. If the process eventually stops (with $\Gamma$
being obtained from finitely many
0-ended or 1-ended groups using free products with
amalgamation and HNN-extensions), then $\Gamma$ is said to be
\emph{accessible}. Thomassen
and Woess \cite{TW} proved that a finitely generated group is
accessible if and only if at least one of its locally finite Cayley graphs is
vertex-accessible, if and only if all of its locally finite Cayley graphs are
vertex-accessible.

\medskip

A finitely generated group is \emph{minor-excluded} if at least one of its Cayley graphs
avoids a finite minor. Similarly a finitely generated group is \emph{$K_\infty$-minor-free} if one of its Cayley graphs avoids the countable clique as a minor, and \emph{planar} if one of its Cayley graphs
is planar. Note that planar groups are minor-excluded and \cref{intro:had} immediately implies that a finitely generated group is minor-excluded if and only if it is $K_\infty$-minor-free.

\medskip

Droms \cite{Droms} proved that finitely generated planar groups are finitely
presented, while Dunwoody \cite{Dunwoody1985} proved that finitely presented groups are accessible, which implies that finitely generated planar groups are accessible.
\cref{intro:vacc} immediately implies the following, which
extends this result to all minor-excluded finitely generated groups, and equivalently to all finitely generated $K_\infty$-minor-free groups.

  \begin{corollary}\label{intro:acc}
          Every  finitely generated $K_\infty$-minor-free group is accessible.
        \end{corollary}

        In fact, combining \cref{intro:main} with techniques
        introduced by Hamann \cite{HamannPlanar,HamannAccessibility}
        in the planar case, we prove the
        following stronger result which also implies \cref{intro:acc} using the result of Dunwoody \cite{Dunwoody1985} that all finitely
        presented groups are accessible.

         \begin{theorem}[see \cref{cor: finpres}]\label{intro:fp}
          Every  finitely generated $K_\infty$-minor-free group is finitely presented.
        \end{theorem}
      
        \subsection{The domino problem}

        We refer to \cite{Aubrun18subshift} for a detailed
        introduction to the domino problem. A \emph{coloring} of a graph $G$
        with colors from a set $\Sigma$ is simply a map $V(G)\to \Sigma$.
The \emph{domino problem} for a finitely generated group $\Gamma$ together
with a finite generating set $S$ is defined as follows. The input is a finite alphabet $\Sigma$ and a finite set $\mathcal F=\{F_1,\ldots,F_p\}$ of
forbidden \emph{patterns}, which are colorings with colors from
$\Sigma$ of the closed neighborhood of the neutral element
$1_{\Gamma}$ in the  Cayley graph $\Cay(\Gamma,S)$, viewed as an
edge-labeled subgraph of $\Cay(\Gamma,S)$ (recall that each element of $S$ corresponds to a different label). The problem then asks if
there is a coloring of $\Cay(\Gamma,S)$ with colors from $\Sigma$,
such that for each  $v\in \Gamma$, the coloring of the closed
neighborhood of $v$ in $\Cay(\Gamma,S)$ (viewed as an edge-labeled
subgraph of  $\Cay(\Gamma,S)$), is not isomorphic to any of the
colorings $F_1,\ldots,F_p$, where we consider isomorphisms preserving
the edge-labels (equivalently, isomorphisms corresponding to the right multiplication by elements  of $\Gamma$).

\medskip

It turns out that the decidability of the domino problem for
$(\Gamma,S)$ is independent of the choice of the finite generating set
$S$, hence we can talk of the decidability of the domino problem for a
finitely generated group $\Gamma$.
If we consider $\Gamma=(\mathbb Z^2, +)$, then the domino problem
corresponds exactly to the well-known Wang tiling problem, which was
shown to be undecidable by Berger in \cite{Berger}. On the other hand,
there is a simple greedy procedure to solve the domino problem in free
groups, which admit trees as Cayley graphs. More generally, the domino
problem is decidable in virtually free groups, which can equivalently
be defined as finitely generated groups having a locally finite Cayley graph of bounded treewidth \cite{Aubrun18subshift, Pichel09}. 
A remarkable conjecture of Ballier and Stein \cite{BS}  asserts that these groups are the only ones for which the domino problem is decidable.

\begin{conjecture}[Domino problem conjecture \cite{BS}]
\label{conj: DP}
A finitely generated group has a decidable domino problem if
and only if it is virtually free.
\end{conjecture}

Recall that virtually free groups are precisely the groups having a locally finite Cayley
graph of bounded treewidth. Since having bounded treewidth is a property that is closed under
taking minor, it is natural to
ask whether \cref{conj: DP}  holds for minor-excluded groups (or equivalently, using \cref{intro:had}, to $K_\infty$-minor-free groups). Using
\cref{intro:acc}, together with classical results on planar
groups and recent results on fundamental groups of surfaces \cite{Aubrun18surface}, we prove that
this is indeed the case.

\begin{theorem}[see \cref{thm: Domino}]\label{intro:domino}
A finitely generated $K_\infty$-minor-free group has a decidable domino problem if
and only if it is virtually free.
\end{theorem}

\subsection{Overview of the proof of \cref{intro:main,intro:main2}} 
Consider a locally finite quasi-transitive graph $G$ that excludes the countable clique $K_\infty$ as a minor. The graph $G$ is said to be \emph{quasi-$4$-connected} if it is $3$-connected 
and for every set $S\subseteq V(G)$ of size $3$ such that $G- S$ is not
connected, $G- S$ has exactly two connected components and one of them
consists of a single vertex. Thomassen proved that if $G$ is quasi-4-connected, then $G$ is planar or has finite treewidth \cite{Thomassen92}, which implies \cref{intro:main2} in this case (with a trivial tree-decomposition consisting of a single node).
 
To deal with the more general case, the first step is to obtain a canonical tree-decomposition of $G$ of adhesion at most 2 in which all torsos are minors of $G$ that are 3-connected graphs, cycles, or complete graphs on at most 2 vertices. The existence of such a decomposition in the finite case is a well-known result of Tutte \cite{Tutte} and was proved in the locally finite case in \cite{DSS98}. 
For our proof we  need to go one step further. Grohe \cite{Gro16}  proved that every finite graph $G$ has  a tree-decomposition  of adhesion at most $3$ whose torsos are minors of $G$ and are complete graphs on at most 4 vertices or  quasi-$4$-connected graphs. A crucial step for us would be to prove a version of this result in which the tree-decomposition would be canonical, and which would hold for locally finite graphs.

However, as observed by Grohe, even in the finite case the decomposition he obtains is not canonical in general.
Our main technical contribution is to extend the resut of Grohe \cite{Gro16} mentioned in the previous paragraph to locally finite graphs, while making sure that most of the construction (except the very end) is canonical. For this, we proceed in two steps. First, we use a result of \cite{CTTD} to find a canonical tree-decomposition of any $3$-connected graph $G$ that distinguishes all its tangles of order $4$. Using this result, we show that we can assume that the graph under consideration admits a unique tangle $\T$ of order $4$. 
We then follow the main arguments from \cite{Gro16} and  
show that $G$ has a canonical tree-decomposition of adhesion $3$ which is a star and whose torsos are all minors of $G$ and finite, except for the torso $H$ associated to the center of the star, which has the following property: there exists a matching $M\subseteq E(H)$ which is invariant under the action of the automorphism group of $G$ and such that the graph $H':=H/M$ obtained after the contraction of the edges of $M$ is quasi-transitive, locally finite, and quasi-$4$-connected. In particular, a result of Thomassen \cite{Thomassen92} then implies that $H'$ is  planar or has bounded treewidth.
We then prove that even if $H$ itself is not necessarily quasi-$4$-connected, it is still planar or has bounded treewidth, which is enough to conclude the proof of \cref{intro:main2}.
The final step to prove \cref{intro:main} consists in refining the tree-decomposition to make sure that torsos of bounded treewidth are replaced by torsos of finite size (moreover, this refinement has to be done in a canonical way).

\subsection{Related work}
Independently of our work, Carmesin and Kurkofka \cite{CK23} recently worked on decompositions of $3$-connected graphs with an approach that differs from Grohe's approach. They obtained a canonical decomposition into basic pieces consisting in quasi-$4$-connected graphs, wheels or thickenings of $K_{3,m}$ for $m\geq 0$. It is possible that their approach could also imply some of the applications we describe. However it is not clear to us whether this work directly implies the existence of canonical tree-decompositions with the properties described in \cref{intro:main,intro:main2}, as they consider \emph{mixed separations}, i.e.\ separations containing both vertices and edges, while we focus on vertex-separations. Vertex-separations yield tree-decompositions, while mixed separations do not yield tree-decompositions in a traditional sense. 


\subsection{Organization of the paper}

We start with some preliminary definitions and results about graphs
and groups in \cref{sec: prel}. \cref{sec: III} is
dedicated to the study of canonical tree-decompositions. It contains
the main technical contribution of the paper, a partial extension of
results of Grohe \cite{Gro16} to infinite graphs. \cref{sec: struct} contains the proof of \cref{intro:main,intro:main2}. \cref{sec: appli} is dedicated to the main applications of \cref{intro:main,intro:main2}.
We conclude with a number of open problems in \cref{sec:ccl}.

\section{Preliminaries}
\label{sec: prel}
For every $n\in \mathbb N$, we let $[n]:=\sg{1, 2, \ldots, n}$. 

\subsection{Graphs}
\label{sec: graphes}
In what follows, we will consider undirected graphs $G=(V(G),E(G))$ which are \emph{simple} (i.e.\ without loops and multi-edges), unless specifically stated otherwise. 
The set of vertices $V(G)$ will always be finite or infinite countable, and we will
always assume our graphs to be connected. Most of the time $G$ will be \emph{locally finite},
meaning that every vertex has finite degree (the only graphs that will not satisfy this additional requirement will be trees of the tree-decompositions). We equip a graph $G$ with
its shortest-path metric $d_G$. For any set of vertices $X\subseteq
V(G)$, the \emph{neighborhood} of $X$ in $G$ is denoted by 
\[N_G(X):=\sg{u\in V(G)\setminus X: \exists v\in X, uv \in E(G)}.\]
When the graph $G$ is clear from the context we will drop the
subscript and write $N(X)$ instead of $N_G(X)$. For each $v\in V(G)$, we set $N_G(v):=N_G(\sg{v})$.
For every graph $G$ and every subset of vertices $X\subseteq V(G)$, we
denote by $G[X]$ the \emph{subgraph of $G$ induced by $X$}, which is  the graph with vertex set $X$ whose edge set consists of all the pairs $uv$ such that $uv\in E(G)$. We let $G- X:=G[V(G)\setminus X]$.
We denote by $G\llbracket X \rrbracket$ the graph with vertex set
$X$ whose edge set consists of all the pairs $uv$ such that  $uv\in E(G)$ or there
exists a connected component $C$ of $G- X$ such that
$\sg{u,v}\subseteq N(C)$. The graphs $G\llbracket X \rrbracket$ are
called the \emph{torsos} of $G$.

\medskip

For each $n\in \mathbb N$, we let $K_n$ denote the complete graph with vertex set $[n]$. We let $\Kinf$ be the \emph{countable clique}, that is the infinite complete graph with vertex set $\mathbb N$. This graph is  sometimes also denoted by $K_{\aleph_0}$, which is less ambiguous, but we prefer to keep the notation $K_\infty$ as used by Thomassen in \cite{Thomassen92}, since we reuse a number of results proved in his paper.

\subsubsection*{Minors and models} Given two graphs $G,H$, we say that $H$ is a \emph{minor} of $G$ if it
can be obtained from $G$ after removing some vertices and edges, and
contracting edges. A \emph{model} of $H$ in $G$ is a family
$(V_{v})_{v\in V(H)}$ of pairwise disjoint vertex subsets of $G$ such
that each $V_v$ induces a connected subgraph of $G$, and for each
$uv\in E(H)$, there exists $u'\in V_u, v'\in V_v$ such that $u'v'\in
E(G)$. Note that $H$ is a minor of $G$ if and only if there is a model
of $H$ in $G$. When $V(H)\subseteq V(G)$, a model $(V_v)_{v\in V(H)}$
of $H$ in $G$ is said to be \emph{faithful} if for each $v\in V(H),
v\in V_v$. $H$ is a \emph{faithful minor} of $G$ if it admits a
faithful model in $G$. Consider for instance a graph $G$ with a subset of vertices $X\subset V(G)$ such that $G-X$ is connected and only two vertices of $X$ (call them $x$ and $y$) have a neighbor in $G-X$. Then the torso $G\llbracket X \rrbracket$ (as defined above), is a faithful minor of $G$ and it consists of the graph $G[X]$ with the addition of the edge $xy$ (if it is not already present in $G$).

\subsubsection*{Connectedness} For every $k\ge 0$, a graph $G$ is \emph{$k$-connected} if it has at least $k+1$ vertices and for every subset $S\subseteq V(G)$ of at most $k-1$ vertices, the graph $G- S$ is connected. We recall that
a graph is said to be \emph{quasi-$4$-connected} if it is $3$-connected 
and for every set $S\subseteq V(G)$ of size $3$ such that $G- S$ is not
connected, $G- S$ has exactly two connected components and one of them
consists of a single vertex.

\subsubsection*{Rays and ends} A \emph{ray} in a graph $G$ is an infinite simple one-way path $P=(v_1,v_2,\ldots)$. A \emph{subray} $P'$ of $P$ is a ray of the form $P'=(v_i,v_{i+1},\ldots)$ for some $i\geq 1$. We say that a ray \emph{lives} in a set $X\subseteq V(G)$ if one of its subrays is included in $X$.
We define an equivalence relation $\sim$ over the set of rays $\mathcal R(G)$ by letting $P\sim P'$ if and only if for every finite set of vertices $S\subseteq V(G)$, there is a component of $G-S$ that contains infinitely many vertices from both $P$ and $P'$. When $G$ is infinite, this is equivalent to saying that for any finite set $S\subseteq V(G)$, $P$ and $P'$ are living in the same component of $G- S$. The \emph{ends} of $G$ are the elements of $\mathcal R(G)/\sim$, the equivalence classes of rays under $\sim$.
For every $X \subseteq V(G)$, we say that an end $\omega$
\emph{lives in $X$} if one of its rays lives in $X$.

When there is a set $X$ of vertices
of $G$, two distinct components $C_1,C_2$ of $G-X$, and two distinct ends
$\omega_1,\omega_2$ of $G$ such that for each $i=1,2$, 
$\omega_i$ lives  in $C_i$, we say that $X$
\emph{separates} $\omega_1$ and $\omega_2$. A graph $G$ is
\emph{vertex-accessible} if there is an integer $k$ such that for any 
two distinct ends $\omega_1,\omega_2$ in $G$, there is a set of at most $k$
vertices that separates $\omega_1$ and $\omega_2$. The \emph{degree} of an end $\omega$ is the supremum
number $k\in \mathbb{N}\cup \sg{\infty}$ of pairwise disjoint rays
that belong to $\omega$. By a result of Halin \cite{HalinGrid}, this supremum is a maximum i.e. if an end $\omega$ has infinite degree, then there exists an infinite countable family of pairwise disjoint rays belonging to $\omega$.
An end is \emph{thin} if it has finite
degree, and \emph{thick} otherwise. It is an easy exercise to check that for every
end $\omega$ of finite degree $k$ and every end $\omega'\ne \omega$, there is a set of size at most $k$ that separates $\omega$ from $\omega'$.

\medskip

The interested reader is referred to Chapter 8 in \cite{diestel2017graph} for more background and important results in infinite graph theory.

\subsection{Groups and Cayley graphs}
An \emph{automorphism} of a graph $G$ is a graph isomorphism from $G$ to itself (i.e., a bijection from $V(G)$ to $V(G)$ that maps edges to edges and non-edges to non-edges). The set of automorphisms of $G$ has a natural group structure (as a subgroup of the symmetric group over $V(G)$); the group of automorphisms of $G$ is denoted by $\mathrm{Aut}(G)$.

For a graph $G$ and a group $\Gamma$, we will say that $\Gamma$ \emph{acts by automorphisms on $G$} (or simply that $\Gamma$ \emph{acts on $G$} when the context is clear) if every element of $\Gamma$ induces an automorphism $g$ of $G$, such that the induced application $\Gamma \to \mathrm{Aut}(G)$ is a group morphism. We will usually use the right multiplicative notation $x\cdot g$ instead of $g(x)$ for $g\in \Gamma$, $x\in V(G)$. For every $X\subseteq V(G), \Gamma'\subseteq \Gamma$ and $g\in \Gamma$, we let $X\cdot g:= g(X) =\sg{x\cdot g: x\in X}$ and $X\cdot \Gamma':=\bigcup_{g\in \Gamma'}X\cdot g$. 
We denote the  set of orbits of $V(G)$ under the action of $\Gamma$ by
$G/\Gamma$ ($\Gamma$ naturally induces an equivalence relation on
$V(G)$, relating elements in the same orbit of $\Gamma$). For every
subset $X\subseteq V(G)$ we let $\Stab_{\Gamma}(X):=\sg{g\in \Gamma:
  X\cdot g=X}$ denote the \emph{stabilizer} of $X$, which is always a
subgroup of $\Gamma$. For each $x\in X$, we let
$\Gamma_x:=\Stab_{\Gamma}(\sg{x})$.

\subsubsection*{Quasi-transitive graphs} The action of
a group $\Gamma$ on a graph $G$ is said to be \emph{vertex-transitive} (or simply \emph{transitive})
when there is only one orbit in $G/\Gamma$, i.e.\ when for every two
vertices $u,v\in V(G)$ there exists an element $g\in \Gamma$ such that
$u\cdot g=v$. The action of $\Gamma$ on $G$ is said to be
\emph{quasi-transitive} if there is only a finite number of  orbits in
$G/\Gamma$. We say that $G$ is \emph{transitive} (resp.\
\emph{quasi-transitive}) if it admits a transitive (resp.\
quasi-transitive) group action.

\smallskip

It was proved, first for finitely generated groups and then in the more general graph-theoretic context, that the number of ends of a quasi-transitive graph is either $0,1,2$ or $\infty$ \cite{Freudenthal44,Hopf43,Diestel93}. A graph with a single end is said to be \emph{one-ended}.

\subsubsection*{Finitely presented groups} A \emph{group presentation} is a pair $\Pres$ where $S$ is a set of letters called \emph{generators} and $R$ a set of finite words over the alphabet $S$ called \emph{relators}. We will always assume that $S$ is finite and closed under taking inverse, i.e.\ that every generator $a\in S$ comes with an associated \emph{inverse} $a^{-1}\in S$ (which may be equal to $a$) such that $(a^{-1})^{-1}=a$, and we assume that $aa^{-1}\in R$.
We say that $\Pres$ is \emph{finite} when both $S$ and $R$ are finite. The group associated to $\Pres$ is the group $F(S)/N(R)$, where $F(S)$ is the free group over $S$ and $N(R)$ is the normal closure of $R$ in $F(S)$, i.e.\ the set of elements of $F(S)$ of the form $(w_1\cdot r_1\cdot w_1^{-1})\cdots (w_\ell\cdot r_\ell\cdot w_\ell^{-1})$ for any $\ell\in \mathbb{N}$, $r_1,\ldots, r_\ell\in R$ and $w_1,\ldots, w_\ell \in F(S)$. For simplicity, we will also denote this group with $\Pres$. Note that by definition of $\Pres$, since $aa^{-1}\in R$ for every element $a\in S$,  the (formal) inverse $a^{-1}$ of $a$ in $S$ is indeed the inverse of  $a$ in the group $\Pres$.

\subsubsection*{Cayley graphs}  The \emph{Cayley graph} of a finitely
generated group $\Gamma$ with respect to the finite set of generators $S$ is the edge-labeled graph 
$\Cay(\Gamma,S)$ whose vertex set is the set of elements of $\Gamma$ and where for every two $g,h\in \Gamma$ we add an arc $(g,h)$ labeled with $a\in S$ when $h=a\cdot g$. Note that $\Cay(\Gamma,S)$ is always locally finite (as we assumed $S$ to be finite), connected and that the group 
$\Gamma$ acts transitively by right multiplication on
$\Cay(\Gamma,S)$. As mentioned above we will always assume that the set $S$ of
generators is symmetric, so whenever there is an arc $(u,v)$ labeled
with $a\in S$, the graph also contained the arc $(v,u)$ labeled
$a^{-1}$. It this case we can consider the non-labeled version of
$\Cay(\Gamma,S)$ as an undirected simple graph, with a single edge
$uv$ instead of each pair of arcs $(u,v)$ and $(v,u)$.

\smallskip

We say that a finitely generated group $\Gamma$ is \emph{planar} (resp.\ \emph{minor-excluded}) if it admits a finite generating set $S$ such that $\Cay(\Gamma, S)$ is planar (resp.\ does not contain every finite graph as a minor). Similarly, we say that $\Gamma$ is $K_\infty$-minor-free if it admits a finite generating set $S$ such that $\Cay(\Gamma, S)$ is $K_\infty$-minor-free.

\subsubsection*{Accessibility}
In order to give a precise definition of accessibility in groups
(without using to Stallings theorem and the notions of free product
with amalgamation and HNN-extension explicitly as we did in the introduction), we
first need to define the notion of a graph of groups. The reader is
referred to \cite[Chapter 1]{Serre} for more details on graphs of
groups and Bass-Serre theory.

\medskip

A \emph{graph of groups} consists of  a pair $(G, \mathcal G)$ such that $G=(V,E)$ is a graph (possibly having loops and
multi-edges), and $\mathcal G$ is a family of 
\emph{vertex-groups} $\Gamma_v$ for
each $v\in V$, \emph{edge-groups} $\Gamma_{uv},
\Gamma_{vu}$ for each edge $uv\in E$, and of group isomorphisms
$\phi_{u,v}:\Gamma_{uv}\to \Gamma_{vu}$ for each $uv\in E(G)$, such that $\Gamma_{uv}$ and $\Gamma_{vu}$ are respectively subgroups of
$\Gamma_u$ and $\Gamma_v$ for each $uv\in E$. 

Let $T$ be a spanning tree of $G$ and $\langle S_v| R_v \rangle$ be a presentation of $\Gamma_v$ for each $v\in V$. The \emph{fundamental group} $\Gamma:=\pi_1(G, \mathcal{G})$ of the graph of groups $(G, \mathcal{G})$ is defined as the group having as generators the set:
$$S:=\left(\sqcup_{v\in V}S_v\right) \sqcup \left(\sqcup_{e\in E\setminus E(T)} \sg{t_e}\right)$$
and as relations:
\begin{itemize}
    \item the relations of each $R_v$;
    \item for every edge $uv\in E(T)$ and every $g\in \Gamma_{uv}\subseteq \Gamma_u$, the relation $\phi_{uv}(g)=g$;
    \item for every edge $e=uv\in E\setminus E(T)$ and every $g\in \Gamma_{uv}$, the relation $t_e\phi_{uv}(g)t_{e}^{-1}=g$.
\end{itemize}
It can be shown that for a given graph of groups $(G, \mathcal{G})$, the definition of its fundamental group does not depend of the choice of the spanning tree $T$ (see for example \cite[Section I.5.1]{Serre}). 

\begin{remark}
At the beginning of this section we have defined  $\Gamma_v$ as the stabilizer of $\{v\}$ by the action of $\Gamma$ on a graph $G$ and the reader might be worried about a possible confusion with the notation $\Gamma_v$ of vertex-groups above. On the one hand the vertex-group $\Gamma_v$ has a close connection with the stabilizer of $\{v\}$ in this context, so the objects are not completely unrelated, and on the other hand vertex-groups will only be used at the very end (in \cref{sec: Domino}, where stabilizers will not be used at all), so hopefully there should not be any risk of confusion.
\end{remark}

A group $\Gamma$ is said to be  \emph{accessible} if it is the fundamental group of a finite graph of groups $G$ with finite vertex set $V(G)$ such that:
\begin{itemize}
    \item the vertex-groups have at most one end, and
    \item the edge-groups are finite.
\end{itemize}

By Bass-Serre theory \cite{Serre,DD}, accessible groups are exactly those groups acting on trees without edge-inversion, such that the vertex-stabilizers have at most one end and the edge-stabilizers of the action are finite.

\medskip

As mentioned in the introduction, Thomassen and Woess \cite{TW}
obtained the following connection between accessibility in groups and
vertex-accessibility in graphs.

\begin{theorem}[\cite{TW}]
\label{thm: TW}
A finitely generated group is accessible if and only if it 
admits a Cayley graph which is vertex-accessible. 
\end{theorem}

It was open for a long time whether there exist finitely generated
groups which are not accessible, and Dunwoody answered this question
negatively in \cite{Dunwoody93}. On the other hand he also proved the following result. 

\begin{theorem}[\cite{Dunwoody1985}]
\label{thm: Dunwoody}
Every finitely presented group is accessible.
\end{theorem}

In particular, planar groups form a proper subclass of accessible groups.
\begin{theorem}[\cite{Droms}]
 \label{thm: Droms}
 Every finitely generated planar group is finitely presented, and thus accessible.
\end{theorem}

\subsubsection*{Virtually free groups} A group is \emph{virtually free} if it contains a
finitely generated free group as a subgroup of finite index. For graphs of
groups, this property can be related to the structure of the vertex groups as
follows.

\begin{theorem}[\cite{KPS}]
\label{thm: multi-ended}
A finitely generated group is virtually free if and only if it is the
fundamental group of a finite graph of groups in which all
vertex-groups are finite.
\end{theorem}

\section{Tree-decompositions and tangles}
\label{sec: III}
\subsection{Separations and canonical tree-decompositions}
As we will be heavily relying on results of Grohe \cite{Gro16}, we use
his notation for all objects related to separations and tangles. A \emph{separation} in a graph $G=(V,E)$ is a triple $\Sep$
such that $Y,S,Z$ are pairwise disjoint, $V=Y\cup S\cup Z$ and there is no edge
between vertices of $Y$ and $Z$. A separation $\Sep$ is \emph{proper}
if $Y$ and $Z$ are nonempty. In this case, $S$ is a \emph{separator} of
$G$. 

The separation $\Sep$ is said to be
\emph{tight} if there are some components $C_Y,C_Z$ respectively of $G[Y],G[Z]$
such that $N_G(C_Y)=N_G(C_Z)=S$. The \emph{order} of a
separation $\Sep$ is $|S|$ and the \emph{order} of a family $\mathcal N$ of separations
is the supremum of the orders of its separations. In what follows,
we will always consider sets of separations of finite order. We will denote
$\separ_{k}(G)$ (respectively $\separ_{<k}(G)$) the set of all separations of
$G$ of order $k$ (respectively less than $k$). 

\medskip

The following lemma was originally stated in \cite{TW} for transitive graphs, but the same proof immediately implies that the result also holds for quasi-transitive graphs.
\begin{lemma}[Corollary 4.3 in \cite{TW}]
\label{lem: TWcut}
 Let $G$ be a locally finite graph. Then for every $v\in V(G)$ and $k\geq 1$, there is 
 only a finite number of tight separations $\Sep$ of order $k$ in $G$ such that $v\in S$.
 Moreover, for any group $\Gamma$  acting
 quasi-transitively on $G$ and  any $k\geq 1$, there is only a
 finite number of $\Gamma$-orbits of tight separations of order at
 most $k$ in $G$. 
\end{lemma}

\subsubsection*{Canonical tree-decompositions}

A \emph{tree-decomposition} of a graph $G$ is a pair $(T,\mathcal V)$ where $T$ is a tree and $\mathcal V=(V_t)_{t\in V(T)}$ is a family of subsets $V_t$ of $V(G)$ such that:
\begin{itemize}
 \item $V(G)=\bigcup_{t\in V(T)}V_t$;
 \item for every nodes $t,t',t''$ such that $t'$ is on the unique path of $T$ from $t$ to $t''$, $V_t\cap V_{t''}\subseteq V_{t'}$; 
 \item every edge $e\in E(G)$ is contained in an induced subgraph $G[V_t]$ for some $t\in V(T)$. 
\end{itemize}
Note that in our definition of tree-decomposition, we allow $T$ to have vertices of infinite degree. The sets $V_t$ for every $t\in V(T)$ are called the \emph{bags} of $(T,\mathcal V)$, and the induced subgraphs $G[V_t]$ the \emph{parts} of $(T,\mathcal V)$. The \emph{width} of $(T,\mathcal V)$ is the supremum of $|V_t|-1$, for $t\in V(T)$. Note that the width of a tree-decomposition can be infinite.
The sets $V_t\cap V_{t'}$ for every $tt'\in E(T)$ are called the \emph{adhesion sets} of $(T,\mathcal V)$ and the \emph{adhesion} of $(T,\mathcal V)$ is the supremum of the sizes of its adhesion sets (possibly infinite). We also let $V_\infty(T)\subseteq V(T)$ denote the set of nodes $t\in V(T)$ such that $V_t$ is infinite.

\medskip

For a group $\Gamma$ acting (by automorphisms) on a graph $G$, we say that a
tree-decomposition $(T,\mathcal V)$ of $G$ is \emph{canonical with respect to
  $\Gamma$}, or simply \emph{$\Gamma$-canonical}, if $\Gamma$ induces a group action on $T$ such that for every $\gamma
\in \Gamma$ and $t\in V(T)$, $V_t\cdot \gamma= V_{t\cdot \gamma}$. By definition of a group action on a graph, $t\mapsto t\cdot \gamma$ is an automorphism of $T$ for any $\gamma\in \Gamma$.
In particular, for every $\gamma\in \Gamma$, note that 
$\gamma$ sends bags of $(T,\mathcal V)$ to bags, and
adhesion sets to adhesion sets. When $(T,\mathcal V)$ is
$\Aut(G)$-canonical, we simply say that it is \emph{canonical}. 

\medskip

\begin{remark}
\label{rem: stab}
 If $(T,\mathcal V)$ is a $\Gamma$-canonical tree-decomposition of a graph $G$, then $\Gamma$ acts both on $G$ and $T$, so there are two different notions of a stabilizer of a node $t\in V(T)$: $\Gamma_t=\Stab_{\Gamma}(t)$ (where we consider the action of $\Gamma$ on $T$), and $\Stab_{\Gamma}(V_t)$ (where we consider the action of $\Gamma$ on $G$). Observe that for any $t\in V(T)$ we have $\Gamma_t\subseteq \Stab_{\Gamma}(V_t)$. The reverse inclusion does not hold in general (when there are adjacent nodes $s,t\in V(T)$ with $V_s=V_t$,  automorphisms of $T$ exchanging $s$ and $t$ stabilize $V_s=V_t$ without stabilizing $s$ or $t$). However, if $t\in V(T)$ is such that $\sg{t'\in V(T): V_{t'}=V_t}=\sg{t}$, then $\Gamma_t= \Stab_{\Gamma}(V_t)$. In particular, if $(T, \mathcal V)$ has finite adhesion, then every bag $V_t$ with $t\in V_\infty(T)$ appears only once in the decomposition, and thus for each such node $t\in V_\infty(T)$ we have $\Gamma_t= \Stab_{\Gamma}(V_t)$. The property that the two notions of stabilizers coincide for infinite bags when the canonical tree-decomposition has finite adhesion will be used repeatedly in the remainder of the paper.
\end{remark}

\subsubsection*{Edge-separations and torsos}

Consider a tree-decomposition $(T,\mathcal V)$ of a graph $G$, with
$\mathcal{V}=(V_t)_{t\in V(T)}$. Let $A$ be an orientation of the
edges of $E(T)$, i.e.\ a choice of either $(t_1,t_2)$ or $(t_2,t_1)$
for every edge $t_1t_2$ of $T$. For an arbitrary pair $(t_1,t_2)\in
A$, and for each $i\in \{1,2\}$, let $T_i$ denote the component of
$T-\{t_1t_2\}$ containing $t_i$. Then the \emph{edge-separation} of
$G$ associated to $(t_1,t_2)$ is $(Y_1,S,Y_2)$ with $S:= V_{t_1}\cap
V_{t_2}$ and $Y_i:=
\bigcup_{s\in V(T_i)} V_s\setminus S$ for $i\in \sg{1,2}$.

\medskip

Given a separation
$\Sep$ and an automorphism $\gamma$ of a graph $G$, let $\Sep\cdot \gamma := (Y\cdot \gamma, S\cdot \gamma, Z\cdot
\gamma)$. 
If $\Gamma\subseteq
\mathrm{Aut}(G)$ and $\mathcal N$ is a family of separations of $G$, we say that
$\mathcal N$ is \emph{$\Gamma$-invariant} if for every $\Sep\in
\mathcal N$ and $
\gamma\in \Gamma$, we have $\Sep\cdot \gamma\in \mathcal N$.
Note that if $(T,\mathcal V)$ is $\Gamma$-canonical, then the associated set of edge-separations is $\Gamma$-invariant.

\medskip

The \emph{torsos} of $(T,\mathcal V)$ are the graphs  with vertex set $V_t$ and edge set 
$E(G[V_t])$ together with the edges $xy$ such that $x$ and $y$ belong
to a common adhesion set of $(T,\mathcal V)$. Note that this
definition coincides with the general definition of torso $G\llbracket V_t\rrbracket$ we gave in
\cref{sec: graphes} when the edge-separations of $(T, \mathcal
V)$ are tight. To prevent any ambiguity between these definitions, we
will always ensure that the tree-decompositions we work with have this
property (i.e., all the associated edge-separations are tight).

\smallskip

\begin{remark}
\label{rem: nbaretes}
 If $(T,\mathcal V)$ is a $\Gamma$-canonical tree-decomposition of a locally finite graph $G$ whose edge-separations are tight, then by \cref{lem: TWcut} the action of $\Gamma$ on $E(T)$ must induce a finite number of orbits. In particular, $\Gamma$ must also act quasi-transitively on $V(T)$.   
\end{remark}

The \emph{treewidth} of a graph $G$ is the infimum of the width of $(T,\mathcal V)$, 
among all tree-decompositions $(T,\mathcal V)$ of $G$.
Note that adding to a tree-decomposition of bounded width the
restriction that it must be canonical can be very costly in the finite case: while it is well known
that every cycle graph $C_n$ on $n$ vertices has treewidth $2$, the
example below shows that in any canonical
tree-decomposition of $C_n$, some bag contains all the nodes of $C_n$.

\begin{example}
\label{ex: cycle}
Let $C_n$ be the cycle graph on $n$ elements. Note that the additive
group $\mathbb Z_n$ acts transitively by rotation on $C_n$. We let $a$
be a generator of $\mathbb{Z}_n$ of order $n$. Let  $(T,(V_t)_{t\in V(T)})$ be a
$\mathbb Z_n$-canonical tree-decomposition of $C_n$. Without loss of generality we may assume that $T$ is finite, by contracting every edge $tt'\in E(T)$ such that $V_t=V_{t'}$. We may also assume that no edge $tt'$ of $T$ is \emph{inverted} by $a$, i.e.\ such that $(t,t')\cdot a=(t',t)$, as if it was the case we could subdivide the edge $tt'$ (i.e.\ add a new vertex $t^*$ between $t$ and $t'$) and let $V_{t^*}:= V_{t}\cap V_{t'}$. If we let $T'$ be the tree of the tree-decomposition obtained after performing such a subdivision, note that the obtained tree-decomposition is still $\mathbb{Z}_n$-canonical as $a$ induces an automorphism of $T'$ that stabilizes the vertex $t^*$ and acts on $V(T)$ the same way that it did before the subdivision.
After this operation none of the edges $tt^*$, $t^*t'$ is inverted by $a$. 
It is an easy exercise to prove that if no edge of $T$ is inverted by $a$,  there exists a vertex $t\in V(T)$ stabilized by $a$, and hence by all the elements of $\mathbb Z_{n}$. Then as $\mathbb Z_n$ acts transitively on $G$, we must have $V_t=V(C_n)$ for such a $t\in V(T)$.  
\end{example}

It turns out that in the quasi-transitive locally-finite case, if a graph has bounded treewidth then adding the
restriction that the tree-decomposition must be canonical is
fairly inexpensive if we only care about having finite bags.

\begin{theorem}[Theorem 7.5 in \cite{HamannStallings22}, \cite{MS83}, \cite{Woess89}, \cite{TW}]
\label{thm: tw-ends}
 Let $G$ be a connected quasi-transitive locally finite graph. Then the following are equivalent:
 \begin{itemize}
     \item $G$ has finite treewidth;
     \item all the ends of $G$ are thin;
     \item there exists $k\geq 1$ such that every end of $G$ has degree at most $k$;
     \item there exists a canonical tree-decomposition of $G$ with tight edge-separations and finite width.
 \end{itemize}
\end{theorem}

Note that the final item above is not stated explicitly in Theorem 7.5 in \cite{HamannStallings22}, but it can be easily deduced from this result (see also Theorem 2.4 in \cite{Ham25}, which uses the result of \cite{HamannStallings22} together with our Lemmas 3.1 and 3.13).

\subsubsection*{Separations of order at most 3}
If $G$ is not connected, then the tree-decomposition $(T,\mathcal V)$ where $T$ is a star whose central bag is empty and where we put a bag for each connected component of $G$ can easily be seen to be a canonical tree-decomposition with adhesion $0$, as every automorphism of $G$ acts on $T$ by permuting some branches.
If we start from a connected graph $G$, it is well-known that the block cut-tree of $G$ is a canonical tree-decomposition $(T,(V_t)_{t\in V(T)})$ of $G$ whose adhesion sets have size $1$ and such that for each $t\in V(T)$, $G[V_t]=G\llbracket V_t \rrbracket$ has either size at most $2$ or is $2$-connected.  A similar result holds for separations of order $2$ (this was proved by Tutte \cite{Tutte} in the finite case, and generalized to infinite graphs in \cite{DSS98}).

\begin{theorem}[\cite{DSS98}]
 \label{thm: Tutte}
 Every  locally finite graph $G$ has  a canonical tree-decompo\-sition  of adhesion at most $2$, whose torsos are minors of $G$ and are complete graphs of order at most $2$, 
 cycles, or $3$-connected graphs.
\end{theorem}

For separations of order $3$, a similar result  was obtained by Grohe for finite graphs \cite{Gro16}. 

\begin{theorem}[\cite{Gro16}]
\label{thm: Grofini}
 Every finite graph $G$ has  a tree-decomposition  of adhesion at most $3$ whose torsos are minors of $G$ and are complete graphs on at most 4 vertices or  quasi-$4$-connected graphs.
\end{theorem}

Our main technical contribution will be to extend \cref{thm: Grofini} to locally finite graphs, while making sure that most of the construction (except the very end) is canonical. More precisely, we reproduce in \cref{sec: single,sec: all} the main steps of the work of \cite{Gro16} and give the additional arguments to extend them to locally finite graphs. A consequence is that \cref{thm: Grofini} extends to locally finite graphs. However in our case, the main difficulty we face is that the decomposition of Grohe is not canonical (although some parts of the construction are canonical, which will be crucial for our purposes).
To give a rough idea of the bulk of the problem it is helpful to consider \cref{ex: gro} below, which was introduced in \cite{Gro16} in the finite case.

\tikzexternaldisable
  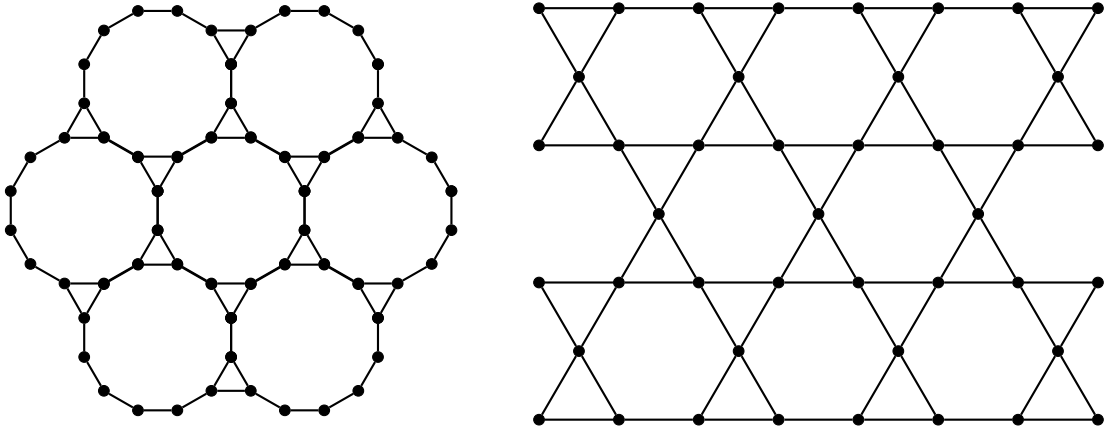
\begin{figure}[h]
    \centering
    \begin{tikzpicture}[scale=1]
    
    \begin{scope}[yshift=6cm]
      \tikzstyle{every node}=[draw, circle,fill=black,minimum size=4pt, inner sep=0pt]
    \pgfmathsetmacro{\alph}{(360/12)}
    \pgfmathsetmacro{\c}{2*sin(\alph/2)}
    \pgfmathsetmacro{\hu}{cos(\alph/2)}
    \pgfmathsetmacro{\hd}{\hu*\c}
    \pgfmathsetmacro{\htr}{\c/2}
    \pgfmathsetmacro{\h}{\hu +\hd +\htr -0.0015cm}
    
     
    \foreach \i/\xs/\ys in {1/-1/1, 2/1/1, 3/-2/0, 4/0/0, 5/2/0, 6/-1/-1, 7/1/-1} {
        \pgfmathsetmacro{\xx}{\xs*\hu}
        \pgfmathsetmacro{\yy}{\ys*\h}
        \begin{scope}[xshift=\xx cm, yshift=\yy cm]
            \foreach \j in {0,...,12} {
                \pgfmathsetmacro{\angl}{\alph/2 + \j*\alph}
                \node (x\i\j) at (\angl:1) [] {};
            }
            \draw[thick] (x\i0) -- (x\i1) -- (x\i2) -- (x\i3) -- (x\i4) -- (x\i5) -- (x\i6) -- (x\i7) -- (x\i8) -- (x\i9) -- (x\i10) -- (x\i11) -- (x\i12);
        \end{scope}
    }
    \foreach \i/\j/\ii/\jj in {1/6/3/3, 1/1/2/4, 2/11/5/2, 3/8/6/5, 6/10/7/7, 5/9/7/0}{
        \draw[thick] (x\i\j) -- (x\ii\jj); 
    }
    \end{scope}
    \begin{scope}[xshift=3cm, yshift=0.5cm, scale = 0.35]
     \pgfmathsetmacro{\r}{3}
    \pgfmathsetmacro{\tri}{1}  
    \pgfmathsetmacro{\long}{8}
    \pgfmathsetmacro{\haut}{1}

    \pgfmathsetmacro{\hh}{int(\haut+1)}
    \pgfmathsetmacro{\lll}{int(\long-1)}
    \pgfmathsetmacro{\ld}{int((\long-1)/2)} 
    \pgfmathsetmacro{\lld}{int((\long)/2)} 
    \tikzstyle{every node}=[draw, circle,fill=black,minimum size=4pt, inner sep=0pt]

    \foreach \j [evaluate=\j] in {1,...,\hh} 
    {\foreach \i [evaluate=\i] in {1,...,\long}
      {\pgfmathsetmacro{\y}{int(4*\j-1)}
        \pgfmathsetmacro{\x}{int(\i)}
        \node[draw=black,circle] (v\x_\y) at ($({\x*\r},{\y*\r*sqrt(3)/2})$) {};
       \ifthenelse{\i > 1}{
        \pgfmathsetmacro{\xm}{int(\x-1)}
        \draw[thick] (v\xm_\y) -- (v\x_\y) {};}{}
      }
    }
    \foreach \j [evaluate=\j] in {1,...,\hh}
    {\foreach \i [evaluate=\i] in {1,...,\long}
      {\pgfmathsetmacro{\y}{int(4*\j+1)}
       \pgfmathsetmacro{\x}{int(\i)}
        \node[draw=black,circle] (v\x_\y) at ($({\x*\r},{\y*\r*sqrt(3)/2})$) {};
       \ifthenelse{\i > 1}{
        \pgfmathsetmacro{\xm}{int(\x-1)}
        \draw[thick] (v\xm_\y) -- (v\x_\y) {};}{} 
      }
    }
    \foreach \j [evaluate=\j] in {1,...,\hh}
        {\foreach \i [evaluate=\i] in {1,...,\lld}
        {\pgfmathsetmacro{\y}{int(4*\j)}
       \pgfmathsetmacro{\x}{int(2*\i)}
        \node[draw=black,circle] (v\x_\y) at ($({(\x-0.5)*\r},{\y*\r*sqrt(3)/2})$) {};
       \pgfmathsetmacro{\ym}{int(\y-1)}
       \pgfmathsetmacro{\yp}{int(\y+1)}
       \pgfmathsetmacro{\xp}{int(2*\i)}
       \pgfmathsetmacro{\xm}{int(2*\i-1)}       
       \draw[thick] (v\xp_\ym) -- (v\x_\y) {};
       \draw[thick] (v\xp_\yp) -- (v\x_\y) {};
       \draw[thick] (v\xm_\ym) -- (v\x_\y) {};
       \draw[thick] (v\xm_\yp) -- (v\x_\y) {};       
        }
    }
    \foreach \j [evaluate=\j] in {1,...,\haut}
    {\foreach \i [evaluate=\i] in {1,...,\ld}
            {\pgfmathsetmacro{\y}{int(4*\j+2)}
        \pgfmathsetmacro{\x}{int(2*\i)}
        \node[draw=black,circle] (v\x_\y) at ($({(\x+0.5)*\r},  {\y*\r*sqrt(3)/2})$) {};
        \pgfmathsetmacro{\ym}{int(\y-1)}
        \pgfmathsetmacro{\yp}{int(\y+1)}
        \pgfmathsetmacro{\xp}{int(2*\i+1)}
        \pgfmathsetmacro{\xm}{int(2*\i)}       
        \draw[thick] (v\xm_\ym) -- (v\x_\y) {};
        \draw[thick] (v\xp_\yp) -- (v\x_\y) {};
        \draw[thick] (v\xp_\ym) -- (v\x_\y) {};
        \draw[thick] (v\xm_\yp) -- (v\x_\y) {};       
            }
    }
    
    \end{scope}

\end{tikzpicture}
\caption{Left: a finite section of the $3$-connected infinite graph obtained by replacing in the infinite hexagonal planar grid each vertex by a triangle with three vertices of degree $3$. 
Right: a finite section of the quasi-$4$-connected torso $G\torso{V_{z_0}}$ of $(T,\mathcal V)$. Note that it does not depend of the choice of $V_{z_0}$.}
\label{fig: Gro2}
\end{figure}

\begin{example}
\label{ex: gro}
Consider  the $3$-connected infinite planar graph $H$ obtained from the infinite hexagonal planar grid by replacing each vertex by a triangle with three vertices of degree $3$ (see \cref{fig: Gro2} (left) for a finite part of this graph). We let $M$ be the set of edges connecting pairs of triangles, or equivalently the set of edges that do not belong to any triangle (note that $M$ is a perfect matching). The tree-decomposition $(T,\mathcal V)$ of $H$ obtained by extending the ideas in \cite{Gro16} to the infinite case has an infinite bag $V_{z_0}$ obtained by
selecting one endpoint of each edge of $M$ (which is equivalent to fixing an orientation of each of these edges). The tree $T$ is a subdivision of a star with center $z_0$, and its other bags are finite. While there are many different choices for $V_{z_0}$, none of them gives a canonical tree-decomposition. 
Indeed one can check more generally that no tree-decomposition of $H$ satisfying the properties of Theorem \ref{thm: Grofini} can be canonical. To see this, assume for the sake of contradiction that such a decomposition $(T,\mathcal V)$ exists. Then one of its edge-separations should be proper of order $3$. Note that the only such separations separate a subgraph of a triangle from the rest of the graph. Let $\Sep$ be such a separation, such that $Z$ is finite. Then there exists an edge $e$ from $M$ with one endpoint in $Z$ and the other in $S$. Note that there exists an automorphism $\gamma\in \Aut(H)$ exchanging the two endpoints of $e$. In particular, as $(T, \mathcal V)$ is canonical, both $\Sep\cdot \gamma$ and $\Sep$ must be edge-separations of $(T,\mathcal V)$, which can be seen to be impossible.

\smallskip

This example illustrates the fact that in general it is impossible to obtain a \emph{canonical} tree-decomposition having exactly the properties described in \cref{thm: Grofini}. However note that here, if we want a canonical tree-decomposition whose torsos are either planar or finite, it is sufficient to take the trivial tree-decomposition with a single bag containing all the vertices. This is exactly what our proof will do when applied to this graph. More precisely, we note that on this example, the set $M$ of edges is invariant under the action of $\Aut(H)$. Based on this observation, our solution to obtain a canonical tree-decomposition will be to start with the same decomposition as that of \cite{Gro16}, but to keep the two endpoints of each edge of $M$ instead of choosing only one of its endpoints as above.
\end{example}

\subsubsection*{Combining canonical tree-decompositions}
Let $(T,\mathcal V)$ and $(T',\mathcal V')$ be  tree-decomposi\-tions of two graphs $G,G'$, respectively. We say that $(T,\mathcal V)$ and $(T',\mathcal V')$ are \emph{isomorphic} if there exists  an isomorphism $\varphi$ from $G$ to $G'$, and an isomorphism $\psi$ from $T$ to $T'$  such that for each $t\in V(T)$, we have: $V'_{\psi(t)}=\varphi(V_t)$.

Let $G$ be a graph and let $\Gamma$ be a group acting on $G$. Let $(T,\mathcal V)$, with $\mathcal{V}=(V_t)_{t\in V(T)}$, be a $\Gamma$-canonical tree-decomposition of $G$, and recall that for any $t\in V(T)$, $\Gamma_t=\Stab_\Gamma(t)$ denotes the stabilizer of the node $t$ in the action of $\Gamma$ on the tree $T$. For each $t\in V(T)$, let $(T_t,\mathcal V_t)$ be a $\Gamma_t$-canonical tree-decomposition of $G\torso{V_t}$.  Our goal will be to refine $(T,\mathcal V)$ by combining it with the tree-decompositions $(T_t,\mathcal V_t)_{t\in V(T)}$. If we want the resulting refined tree-decomposition of $G$ to be $\Gamma$-canonical, we need to impose a condition on the tree-decompositions $(T_t,\mathcal V_t)_{t\in V(T)}$ (namely that they are consistent with the action of $\Gamma$ on $G$). This is captured by the following definition.
For $g\in \Gamma$ and $t\in V(T)$, we define $(T_t, \mathcal V_t)\cdot g:=(T_t, \mathcal V_t \cdot g)$ as the  tree-decomposition of $G\torso{V_t\cdot g}=G\torso{V_{t\cdot g}}$ with underlying tree $T_t$ and bags $\mathcal V_t\cdot g:=(V_s \cdot g)_{s\in V(T_t)}$. Observe that  $(T_t, \mathcal V_t)\cdot g$ is $\Gamma_{t\cdot g}$-canonical. We say that the construction $t\mapsto (T_t, \mathcal V_t)$ is \emph{$\Gamma$-canonical} if for each $g\in \Gamma$ and $t\in V(T)$, the tree-decompositions $(T_t, \mathcal V_t)\cdot g$ and $(T_{t\cdot g}, \mathcal V_{t\cdot g})$ are isomorphic. We emphasize here that the first tree-decomposition is indexed by $T_t$, while the second is indexed by $T_{t\cdot g}$.


\medskip

The \emph{trivial tree-decomposition} of a graph $G$ consists of a tree $T$ with a single node, whose bag is $V(G)$. Note that the trivial tree-decomposition is canonical.

\begin{lemma}
 \label{lem: canonical}
 Assume that $G$ is locally finite, $\Gamma$ is a group acting on $G$, and
 $(T,\mathcal V)$ is a $\Gamma$-canonical tree-decomposition of $G$ with
 finitely bounded adhesion. Let $\sg{t_i: i\in I_{\infty}}$ be a set of representatives of the $\Gamma$-orbits of $V_\infty(T)$, indexed by some set $I_{\infty}$. 
 Assume that for every $i\in I_{\infty}$ there exists a $\Gamma_{t_i}$-canonical tree-decomposition $(T_{t_i}, \mathcal V_{t_i})$ of the torso $G\torso{V_{t_i}}$ of finitely bounded adhesion. Then we can find some family $(T_t, \mathcal V_t)_{t\in V(T)}$ extending the family $(T_{t_i}, \mathcal V_{t_i})_{i\in I_{\infty}}$  such that each $(T_t, \mathcal V_t)$ is a $\Gamma_t$-canonical tree-decomposition of $G\torso{V_t}$, the construction $t\mapsto (T_t,\mathcal V_t)$ is $\Gamma$-canonical, and for each $t\in V(T)\setminus V_\infty(T)$, $(T_t, \mathcal V_t)$ is the trivial tree-decomposition of $G\torso{V_t}$.
\end{lemma}
\begin{proof}
First we check that for each $g\in \Gamma$ and every $t\in V_\infty(T)$, $g^{-1}\cdot
\Gamma_t \cdot g=\Gamma_{t\cdot g}$. For this we claim that we only need to prove the inclusion $g^{-1}\cdot
\Gamma_t \cdot g\subseteq\Gamma_{t\cdot g}$, as the converse then follows from replacing $(g,t)$ by $(g^{-1},t\cdot g)$. Let $y\in V_{t\cdot g}= V_t\cdot g$ (where the equality follows from the assumption that $(T,\mathcal V)$ is $\Gamma$-canonical) and let $h\in \Gamma_t$. Then $y=x\cdot g$ for some $x\in V_t$ and we have:
$$y\cdot (g^{-1}\cdot h \cdot g) = x\cdot g \cdot (g^{-1}\cdot h \cdot g) = (x\cdot h)\cdot g.$$
Since $h\in \Gamma_t$,  we have $x\cdot h\in V_t$, we thus get that $y\cdot (g^{-1}\cdot h \cdot g)\in V_t \cdot g = V_{t\cdot g}$, so we just proved that every element of $g^{-1} \cdot \Gamma_t \cdot g$ stabilizes $V_{t\cdot g}$. Since $t\in V_\infty(T)$, $V_t$ and $V_{t\cdot g}$ are infinite and it follows from Remark \ref{rem: stab} that $\Stab_{\Gamma}(V_{t\cdot g})=\Gamma_{t\cdot g}$. This implies that every element of $g^{-1} \cdot \Gamma_t \cdot g$ lies in $\Gamma_{t\cdot g}$, and thus $g^{-1}\cdot
\Gamma_t \cdot g\subseteq\Gamma_{t\cdot g}$, as desired. 

We complete $I_{\infty}$ into a set $I$ of representatives of the $\Gamma$-orbits of $V(T)$, and for each $i\in I\setminus I_{\infty}$ we let $(T_{t_i}, \mathcal V_{t_i})$ denote the trivial tree-decomposition of $G\torso{V_{t_i}}$.
For each $t\in V(T)$, we let $g\in \Gamma$ and $i\in I$ be
such that $t=t_i\cdot g$ and let $(T_{t}, \mathcal V_t):=(T_{t_i},\mathcal V_{t_i})\cdot g$. We
check that $(T_t,\mathcal V_t)$ is well-defined: for any two $g,g'\in \Gamma$ such that $t_i\cdot
g=t_i\cdot g' = t$, we have $g'\cdot g^{-1}\in \Gamma_{t_i}$. As $(T_{t_i}, \mathcal V_{t_i})$ is $\Gamma_{t_i}$-canonical, $(T_{t_i}, \mathcal V_{t_i})\cdot g'\cdot g^{-1}= (T_{t_i}, \mathcal V_{t_i})$ so we have $(T_{t_i}, \mathcal V_{t_i})\cdot g=(T_{t_i}, \mathcal V_{t_i})\cdot g'$ and  $(T_t, \mathcal V_t)$ is well-defined for each $t\in V(T)$. The fact that the construction
$t\mapsto (T_t,\mathcal V_t)$ is $\Gamma$-canonical immediately follows from the definition. Finally, the tree-decomposition $(T_t, \mathcal V_t)$ is $\Gamma_t$-canonical because if it is not trivial, and $i\in I_{\infty}$ and $g\in \Gamma$ are such that $t_i\cdot g = t$, then $\Gamma_t=g^{-1}\cdot \Gamma_{t_i} \cdot g$ and thus $\Gamma_t$ induces a group action on $(T_t, \mathcal V_t)=(T_{t_i}, \mathcal V_{t_i})\cdot g$.
\end{proof}

Given two tree-decompositions $(T,\mathcal V),(T',\mathcal V')$ of a graph $G$, with $\mathcal{V}=(V_t)_{t\in V(T)}$ and $\mathcal{V}'=(V_t')_{t\in V(T')}$, we say that $(T',\mathcal V')$ \emph{refines} $(T,\mathcal V)$ with respect to some family $(T_t, \mathcal V_t)_{t\in V(T)}$ of tree-decompositions if for every $t\in V(T)$,
$T_t$ is a subtree of $T'$ such that $V_t= \bigcup_{s\in V(T_t)}V'_s$ and the trees $(T_t)_{t\in V(T)}$ are pairwise vertex-disjoint, cover $V(T')$ and for every edge $uv\in E(T)$, there exist $u'\in V(T_u), v'\in V(T_v)$ such that $u'v'\in E(T')$.

We say that $(T',\mathcal V')$ is a \emph{subdivision} of $(T,\mathcal V)$ if $T'$ is obtained from $T$ after considering a subset $E'\subseteq E(T)$ and doing the following for every edge $tt'\in E'$: we subdivide the edge $tt'$ (by adding a new vertex $t^*$ between $t$ and $t'$), and we add  a corresponding bag $V_{t^*}:=V_t\cap V_{t'}$ in the tree-decomposition. Note that if $(T',\mathcal V')$ is a subdivision of $(T,\mathcal V)$, the two tree-decompositions have the same edge-separations.

\medskip

The following result from \cite{CTTD}  will allow us to construct  canonical tree-decomposi\-tions inductively:
\begin{proposition}[Proposition 7.2 in \cite{CTTD}]
\label{prop: refinement}
Assume that $G$ is locally finite, $\Gamma$ is a group acting on $G$ and $(T,\mathcal V)$ is a $\Gamma$-canonical tree-decomposition of $G$ with finitely bounded adhesion, with $\mathcal{V}=(V_t)_{t\in V(T)}$. Assume that for every $t\in V(T)$, there exists a $\Gamma_t$-canonical tree-decomposition $(T_t,\mathcal V_t)$ of the torso $G\llbracket V_t \rrbracket$ of finitely bounded adhesion such that the edge-separations induced by $(T_t,\mathcal V_t)$ in $G\torso{V_t}$ are tight
and pairwise distinct, and the construction $t\mapsto (T_t,\mathcal V_t)$ is $\Gamma$-canonical. Then there exists a $\Gamma$-canonical tree-decomposition $(T',\mathcal V')$ of $G$ that refines $(T,\mathcal V)$ with respect to a family $(T'_t,\mathcal V'_t)_{t\in V(T)}$ such that for each $t\in V(T)$, $(T'_t, \mathcal V'_t)$ is a $\Gamma_t$-canonical tree-decomposition of $G\torso{V_t}$ which is a subdivision of $(T_t,\mathcal V_t)$, and such that every adhesion set of $(T', \mathcal V')$ is either an adhesion set of $(T,\mathcal V)$ or an adhesion set of some $(T_t, \mathcal V_t)$ for some $t\in V(T)$. Moreover, the construction $t\mapsto (T'_t, \mathcal V'_t)$ is $\Gamma$-canonical.
\end{proposition}

\begin{remark}
In the original statement of \cite[Proposition 7.2]{CTTD}, the fact that each tree-decomposition $(T'_t, \mathcal V'_t)$ is $\Gamma_t$-canonical and that  the construction $t\mapsto (T'_t, \mathcal V'_t)$ is also $\Gamma$-canonical is not stated explicitly, however the authors show it explicitly in the proof.
\end{remark}

Hence putting \cref{lem: canonical} together with \cref{prop: refinement}, we immediately get: 

\begin{corollary}
\label{cor: refinement}
 Assume that $G$ is locally finite, $\Gamma$ is a group acting on $G$, and $(T,\mathcal V)$ a $\Gamma$-canonical tree-decomposition of $G$ of finitely bounded adhesion, with $\mathcal{V}=(V_t)_{t\in V(T)}$. Let $\sg{t_i: i\in I_\infty}$ denote a set of representatives of the orbits $V_\infty(T)/\Gamma$ such that for each $i\in I_\infty$, 
 there exists a $\Gamma_{t_i}$-canonical tree-decomposition $(T_{t_i}, \mathcal V_{t_i})$ of $G\torso{V_{t_i}}$ with finitely bounded adhesion, such that the edge-separations induced by each $(T_{t_i},\mathcal V_{t_i})$ in $G\torso{V_{t_i}}$ are tight and pairwise distinct. Then there exists a $\Gamma$-canonical tree-decomposition of $G$ that refines $(T,\mathcal V)$ with respect to some  family $(T'_t, \mathcal V'_t)_{t\in V(T)}$ of $\Gamma_t$-canonical tree-decompositions of $G\torso{V_t}$ such that for each $i\in I_{\infty}$, $(T'_{t_i}, \mathcal V'_{t_i})$ is a subdivision of $(T_{t_i}, \mathcal V_{t_i})$, and for every $t\in V(T)\setminus V_\infty(T)$, 
 $(T'_t, \mathcal V'_t)$ is the trivial tree-decomposition of $G\torso{V_t}$. Moreover, the construction $t\mapsto (T'_t, \mathcal V'_t)$ is $\Gamma$-canonical.
\end{corollary}

The main objects of study of this paper are canonical tree-decompositions of quasi-transitive graphs. A crucial property is that the torsos or parts of the tree-decomposition are themselves quasi-transitive. This is proved in \cite[Proposition 4.5]{HamannStallings22} in the special case where $\Gamma$ acts transitively on $E(T)$. We give here a more general proof, which is self-contained.

\begin{lemma}
 \label{lem: quasitrans}
Let $k\in \mathbb N$,  let $G$ be a locally finite graph, and let $\Gamma$ be a group acting  quasi-transitively on $G$. Let $(T,\mathcal V)$,   with $\mathcal V=(V_t)_{t\in V(T)}$, be a $\Gamma$-canonical tree-decomposition of $G$ whose edge-separations are tight and have order at most $k$. Then, for any $t\in V(T)$, the group
$\Gamma_t := \mathrm{Stab}_{\Gamma}(t)$ induces a quasi-transitive action on $G[V_t]$, and thus also on $G\llbracket V_t \rrbracket$.
\end{lemma}

\begin{proof}
 \medskip
Let $G,\Gamma, (T, \mathcal V)$ be as described above.
By \cref{lem: TWcut}, as the edge-separations are tight and of bounded size, there are only finitely many $\Gamma$-orbits of $E(T)$. We fix an orientation $A$ of $E(T)$.
Let $e_1,\ldots, e_m\in A$ be representatives of each of the $\Gamma$-orbits of $A$, and let $\mean{e_1}, \ldots, \mean{e_m}$ denote their inverse pairs. Fix any node $t\in V(T)$.

We consider an arbitrary vertex $z\in V(G)$, and let $t_0\in V(T)$ be such that $z\in V_{t_0}$. Let $\Omega_z:=z\cdot \Gamma$ denote the $\Gamma$-orbit of $z$. We define the following subset of $V_t\cap \Omega_z$.
\[\Theta_z:=\sg{y \in V_t\cap \Omega_z: y=z\cdot \gamma \text{ for some }\gamma \text{ such that } t_0\cdot \gamma = t}.\]
We first show that $\Gamma_t$ acts transitively on $\Theta_z$. 
Let $y,y'\in \Theta_z$ and $\gamma,\gamma'\in \Gamma$ be such that $y=z\cdot \gamma$, $y'=z\cdot \gamma'$ and $ t_0\cdot \gamma=t_0\cdot \gamma'=t$. Then if we set $\alpha:=\gamma'^{-1}\cdot \gamma$, we have  \[t\cdot \alpha=t\cdot\gamma'^{-1}\cdot \gamma=t_0\cdot \gamma=t,\] and thus $\alpha \in \Gamma_t$. As $y'\cdot \alpha=y$, this  shows that $\Gamma_t$ acts transitively on $\Theta_z$. 

\medskip

For $i\in [m]$, we define:
\[\Psi_{i}:=\sg{y \in V_t\cap V_{t'}: \text{there exists } \gamma \in \Gamma \text{ such that }
(t,t') = e_i\cdot \gamma}\]
\[\Psi_{i+m}:=\sg{y \in V_t\cap V_{t'}: \text{there exists } \gamma \in \Gamma \text{ such that }
(t,t') = \mean{e_i}\cdot \gamma}.\]
We observe that if a vertex of $V_t\cap \Omega_z$ does not lie in  $\Theta_z$, it has to lie in one of sets $\Psi_i$ for $i\in [2m]$. To see this, let $y\in V_t\cap \Omega_z$, and $\gamma \in \Gamma$ be such that $y=z\cdot \gamma$. If $y\notin \Theta_z$, then $t_0\cdot \gamma\neq t$. In this case the unique path in $T$ from $t$ to $t_0\cdot \gamma$ contains at least one edge. Let $t'$ be the neighbor of $t$ on this path. As $V_{t_0\cdot \gamma}=V_{t_0}\cdot \gamma$ and $z\in V_{t_0}$,
we have $y\in V_t \cap V_{t_0 \cdot \gamma}$. Hence as $(T,\mathcal V)$ is a tree-decomposition, $y\in V_t\cap V_{t'}$.
Thus if we let $i$ be such that $(t,t')=e_i\cdot \beta$ or $(t,t')=\mean{e_i}\cdot \beta$ for some $\beta \in \Gamma$, we obtain that $y\in \Psi_i\cup \Psi_{i+m}$. This shows that $V_t$ is covered by the union of the sets $\Theta_z$, $z\in V(G)$ (there are at most $|V(G)/\Gamma|$ such sets), and the sets $\Psi_i$, $i\in [2m]$.

\medskip

We now show that $\Gamma_t$  acts quasi-transitively on each $\Psi_i$, $i\in [2m]$. 
Let $i\in [m]$, $y_1,y_2\in \Psi_{i}, t_1,t_2\in V(T)$ and $\beta_1,\beta_2\in \Gamma$ such that $(t,t_1),(t,t_2)\in E(T)$, $(t,t_1) = e_i\cdot \beta_1$ and $(t,t_2) = e_i\cdot \beta_2$. We set $\alpha:=\beta_1^{-1}\cdot \beta_2$ and note that $\alpha$ sends the directed edge $e_i\cdot \beta_1$ to $e_i\cdot \beta_2$. Let $S_i$ be the separator of $G$ associated to the edge-separation induced by the edge $e_i$ in $(T, \mathcal V)$. The previous remark implies that $\alpha$ sends $S_i\cdot \beta_1$ to $S_i\cdot \beta_2$ and that $\alpha\in \Gamma_t$. As for every $i\in [m]$, $S_i$ has size at most $k$, we just proved that the action of $\Gamma_t$ on $\Psi_{i}$ induces at most $k$ orbits. The case $i\in \sg{m+1,\ldots,2m}$ is exactly the same.

\medskip

As $V_t$ is covered by the union of the sets $\Theta_z$, $z\in V(G)$ (there are at most $|V(G)/\Gamma|$ such sets, and $\Gamma_t$ acts transitively on each of these sets), and the sets $\Psi_i$, $i\in [2m]$ (and the action of $\Gamma_t$ on each of these sets induces at most $k$ orbits), we have  $|V_t/\Gamma_t|\leq 2km+|V(G)/\Gamma|$, which implies that $\Gamma_t$ acts quasi-transitively on $G[V_t]$. As $(T,\mathcal{V})$ is $\Gamma$-canonical, for each $g\in \Gamma$ and each edge $e$ lying inside some adhesion set of the tree-decomposition, $g$ sends $e$ to a pair of vertices in another adhesion set of the tree-decomposition (and this pair of vertices must thus be joined by an edge in the corresponding torso). It follows that any automorphism $g\in \Gamma_t$ of $G[V_t]$ is also an automorphism of the torso $G\llbracket V_t \rrbracket$. Hence, $\Gamma_t$ also acts quasi-transitively on $G\llbracket V_t \rrbracket$. 
\end{proof}
Note that \cref{lem: quasitrans} still holds if we only require  $E(T)/\Gamma$ to be finite (instead of requiring the edge-separations of $(T,\mathcal V)$ to be tight).

\subsection{Tangles}
\label{sec: Tangles}

Tangles were introduced by Robertson and Seymour \cite{RSX} and play a
fundamental role in their proof of the Graph Minor Structure Theorem. We will consider here the
equivalent definition used by Grohe \cite{Gro16}.  A \emph{tangle of order}
$k$ in $G$ is a subset $\T$ of $\separ_{<k}(G)$ such that
\begin{enumerate}[label=($\mathcal T$\arabic*)]
\item\label{tangle: it1} For all separations $\Sep \in \separ_{<k}(G)$, either $\Sep \in \T$ or
  $(Z,S,Y) \in \T$;
\item\label{tangle: it2} For all separations $\Sepi{1}, \Sepi{2}, \Sepi{3} \in \T$, either $Z_1
  \cap Z_2 \cap Z_3 \neq \emptyset$ or there exists an edge with an endpoint in
  each $Z_i$.
\end{enumerate}
Note that \ref{tangle: it2} with $\Sepi{i}=\Sep$ for each $i\in [3]$ implies in particular that  for every separation $\Sep \in \T$, $Z \neq
  \emptyset$.   A tangle in $G$ will be called a
$G$-tangle, for brevity. Intuitively, a $G$-tangle is a consistent orientation
of the separations of $G$, pointing towards a highly connected region of $G$. We
refer the reader to \cite{RSX,Gro16} for background on tangles.  In general when $G$ is finite
there is a one-to-one correspondence between the $G$-tangles of order 1 and the
connected components of $G$, between the $G$-tangles of order 2 and the \emph{biconnected} components of $G$,
and between the $G$-tangles of order 3 and the
\emph{triconnected} components of $G$ 
(we omit the definitions as these notions will not be needed in the remainder of the paper, and instead refer the interested reader to \cite{RSX} and \cite{Gro3-conn}).

\medskip 

In infinite graphs, tangles can also be seen as a notion generalizing the notion of ends: for each end $\omega$ of a graph $G$ and every $k\geq 2$, we define the $G$-tangle $\T_{\omega}^k$ \emph{of order $k$ induced by $\omega$} by: 
\[ \T_{\omega}^k := \sg{\Sep\in \separ_{<k}(G): \omega ~\text{lives in a component of $Z$}}.\]

\medskip

The fact that $\T_{\omega}^k$ is indeed a tangle is a folklore result. One of the basic properties of tangles is that for any fixed model of a graph $H$ in a graph $G$, any $H$-tangle of order $k$ induces a $G$-tangle of order $k$. More precisely, if $\mathcal M = (M_v)_{v\in V(H)}$ is a model of $H$ in $G$ and $\Sep$ is a separation of order less than $k$ in $G$, then its \emph{projection} with respect to $\mathcal{M}$ is the separation $\pi_{\mathcal M}\Sep=\Sepp$ of $H$ of order less than $k$ defined by: 
$Y':=\sg{v\in V(H): M_v \subseteq Y}, S':=\sg{v\in V(H): M_v\cap S \neq \emptyset}$ and $Z':=\sg{v\in V(H): M_v \subseteq Z}$.

\smallskip

A proof of the following result can be found in \cite[$(6.1)$]{RSX} or in a more similar version in \cite[Lemma 3.11]{Gro16}. Its proof extends to the locally finite case.

\begin{lemma}[Lemma 3.11 in \cite{Gro16}]
 \label{lem: RS}
 Let $G$ be a locally finite graph. Let $\mathcal M = (M_v)_{v\in V(H)}$ be a
 model of a graph $H$ in $G$ and $\T'$ be an $H$-tangle of order $k\ge2$. Then the set 
 $$\T:= \sg{\Sep \in \separ_{<k}(G):\pi_{\mathcal M}\Sep\in \T'}$$
 is a $G$-tangle of order $k$, called the \emph{lifting} of $\T'$ in $G$ with respect to $\mathcal M$. 
\end{lemma}

\begin{remark}
\label{rem: proj-inj}
Assume that $\mathcal M$ is a faithful model of $H$ in $G$ with the property that for each $\Sepp\in \separ_{<k}(H)$, there exists some $\Sep\in \separ_{<k}(G)$ such that $\pi_{\mathcal M}\Sep=\Sepp$ and $S'=S$. Then the function that maps every tangle of order $k$ in $H$ to its lifting in $G$ with respect to $\mathcal M$ is injective. To see this, consider two distinct tangles $\mathcal T'_1\neq \mathcal T'_2$ of order $k$ in $H$. Then there exists some $\Sepp\in \mathcal T'_1$ such that $(Z',S',Y')\in \mathcal T'_2$. If we consider $\Sep\in \separ_{<k}(G)$ such that $\pi_{\mathcal M}\Sep=\Sepp$, we have $\Sep\in \mathcal T_1$ and $(Z,S,Y)\in \mathcal T_2$, where for each $i\in \sg{1,2}$, $\mathcal T_i$ denotes the lifting of $\mathcal T'_i$ with respect to $\mathcal M$. It then follows that $\mathcal T_1\ne \mathcal T_2$, as desired. Note that if $(T,\mathcal V)$ is a tree-decomposition with finitely bounded adhesion and $t\in V(T)$ is such that $G\torso{V_t}$ is a faithful minor of $G$, then any faithful model $\mathcal M$ of $G\torso{V_t}$ has the property we just described.
\end{remark}

If $\mathcal M = (M_v)_{v\in V(H)}$ is a model of $H$ in $G$, and
$\T$ is a tangle of $G$, then $\T' := \{\pi_{\mathcal{M}} \Sep : \Sep \in \T\}$
is called the \emph{projection} of $\T$. Note that $\T'$ is not a tangle in general.
Projecting is the converse operation of lifting in the sense that if
$\M$ is faithful, and $\T$ and $\T'$ are tangles of $G$ and $H$, then $\T$ is the
lifting of $\T'$ if and only if $\T'$ is the projection of $\T$.

\medskip

We define a partial order $\prec$ over the set of separations of a graph $G$ by letting for every two separations $\Sep, \Sepp$, $(Y,S,Z)\preceq(Y',S',Z')$
if and only if $S\cup Z\subsetneq S'\cup Z'$ or $(S\cup Z=S'\cup Z' \text{ and } S\subseteq S')$. Intuitively, $\Sep\preceq \Sepp$ means that $\Sep$ points towards a direction in a more accurate way than $\Sepp$ does.
Note that our definition of $\preceq$ is the same as in 
\cite[Subsection 3.2]{Gro16} and 
slightly differs from the more conventional one of \cite{ RSX, CTTD}. 

\medskip

A partially ordered set $(X, <)$ is said to be \emph{well-founded} if every strictly decreasing sequence of elements of $X$ is finite. In particular, if $(X, <)$ is well-founded then for every $x\in X$, there exists $y\in X$ which is minimal with respect to $<$ and such that $y\le x$. In the remainder of the paper, whenever we consider a minimal separation or a well-founded family of separations, we always implicitly refer to the partial order $\prec$ defined in the paragraph above.

\medskip
We will distinguish two types of tangles in infinite graphs:
\begin{itemize}
 \item the \emph{region tangles}, defined as those which are
   well-founded (with respect to the order $\prec$), and 
   
 \item the \emph{evasive tangles}, which contain some infinite decreasing sequence of separations (with respect to the order $\prec$).
\end{itemize}

The tangles we consider in this work will always have order at most $4$. Note that if $G$ is $3$-connected, an evasive tangle $\T$ of order $4$ is exactly a tangle $\T_{\omega}^4$ induced by an end $\omega$ of degree $3$. On the other hand, a region tangle is either a tangle of order $4$ induced by some end $\omega$ of degree at least $4$, or a tangle which is not induced by an end. 
For example, one can check that both graphs in \cref{fig: Gro2} have a unique tangle of order $4$ which is the tangle induced by their unique end (which is thick), and this tangle is a region tangle in both cases.

We say that a separation $(Y,S,Z)$ \emph{distinguishes} two tangles $\mathcal T, \mathcal T'$ if $\Sep \in \mathcal T$ and $(Z,S,Y)\in \mathcal T'$, or vice versa. We say that $\Sep$ distinguishes  $\mathcal T$ and $\mathcal T'$  \emph{efficiently} if there is no separation of smaller order distinguishing $\mathcal T$ and $\mathcal T'$.
A tree-decomposition $(T,\mathcal V)$ \emph{distinguishes} a set of tangles $\mathcal A$ if for every two distinct tangles $\mathcal T, \mathcal T'\in \mathcal A$ there exists an edge-separation of $(T, \mathcal V)$ distinguishing $\mathcal T$ and $\mathcal T'$. A separation is called \emph{relevant} with respect to $\mathcal A$ if it distinguishes at least two tangles of $\mathcal A$. A tree-decomposition is \emph{nice} (with respect to $\mathcal A$) if all its edge-separations are relevant (with respect to $\mathcal A$). 

\medskip

We will need the following result, which extends earlier results of \cite{RSX, DHL18}, and which is a canonical version of one of the main results of the grid-minor series in the locally finite case. We will only use it with $k=4$, but we nevertheless state the result in its most general form.
\begin{theorem}[Theorem 7.3 in \cite{CTTD}]
 \label{thm: tree-tangles}
 Let $k\geq 1$ and let $G$ be a locally finite graph. Then there exists a canonical tree-decomposition $(T,\mathcal V)$ of $G$ that efficiently distinguishes the set $\mathcal A_k$ of tangles of order at most $k$ and that is nice with respect to $\mathcal A_k$.
\end{theorem}

\begin{remark}
 \label{rem: relevant} The fact that $(T,\mathcal V)$ is nice in \cref{thm: tree-tangles} is not explicit in the original statement, however it directly follows from the proof. Moreover, the proof also ensures that the edge-separations of $(T,\mathcal V)$ are pairwise distinct.
\end{remark}

\subsection{An example}

We give here an example of a one-ended graph that excludes some minor and has infinitely many region tangles of order $4$. We show how to distinguish them on this example with a canonical tree-decomposition. As the application of \cref{thm: tree-tangles} allowing to distinguish all tangles of order $4$ is the very first step of our proof of \cref{thm: mainCTTD,thm: main2}, this example may also be useful to have some intuition on it.

\tikzexternaldisable
\begin{figure}[h]
\centering
\begin{tikzpicture}[scale=0.7]
\pgfmathsetmacro{\r}{5}
\pgfmathsetmacro{\long}{4}
\pgfmathsetmacro{\haut}{1}
\pgfmathsetmacro{\alpha}{1/3}

\pgfmathsetmacro{\hh}{int(\haut+1)}
\pgfmathsetmacro{\lll}{int(\long-1)}
\tikzstyle{every node}=[draw, circle,fill=black,minimum size=4pt, inner sep=0pt]
\foreach \j [evaluate=\j] in {1,...,\hh}
    {\foreach \i [evaluate=\i] in {1,...,\lll}
    {
        \pgfmathsetmacro{\y}{int(2*\j-1)}
        \pgfmathsetmacro{\ys}{\r*(sqrt(3)/2)*(\y + (2/3))}        
        \pgfmathsetmacro{\x}{int(\i)}
        \pgfmathsetmacro{\xs}{\r*(\x +0.5)}
        \pgfmathsetmacro{\H}{\r*(sqrt(3)/2)*(2/3)}
        \pgfmathsetmacro{\HP}{\H*\alpha}       
        \begin{scope}[xshift=\xs cm, yshift=\ys cm]
            \node (v1\x_\y) at (90:\H) {};
            \node (v2\x_\y) at (210:\H) {};
            \node (v3\x_\y) at (330:\H) {};
            \node (w1\x_\y) at (90:\HP) {};
            \node (w2\x_\y) at (210:\HP) {};
            \node (w3\x_\y) at (330:\HP) {};
            \node (z\x_\y) at (0:0) {};            
            \draw (v1\x_\y) -- (v2\x_\y) -- (w1\x_\y) -- (v2\x_\y) -- (v3\x_\y) -- (w1\x_\y) -- (z\x_\y) --(w2\x_\y) -- (v1\x_\y) -- (w1\x_\y) -- (v1\x_\y) -- (v3\x_\y) -- (w3\x_\y) -- (z\x_\y) -- (w3\x_\y) -- (v2\x_\y) --(w2\x_\y) -- (v3\x_\y) -- (w3\x_\y) -- (v1\x_\y);
        \end{scope}
    }
}
\foreach \j [evaluate=\j] in {1,...,\hh}
    {\foreach \i [evaluate=\i] in {1,...,\long}
    {
        \pgfmathsetmacro{\y}{int(2*\j-1)}
        \pgfmathsetmacro{\ys}{\r*(sqrt(3)/2)*(\y + 1)}        
        \pgfmathsetmacro{\x}{int(\i)}
        \pgfmathsetmacro{\xs}{\r*(\x)}
        \pgfmathsetmacro{\H}{\r*(sqrt(3)/2)*(2/3)}
        \pgfmathsetmacro{\HP}{\H*\alpha}       
        \begin{scope}[xshift=\xs cm, yshift=\ys cm]
            \node (v1\x_\y) at (-90:\H) {};
            \node (v2\x_\y) at (30:\H) {};
            \node (v3\x_\y) at (150:\H) {};
            \node (w1\x_\y) at (-90:\HP) {};
            \node (w2\x_\y) at (30:\HP) {};
            \node (w3\x_\y) at (150:\HP) {};
            \node (z\x_\y) at (0:0) {};            
            \draw (v1\x_\y) -- (v2\x_\y) -- (w1\x_\y) -- (v2\x_\y) -- (v3\x_\y) -- (w1\x_\y) -- (z\x_\y) --(w2\x_\y) -- (v1\x_\y) -- (w1\x_\y) -- (v1\x_\y) -- (v3\x_\y) -- (w3\x_\y) -- (z\x_\y) -- (w3\x_\y) -- (v2\x_\y) --(w2\x_\y) -- (v3\x_\y) -- (w3\x_\y) -- (v1\x_\y);
        \end{scope}
    }
}
\foreach \j [evaluate=\j] in {1,...,\haut}
    {\foreach \i [evaluate=\i] in {1,...,\lll}
    {
        \pgfmathsetmacro{\y}{int(2*\j)}
        \pgfmathsetmacro{\ys}{\r*(sqrt(3)/2)*(\y + 1)}        
        \pgfmathsetmacro{\x}{int(\i)}
        \pgfmathsetmacro{\xs}{\r*(\x+0.5)}
        \pgfmathsetmacro{\H}{\r*(sqrt(3)/2)*(2/3)}
        \pgfmathsetmacro{\HP}{\H*\alpha}       
        \begin{scope}[xshift=\xs cm, yshift=\ys cm]
            \node (v1\x_\y) at (-90:\H) {};
            \node (v2\x_\y) at (30:\H) {};
            \node (v3\x_\y) at (150:\H) {};
            \node (w1\x_\y) at (-90:\HP) {};
            \node (w2\x_\y) at (30:\HP) {};
            \node (w3\x_\y) at (150:\HP) {};
            \node (z\x_\y) at (0:0) {};            
            \draw (v1\x_\y) -- (v2\x_\y) -- (w1\x_\y) -- (v2\x_\y) -- (v3\x_\y) -- (w1\x_\y) -- (z\x_\y) --(w2\x_\y) -- (v1\x_\y) -- (w1\x_\y) -- (v1\x_\y) -- (v3\x_\y) -- (w3\x_\y) -- (z\x_\y) -- (w3\x_\y) -- (v2\x_\y) --(w2\x_\y) -- (v3\x_\y) -- (w3\x_\y) -- (v1\x_\y);
        \end{scope}
    }
}
\foreach \j [evaluate=\j] in {1,...,\haut}
    {\foreach \i [evaluate=\i] in {1,...,\long}
    {
        \pgfmathsetmacro{\y}{int(2*\j)}
        \pgfmathsetmacro{\ys}{\r*(sqrt(3)/2)*(\y + (2/3))}        
        \pgfmathsetmacro{\x}{int(\i)}
        \pgfmathsetmacro{\xs}{\r*(\x)}
        \pgfmathsetmacro{\H}{\r*(sqrt(3)/2)*(2/3)}
        \pgfmathsetmacro{\HP}{\H*\alpha}       
        \begin{scope}[xshift=\xs cm, yshift=\ys cm]
            \node (v1\x_\y) at (90:\H) {};
            \node (v2\x_\y) at (210:\H) {};
            \node (v3\x_\y) at (330:\H) {};
            \node (w1\x_\y) at (90:\HP) {};
            \node (w2\x_\y) at (210:\HP) {};
            \node (w3\x_\y) at (330:\HP) {};
            \node (z\x_\y) at (0:0) {};            
            \draw (v1\x_\y) -- (v2\x_\y) -- (w1\x_\y) -- (v2\x_\y) -- (v3\x_\y) -- (w1\x_\y) -- (z\x_\y) --(w2\x_\y) -- (v1\x_\y) -- (w1\x_\y) -- (v1\x_\y) -- (v3\x_\y) -- (w3\x_\y) -- (z\x_\y) -- (w3\x_\y) -- (v2\x_\y) --(w2\x_\y) -- (v3\x_\y) -- (w3\x_\y) -- (v1\x_\y);
        \end{scope}
    }
    \foreach \j [evaluate=\j] in {1,...,\haut}
    {\foreach \i [evaluate=\i] in {2}
    {
        \pgfmathsetmacro{\y}{int(2*\j)}
        \pgfmathsetmacro{\ys}{\r*(sqrt(3)/2)*(\y + 1)}        
        \pgfmathsetmacro{\x}{int(\i)}
        \pgfmathsetmacro{\xs}{\r*(\x+0.5)}
        \pgfmathsetmacro{\H}{\r*(sqrt(3)/2)*(2/3)}
        \pgfmathsetmacro{\HH}{\r*(sqrt(3)/2)*(1/3 + 0.05)}    
        \pgfmathsetmacro{\HHH}{\r*(sqrt(3)/2)*(1/3 - 0.05)}        
        \pgfmathsetmacro{\HP}{\H*\alpha}
        \pgfmathsetmacro{\HHP}{\HH* \alpha + 0.1}    
        \pgfmathsetmacro{\HHHP}{\HHH *\alpha - 0.05}        
        
        \begin{scope}[xshift=\xs cm, yshift=\ys cm]
            \node[draw, circle,fill=red,minimum size=4pt, inner sep=0pt, label=below:\textcolor{red}{$v_1$}] (v1\x_\y) at (-90:\H) {};
            \node[draw, circle,fill=red,minimum size=4pt, inner sep=0pt, label=above:\textcolor{red}{$v_2$}] (v2\x_\y) at (30:\H) {};
            \node[draw, circle,fill=red,minimum size=4pt, inner sep=0pt, label=above:\textcolor{red}{$v_3$}] (v3\x_\y) at (150:\H) {};
            \node[draw, circle,fill=blue,minimum size=4pt, inner sep=0pt, label=below:\textcolor{blue}{$w_1$}] (w1\x_\y) at (-90:\HP) {};
            \node[draw, circle,fill=blue,minimum size=4pt, inner sep=0pt, label=above:\textcolor{blue}{$w_2$}] (w2\x_\y) at (30:\HP) {};
            \node[draw, circle,fill=blue,minimum size=4pt, inner sep=0pt, label=above:\textcolor{blue}{$w_3$}] (w3\x_\y) at (150:\HP) {};
            \node[draw, circle,fill=blue,minimum size=4pt, inner sep=0pt, label=above:\textcolor{blue}{$z$}] (z\x_\y) at (0:0) {};            
            \draw (v1\x_\y) -- (v2\x_\y) -- (w1\x_\y) -- (v2\x_\y) -- (v3\x_\y) -- (w1\x_\y) -- (z\x_\y) --(w2\x_\y) -- (v1\x_\y) -- (w1\x_\y) -- (v1\x_\y) -- (v3\x_\y) -- (w3\x_\y) -- (z\x_\y) -- (w3\x_\y) -- (v2\x_\y) --(w2\x_\y) -- (v3\x_\y) -- (w3\x_\y) -- (v1\x_\y);
            \draw[thick, color=red] (v2\x_\y) -- (v3\x_\y) -- (v1\x_\y) -- (v2\x_\y);
            \draw[->, very thick, red] (90:\HH) -- (90:\HHH);
            \draw[->, very thick, red] (210:\HH) -- (210:\HHH);
            \draw[->, very thick, red] (330:\HH) -- (330:\HHH);
            \draw[->, very thick, red] (90:\HHHP) -- (90:\HHP);
            \draw[->, very thick, red] (210:\HHHP) -- (210:\HHP);
            \draw[->, very thick, red] (330:\HHHP) -- (330:\HHP);            
            \node[draw, circle,fill=red,minimum size=4pt, inner sep=0pt, label=below:\textcolor{red}{$v_1$}] (v1\x_\y) at (-90:\H) {};
            \node[draw, circle,fill=red,minimum size=4pt, inner sep=0pt, label=above:\textcolor{red}{$v_2$}] (v2\x_\y) at (30:\H) {};
            \node[draw, circle,fill=red,minimum size=4pt, inner sep=0pt, label=above:\textcolor{red}{$v_3$}] (v3\x_\y) at (150:\H) {};

        \end{scope}
    }
}
}

\end{tikzpicture}
\caption{A useful example.}
\label{fig: Tangles}
\end{figure}
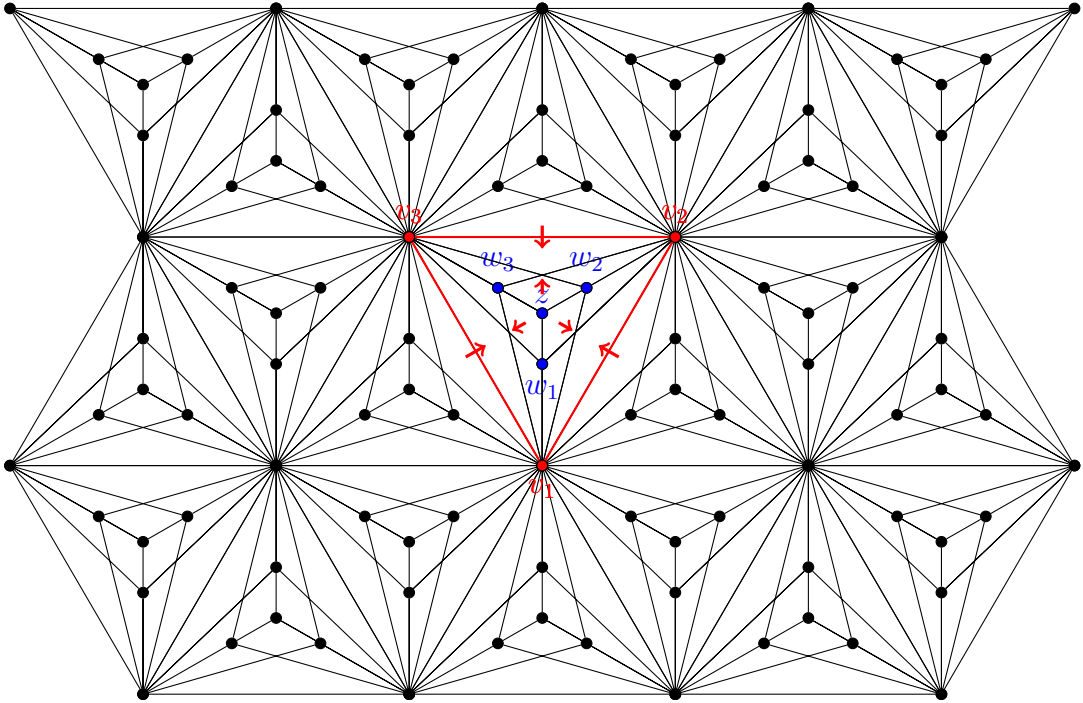

We consider the infinite graph $G$ (a finite section of which is illustrated in \cref{fig: Tangles}), which is obtained from the infinite triangular grid by adding in each triangular face $f=v_1v_2v_3$ three vertices $\sg{w_1, w_2, w_3}$ inducing a $K_{3,3}$ with the vertices of the triangle, and another vertex $z$ connected to each of the $w_i$'s. $G$ has two types of tangles of order $4$: one is the tangle $\T_{\omega}^4$ induced by the unique end $\omega$ of $G$, and all the others are the tangles $\T_{f}^4$ pointing towards each face $f=v_1v_2v_3$ of the triangular grid; more precisely, $\T_{f}^4$ has the same set of separations as $\T_{\omega}^4$ except for $(G- A, \sg{v_1, v_2, v_3}, \sg{w_1, w_2, w_3, z})\in \T_f^4$, where $A:=\sg{v_1, v_2, v_3, w_1, w_2, w_3, z}$. Note that with respect to our definition, all the tangles of order $4$ of $G$ are region tangles.
We represented with red arrows the two separations of $\T_f^4$ that are minimal with respect to the order $\prec$ but which  are not minimal separations of $\T_{\omega}^k$ (for one fixed face $f$). The three red arrows crossing the red triangle correspond to the minimal separation $(G- A, \sg{v_1, v_2, v_3}, \sg{w_1, w_2, w_3, z})$ of $\T_f^4$ that points towards the triangular face $v_1v_2v_3$, while the three arrows directed away from $z$ correspond to the minimal separation $(\sg{z}, \sg{w_1, w_2, w_3}, \sg{v_1, v_2, v_3}\cup (G- A))$ of $\T_f^4$.
The tree-decomposition $(T,\mathcal V)$ of \cref{thm: tree-tangles} distinguishing all the tangles of order $4$ is such that $T$ is a star with center $t_0\in V(T)$ such that $G\torso{V_{t_0}}$ is the infinite planar triangular grid. Then $T$ has one vertex $t_f$ for each face $f=v_1v_2v_3$ of $G\torso{V_{t_0}}$ and the bag $V_{t_f}$
is finite and contains the $7$ vertices $\sg{v_1,v_2,v_3,w_1,w_2,w_3,z}$ associated to $f$. Note that such a tree-decomposition enjoys the properties of \cref{thm: mainCTTD,thm: main2}. However, this is not always the case and we need in general to decompose further some torsos of the tree-decomposition given by \cref{thm: tree-tangles} in order to obtain such a decomposition.

\subsection{Tangles of order $4$: orthogonality and crossing-lemma}
In this section we introduce some notions from \cite{Gro16} and briefly explain  how to extend them in the locally finite case. Unless specified otherwise,  we assume in the whole section that the graphs we consider are locally finite and $3$-connected.

\medskip

A separation $\Sep\in \separ_{<4}(G)$ is said to be \emph{degenerate} if 
\begin{itemize}
    \item $\Sep$ has order $3$, \item $G[S]$ is an independent set, and 
    \item $|Y|=1$.
\end{itemize}
The following result from \cite{Gro16} immediately generalizes to locally finite graphs:
\begin{lemma}[Lemma 4.13 and Remark 4.14 in \cite{Gro16}]
 \label{lem: non-degenerate}
 Let $G$ be a  locally finite $3$-connected graph, and $\Sep$ be a proper
 separation of order $3$. Then $G\llbracket Z\cup S\rrbracket$ is a faithful
 minor of $G$ if and only if $\Sep$ is non-degenerate.
\end{lemma}

We say that the edge-separations of a tree-decomposition $(T,\mathcal{V})$ are \emph{non-degenerate} if for every $e\in E(T)$, none of the two edge-separations associated to $e$ are degenerate.

\begin{lemma}
  \label{lem: torso-minor} 
   Let $G$ be a locally finite $3$-connected graph and let $(T,\mathcal V)$, with $\mathcal{V}=(V_t)_{t \in V(T)}$, be a tree-decomposition of $G$ whose edge-separations have order $3$ and are non-degenerate. Then $G\llbracket V_t\rrbracket$ is a faithful minor of $G$ for each $t\in V(T)$.
\end{lemma}

\begin{proof}
Let $t\in V(T)$, and $t'$ be a neighbor of $t$ in $T$. Let $(Y_{t'}, S_{t'}, Z_{t'})$ be the edge-separation of $G$ associated to the (oriented) edge $(t',t)\in E(T)$, that is  $S_{t'}=V_t\cap V_{t'}$, $V_{t'}\subseteq Y_{t'}\cup S_{t'}$, and $V_t\subseteq Z_{t'}\cup S_{t'}$. By \cref{lem: non-degenerate}, there is a faithful model $(M^{t'}_v)_{v\in (Z_{t'}\cup S_{t'})}$ of $G\torso{Z_{t'}\cup S_{t'}}$ in $G$. As the only edges of $G\llbracket Z_{t'}\cup S_{t'}\rrbracket$ that are not edges of $G$ must be between pairs of vertices of $S_{t'}$, we may assume that every $M^{t'}_v$ has size $1$, except possibly when $v\in S_{t'}$, in which case the only vertices distinct from $v$ that $M^{t'}_v$ can have must lie in $Y_{t'}$. 
For every $v\in V_t$, we let:
$$M_v:=\bigcup_{\substack{t'\in V(T),\\ tt'\in E(T)}}M^{t'}_v.$$

We show that $(M_v)_{v\in V_t}$ is a faithful model of $G\llbracket V_t\rrbracket$ in $G$.
As $(T,\mathcal V)$ is a tree-decomposition, for every two distinct neighbors $t',t''$ of $t$ in $T$, $Y_{t'}\cap Y_{t''}=\emptyset$ so we must have $M^{t'}_v\cap M^{t''}_v=\sg{v}$ and $M^{t'}_v\cap M^{t''}_u=\emptyset$ for each distinct vertices $u,v \in V_t$. As $(M^{t'}_v)_{v\in (Z_{t'}\cup S_{t'})}$ is a model, we have $M^{t'}_u\cap M^{t'}_v = \emptyset$ for each $u\neq v\in V_t$. It follows that $M_u\cap M_v = \emptyset$ for each distinct $u,v\in V_t$. 
Now if  $uv\in E(G\torso{V_t})$ and $uv\notin E(G)$, there must exist some edge-separation $\Sepi{t'}$ such that $u,v\in S_{t'}$ and there exists a path from $u$ to $v$ in $G[S_{t'}\cup Y_{t'}]$. In particular, there must exist $u'\in M_u^{t'}$ and $v'\in M_v^{t'}$ such that $u'v'\in E(G)$. As $u'\in M_u$ and $v'\in M_v$,  we proved that $(M_v)_{v\in V_t}$ is a faithful model of $G\torso{V_t}$ in $G$.
\end{proof}

For every tangle $\mathcal T$ of a graph $G$, we denote by $\mathcal T_{\mathrm{min}}$ its set of minimal separations (here and in the remainder, minimality of separations is always with respect to the partial order  $\preceq$ defined above). If $\mathcal T$ has order $4$, then we let $\mathcal T_{\mathrm{nd}}$ be its set of non-degenerate minimal separations. 

\begin{remark}
 \label{rem: degenerate-tangle}
 Let $G$ be locally finite, let $\mathcal T$ be a $G$-tangle of order $4$, and let $(Y,S,Z)$ be a degenerate separation of $G$. Then $(Y,S,Z)\in \mathcal T$. This is a direct consequence of \cite[Lemma 3.3]{Gro16}, which states that if $\mathcal T$ is a tangle of order $k$ then for every separation $(Y,S,Z)$ of order $k-1$ such that $|Y\cup S|\leq \frac{3}{2}(k-1)$ we have $(Y,S,Z)\in \mathcal T$.
\end{remark}

For every tangle $\mathcal T$ of order $4$, we let:
$$X_{\mathcal T} := \bigcap_{\substack{(Y,S,Z)\in \mathcal T,\\ (Y,S,Z) \text{ is non-degenerate}}} (Z\cup S).
$$

Note that if $\mathcal T$ is an evasive tangle, then $X_{\mathcal T}$ is empty.  In this case, and because $G$ is $3$-connected, there exists a unique end $\omega$ of degree $3$ such that for any finite subset $\S$ of $ \T$, the end $\omega$ lies in $$\bigcap_{\substack{(Y,S,Z)\in \mathcal \S,\\ (Y,S,Z) \text{ is non-degenerate}}} (Z\cup S).$$

 \begin{remark}
 \label{rem: non-degenerate-tangle}
 If $(Y,S,Z),(Y',S',Z')\in \mathcal T$ are such that $(Y',S',Z')\preceq(Y,S,Z)$ and $(Y,S,Z)$ is non-degenerate, then it is easy to see that $(Y',S',Z')$ is also non-degenerate (recall that $G$ is $3$-connected). Also if $\mathcal T$ is a region tangle, for every $(Y,S,Z)\in \mathcal T$, there exists a separation $(Y',S',Z')\in \mathcal T_{\mathrm{min}}$ such that $(Y',S',Z')\preceq(Y,S,Z)$. These observations imply that if $\mathcal T$ is a region tangle of order $4$ and $G$ is $3$-connected and locally finite, then:

$$X_{\mathcal T} = \bigcap_{(Y,S,Z)\in \mathcal T_{\mathrm{nd}}} (Z\cup S).$$
\end{remark}

Two separations $\Sepi{1},\Sepi{2}$ are \emph{orthogonal} if $(Y_1\cup S_1)\cap
(Y_2\cup S_2) \subseteq S_1\cap S_2$ (see \cref{fig:orthogonal}). A set $\mathcal N$ of
separations  is said to be  \emph{orthogonal} if its separations are
pairwise orthogonal. One can easily show that the set of minimal separations of
a (region) tangle of order at most 3 is orthogonal. This does not hold for tangles of
order 4, but Grohe~\cite{Gro16} proved that for tangles of order 4, minimal separations can only cross
in a restricted way. Two separations $\Sepi{1}$ and $\Sepi{2}$ are
\emph{crossing} if $Y_1 \cap Y_2 = S_1 \cap S_2 = \emptyset$ and there is an
edge $s_1s_2 \in E(G)$, with $S_1\cap Y_2=\{s_1\}$ and $S_2\cap Y_1=\{s_2\}$
(see \cref{fig:crossing}). In this case, we call $s_1s_2$ the \emph{crossedge}
of $\Sepi{1}$ and $\Sepi{2}$. We denote by $\Exnd{\T}$
the set of crossedges of $\T_\mathrm{nd}$.  Lemma 4.16 from~\cite{Gro16}
generalizes to region tangles of order 4 of locally finite graphs:
\begin{lemma}[Lemma 4.16 and Corollary 4.20 in~\cite{Gro16}]
\label{lem: cross}
  Let $G$ be a locally finite $3$-connected graph. Let $\T$ be a region $G$-tangle of order 4. 
  Then every two distinct minimal separations of $\T$ are either crossing or orthogonal. Moreover, $\Exnd{\T}$ forms a
  matching in $G$.
\end{lemma}

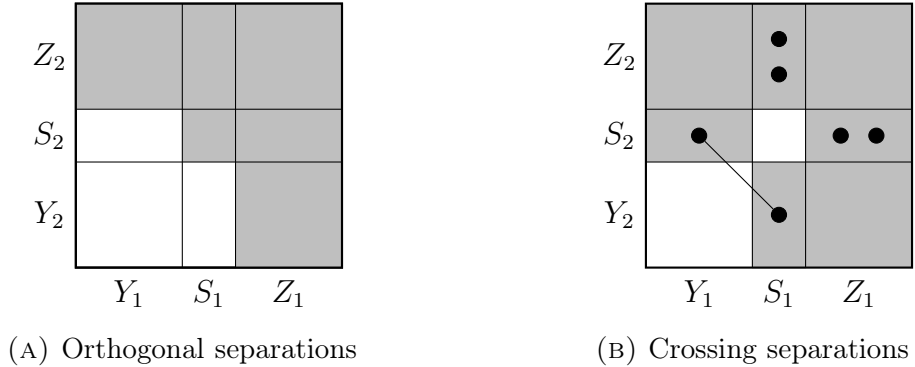
\begin{figure}[htb]
  \begin{subfigure}[b]{.5\linewidth}
        \centering
        \begin{tikzpicture}[scale=0.7]
          \draw[thick] (0,0) rectangle (5,5);
          \draw[draw=none,fill=lightgray] (0,3) rectangle (5,5);
          \draw[draw=none,fill=lightgray] (2,2) rectangle (5,3);
          \draw[draw=none,fill=lightgray] (3,0) rectangle (5,2);
          \draw[thick] (0,0) rectangle (5,5);
          \draw (2,0) rectangle (3,5);
          \draw (0,2) rectangle (5,3);
          \node (Y1) at (1,-.5) {$Y_1$};
          \node (S1) at (2.5,-.5) {$S_1$}; 
          \node (Z1) at (4,-.5) {$Z_1$}; 
          \node (Y2) at (-.5,1) {$Y_2$}; 
          \node (S2) at (-.5,2.5) {$S_2$}; 
          \node (Z2) at (-.5,4) {$Z_2$}; 
        \end{tikzpicture}
        \caption{Orthogonal separations}
        \label{fig:orthogonal}
    \end{subfigure}%
    \begin{subfigure}[b]{.5\linewidth}
        \centering
        \begin{tikzpicture}[scale=0.7]
          \draw[draw=none,fill=lightgray] (0,3) rectangle (5,5);
          \draw[draw=none,fill=lightgray] (3,0) rectangle (5,3);
          \draw[draw=none,fill=lightgray] (0,2) rectangle (2,3);
          \draw[draw=none,fill=lightgray] (2,0) rectangle (3,2);
          \draw[thick] (0,0) rectangle (5,5);
          \draw (2,0) rectangle (3,5);
          \draw (0,2) rectangle (5,3);
          \node[draw=black, circle, fill=black,inner sep = 2pt] (s11) at (1,2.5) {};
          \node[draw=black, circle, fill=black,inner sep = 2pt] (s12) at (3.66,2.5) {};
          \node[draw=black, circle, fill=black,inner sep = 2pt] (s13) at (4.33,2.5) {};
          \node[draw=black, circle, fill=black,inner sep = 2pt] (s21) at (2.5,1) {};
          \node[draw=black, circle, fill=black,inner sep = 2pt] (s22) at (2.5,3.66) {};
          \node[draw=black, circle, fill=black,inner sep = 2pt] (s23) at
          (2.5,4.33) {};

          \node (Y1) at (1,-.5) {$Y_1$};
          \node (S1) at (2.5,-.5) {$S_1$}; 
          \node (Z1) at (4,-.5) {$Z_1$}; 
          \node (Y2) at (-.5,1) {$Y_2$}; 
          \node (S2) at (-.5,2.5) {$S_2$}; 
          \node (Z2) at (-.5,4) {$Z_2$}; 
          \draw (s11) -- (s21);
        \end{tikzpicture}
        \caption{Crossing separations}
        \label{fig:crossing}
    \end{subfigure}
    \caption{Interaction between minimal separations. The white zones represent empty sets while the grey represent potentially non-empty sets.}
\end{figure}

In \cite{Gro16}, orthogonal sets of separations are presented as the nice case,
as they allow to efficiently find quasi-$4$-connected regions. We show that, up
to some additional assumptions, this observation still holds in the locally
finite case. We recall that for a tangle $\mathcal{T}$ of order 4, $\mathcal T_{\mathrm{nd}}$ denote its set of minimal non-degenerate separations.

\begin{lemma}
 \label{lem: ortho}
 Let $G$ be a locally finite $3$-connected graph.
 Let $\mathcal T$ be a region $G$-tangle of order $4$. Assume that $\mathcal T_{\mathrm{nd}}$ is orthogonal.
 Then $X_{\mathcal T}\neq \emptyset$ and the torso $G\llbracket X_{\mathcal T} \rrbracket$ has size $3$ or is a quasi-$4$-connected minor of $G$.
\end{lemma}

\begin{proof}
 If every separation of order $3$ in $G$ is degenerate, then $G$ is quasi-$4$-connected and all the separations of $\mathcal T_{\mathrm{nd}}$ are non-proper. It follows that $G=G\llbracket X_{\mathcal T}\rrbracket$ and the desired properties hold.
 
 Assume now that $G$ has a proper non-degenerate separation $(Y,S,Z)$ of order $3$. 
 As $\mathcal T$ is a region tangle, there is a separation $(Y',S',Z')\in \mathcal T_{\mathrm{min}}$ such that $(Y',S',Z')\preceq (Y,S,Z)$. As observed in \cref{rem: non-degenerate-tangle}, $(Y',S',Z')\in \mathcal T_{\mathrm{nd}}$. We claim that $S'\subseteq X_{\mathcal T}$ so $X_{\mathcal T}\neq \emptyset$: let $\Sepi{0}\in \mathcal T_{\mathrm{nd}}\setminus{\sg{\Sepp}}$. As $\Sepi{0}$ and $\Sepp$ are orthogonal, we must have: $S'\cap Y_0 = \emptyset$ so $S'\subseteq Z_0\cup S_0$. As $\mathcal T$ is a region tangle, the equality $X_{\mathcal T} = \bigcap_{(Y,S,Z)\in \mathcal T_{\mathrm{nd}}} (Z\cup S)$ holds, and 
 thus we proved that $S'\subseteq X_{\mathcal T}$, and so $X_{\mathcal T}\neq \emptyset$. Moreover, as $G$ is $3$-connected, the separations of $\mathcal T_{\mathrm{min}}$ have order $3$ so $|X_{\mathcal T}|\geq |S'|\geq 3$.
 
 We now assume that $|X_{\mathcal T}|\geq 4$ and show that $G\llbracket X_{\mathcal T}\rrbracket$ is quasi-$4$-connected. Since $G$ is 3-connected and $|X_{\mathcal T}|\geq 4$, $G\llbracket X_{\mathcal T}\rrbracket$ is 3-connected (any proper separation of order at most 2 in $G\llbracket X_{\mathcal T}\rrbracket$ would induce a proper separation of order at most 2 in $G$). 
 If $|X_{\mathcal T}|= 4$, then $G\llbracket X_{\mathcal T}\rrbracket$ is clearly also quasi-$4$-connected so we can assume that $|X_{\mathcal T}|\geq 5$. Suppose that $G\torso{X_{\mathcal T}}$ is not $4$-connected and let $\Sepi{0}$ be a proper separation of $G\llbracket X_{\mathcal T}\rrbracket$ of order at most $3$. We will prove that $|Y_0|=1$ or $|Z_0|=1$, which will immediately imply that $G\torso{X_{\T}}$ is quasi-$4$-connected.
 We let $Y_1$ be the union of all connected components of $G- S_0$ that intersect $Y_0$ or have a neighbor in $Y_0$. Let  $S_1:=S_0$ and $Z_1:=V(G)\setminus (S_0\cup Y_1)$. By definition of the torso $G\llbracket X_{\mathcal T}\rrbracket$,
 $\Sepi{1}$ is a proper separation of order at most $3$ in $G$, hence we must have $|S_0|=3$ as $G$ is $3$-connected. Assume first that $(Y_1, S_1, Z_1)\in \mathcal T$. If $\Sepi{1}$ is non-degenerate, then $X_{\mathcal T}\cap Y_0 \subseteq X_{\mathcal T} \cap Y_1 = \emptyset$ by definition of $X_{\mathcal T}$. It follows that $Y_0=\emptyset$, which contradicts the assumption that $(Y_0, S_0, Z_0)$ is proper.
 If $\Sepi{1}$ is degenerate, then $|Y_0|\leq |Y_1|= 1$ so as $\Sepi{0}$ is proper we must have $|Y_0|=1$.
 The case $(Z_1, S_1, Z_1)\in \T$ is symmetric. 
 Hence we proved that every separation $\Sep$ of $G\llbracket X_{\mathcal T}\rrbracket$ of order at most $3$ satisfies $|Y|\leq 1$ or $|Z|\leq 1$ so we are done. 
 
 Finally the fact that $G\llbracket X_{\mathcal T}\rrbracket$ is a minor of $G$ easily follows from \cref{lem: torso-minor}: we consider the tree-decomposition $(T,\mathcal V)$ where $T$ is a star with a central vertex $z_0$ and one edge $z_0z_i$ for each $\Sepi{i}\in \T_{\mathrm{nd}}$. We let $V_{z_0}:=X_{\T}$ and $V_{z_i}:= Y_i \cup S_i$ for each $\Sepi{i}\in \T_{\mathrm{nd}}$. The fact that $(T,\mathcal V)$ is a tree-decomposition follows from the orthogonality of $\T_{\mathrm{nd}}$. Hence by \cref{lem: torso-minor}, $G\torso{X_{\T}}$ is a minor of $G$.
\end{proof}

Whenever $\Tnd$ is not orthogonal, \cref{lem: ortho} does not hold anymore and if we want to obtain a canonical tree-decomposition, we will need to consider a larger set, whose torso is not necessarily quasi-$4$-connected, but can be defined uniquely from the structural properties of $\T$, which will ensure that the resulting decomposition is canonical.
For every region tangle $\mathcal T$ of order $4$ we let:
$$R_{\mathcal T} := \left(\bigcup_{(Y,S,Z)\in \mathcal T_{\mathrm{nd}}}S\right)\cup \left(\bigcap_{\substack{(Y,S,Z)\in \mathcal T_{\mathrm{nd}}}} Z\right).
$$
The set $R_{\T}$ corresponds to the set called $R^{(0)}$ in \cite[Section 4.5]{Gro16}.
Note that we always have $X_{\mathcal T}\subseteq R_{\mathcal T}$ and that equality holds when $\mathcal T_{\mathrm{nd}}$ is orthogonal. To illustrate the definition of $R_{\T}$, it is helpful to go back to the  graph $G$ on \cref{fig: Gro2} (left). Then $G$ has a unique tangle $\T$ of order $4$, the set $\Tnd$ is the set of  separations $\Sep$ of order $3$ where $Y$ is a triangular face of $G$, $S=N(Y)$ and $Z=V(G)\setminus (Y\cup S)$. Hence on this example, the set of crossedges $\Exnd{G}$ is the set of edges joining two triangular faces, and $R_{\T}=V(G)$. 

\medskip

While the proof of the following result was originally written for finite graphs, it immediately generalizes to locally finite graphs.
\begin{lemma}[Lemma 4.32 in \cite{Gro16}]
 \label{lem: Grominor}
 If $G$ is a locally finite 3-connected graph and if $\T$ is a region tangle of order $4$ in $G$, then
 $G\torso{R_{\T}}$ is a faithful minor of $G$.
\end{lemma}

For each $\Sep\in \T_{\mathrm{nd}}$, the \emph{fence} $\mathrm{fc}(S)$ of $S$ in $G$ is
the union of 
\begin{itemize}
    \item the subset of vertices of $S$ that are not the endpoint of some crossedge of
$\mathcal T$, and
\item the subset of vertices $s'$ such that $ss'$ is a crossedge of $\mathcal T$ and $s\in S$.
\end{itemize}
In particular, as the crossedges form a matching in $G$ (\cref{lem: cross}), $|\mathrm{fc}(S)|=|S|=3$ for each $\Sep \in \T_{\mathrm{nd}}$.
A consequence of \cref{lem: cross} is the following:
\begin{lemma}
\label{lem: cross-star} 
Let $G$ be a locally finite $3$-connected graph and $\mathcal T$ be a region $G$-tangle of order $4$. Then $G$ has a tree-decomposition $(T, \mathcal V)$ of adhesion $3$ where $\mathcal{V}=(V_t)_{t\in V(T)}$ and $T$ is a star with central vertex $z_0$ such that $V_{z_0}=R_{\mathcal T}$.
If moreover, $\mathcal T$ is the unique  $G$-tangle of order $4$, then $(T,\mathcal V)$ is canonical and every bag except possibly $V_{z_0}$ is finite.
\end{lemma}

\begin{proof}
 We let:  
 $$V(T):=\sg{z_0}\cup \sg{z_C : C \text{ connected component of } G- R_{\T}}
 $$
 where we choose the $z_C$'s to be pairwise distinct nodes. We let $T$ be the star with vertex set $V(T)$ and central vertex $z_0$, and we define $\mathcal V=(V_t)_{t\in V(T)}$ by setting $V_{z_0}:= R_{\mathcal T}
 $, and for each connected component $C$ of $G- R_{\T}$: 
 $V_{z_C}:=C\cup N(C)$. It is not hard to check that $(T, \mathcal V)$ is a tree-decomposition of $G$. By \cite[Lemma 4.31]{Gro16} (whose proof extends to the locally finite case), for each component
 $C$ of $G- R_{\T}$ there exists a unique separator $S$ such that $\Sep\in \T_{\mathrm{nd}}$ for some separation $\Sep$, $N(C)=\fc(S)$ and $C\subseteq Y$. This implies that $(T,\mathcal{V})$ has adhesion $3$, so in particular its edge-separations are tight.

We now prove the second part of \cref{lem: cross-star} and assume that $\mathcal T$ is the unique  tangle of order $4$ of $G$. Then $\mathcal T$ is $\Aut(G)$-invariant, and $(T,\mathcal V)$ is clearly canonical. If some $V_{z_C}$ is infinite for some $z_C\neq z_0$, then as $G$ is locally finite, $G[V_{z_C}]$ has at least one infinite connected component, and hence there exists some end $\omega$ living in $G[V_{z_C}]$. 
In particular, $\omega$ induces some $G$-tangle $\mathcal T_{\omega}$ of order $4$. We let $\Sep\in \T_{\mathrm{nd}}$ be the separation given by \cite[Lemma 4.31]{Gro16} such that $N(C)=\fc(S)$ and $C\subseteq Y$.
Then $\Sep$ distinguishes $\mathcal T_{\omega}$ from $\mathcal T$, which contradicts the uniqueness of $\mathcal T$ in $G$. 
\end{proof}

Note that in the non-orthogonal case, the tree-decomposition from \cref{lem: cross-star} is not the same as the one from \cite{Gro16}, as the torso $G\torso{R_{\T}}$ associated to the center of the star might not be quasi-$4$-connected. However, we will prove in \cref{sec: planar} that $G\torso{R_{\T}}$ still enjoys the same useful properties as a quasi-transitive quasi-$4$-connected graph, namely it is either planar or has bounded treewidth. We note that the crucial ingredient that allows us to obtain a canonical tree-decomposition in the second part of \cref{lem: cross-star} (contrary to Grohe's decomposition) is the assumption that $\mathcal{T}$ is the unique $G$-tangle of order 4. So one of the most important steps in the proof of our main results will be a reduction to the case where graphs have a single tangle of order 4.

\subsection{Contracting a single crossedge}
\label{sec: single}

In what follows, we let $G$ be a locally finite $3$-connected graph, and
$\mathcal T$ be a region tangle of order $4$ in $G$. Recall that by \cref{lem: cross} the set $\Exnd{\T}$ of the crossedges forms a
matching. We will see that contracting a crossedge results in a
3-connected graph $G'$ that has a tangle $\T'$ of order 4 induced by
$\T$~\cite[Subsection 4.5]{Gro16}. More precisely, $\T'$ contains as a subset the projection of
$\T$ with respect to the minor $G'$ of $G$. We
give here an overview of the lemmas stated in~\cite[Subsection 4.5]{Gro16} which
all hold when $G$ is locally finite instead of finite, by using the exact same
proofs. The only additional property that we need in the locally finite case is
that $\T'$ is still a region tangle, which is proved in \cref{lem: regiontangle}
below.
 
In the remainder of this subsection, we let $\Sepi{1}$ and $\Sepi{2}$ be two
crossing separations of $\T_\mathrm{nd}$ with crossedge $s_1s_2$. \emph{Contracting $s_1s_2$}
consists in deleting $s_1$ and $s_2$ and adding a new
vertex $s'$ whose neighborhood is equal to $N_G(s_1)\cup N_G(s_2)\setminus \{s_1,s_2\}$. We denote by $G'$ the graph obtained after contracting $s_1s_2$. The
\emph{projection} (referred to as \emph{contraction} in~\cite{Gro16}) of a set $X$
of vertices of $G$ is defined as
$$X^{\vee}:= \left\{\begin{array}{ll}
X & \text{ if } X \cap \{s_1,s_2\} = \emptyset\\ 
X\setminus\{s_1,s_2\} \cup \{s'\} & \text{ if } X \cap \{s_1,s_2\} \neq \emptyset.\\ 
\end{array}\right.
$$
Given a set $X'$ of vertices of $G'$, the \emph{expansion} $X'_\wedge$ of $X'$ is defined as
$$X'_{\wedge}:= \left\{\begin{array}{ll}
X' & \text{ if } s' \notin X' \\ 
X'\setminus\{s'\} \cup \{s_1, s_2\} & \text{ if } s' \in X' .\\ 
\end{array}\right.
$$

Observe that for all $X' \subseteq V(G')$, we have $(X'_\wedge)^\vee = X'$ and for all $X \subseteq V(G)$,
we have $X \subseteq (X^\vee)_\wedge$ (where the inclusion might be strict). We also define for every $\Sep \in \separ_{<4}(G)$:
$$\Sep^{\vee}:= \left\{\begin{array}{ll}
(Y^{\vee}\setminus \sg{s'}, S^{\vee}, Z^{\vee}\setminus \sg{s'}) & \text{ if } S \cap \{s_1,s_2\} \neq \emptyset\\ 
(Y^{\vee}, S^{\vee}, Z^{\vee}) & \text{ if } S \cap \{s_1,s_2\} = \emptyset.\\ 
\end{array}\right.
$$

Note that $\Sep^{\vee}$ is exactly the projection $\pi_{\mathcal M}\Sep$ with respect to the model
$\M = (\sg{v}_\wedge)_{v \in V(G')}$ of $G'$ in $G$. 
\smallskip

In the context of finite graphs,~\cite{Gro16} proves the following lemmas that
extend directly to the locally finite case:

\begin{lemma}[Corollary 4.24 in \cite{Gro16}] \label{cor: G3conn}
  The graph $G'$ resulting from the contraction of $s_1s_2$ is 3-connected. 
\end{lemma}
\begin{lemma}[Lemmas 4.26 and 4.27 in \cite{Gro16}]
\label{lem: Gro-T'}
There exists a tangle $\mathcal T'$ of order $4$ in $G'$ containing the projection of $\mathcal T$ with respect to the model $\M = (\sg{v}_\wedge)_{v \in V(G')}$.
\end{lemma}

Note that the projection of $\mathcal T$ with respect to $\M$ is exactly the set $\sg{\Sep^{\vee}: \Sep \in \mathcal T}$.
In the remainder of this subsection, we let $\mathcal T'$ be the tangle given by Lemma \ref{lem: Gro-T'}. In \cite{Gro16}, the author gives an explicit definition of $\T'$, but for the sake of clarity we only summarize here the properties of $\T'$ that will be of interest for our purposes.

Note that the inclusion $\sg{\Sep^{\vee}: \Sep \in \T}\subseteq \T'$ is strict in general, as some separations from $\separ_{<4}(G')$ might not be projections of separations from $\separ_{<4}(G)$. The next lemma intuitively states that every separation of $\T'$ is close to an element from $\sg{\Sep^{\vee}: \Sep \in \T}\subseteq \T'$.
\begin{lemma}[Definition of $\T'$ and Lemmas 4.23 and 4.25 in \cite{Gro16}]
\label{lem: essential-sep}
 For every separation $\Sepp\in \T'$ such that $s'\in S'$ and $G'[Z']$ is connected, there exists a separation $\Sep\in \T$ such that $S^{\vee}=S'$ and $Z\setminus S_{\vee}= Z\setminus \sg{s_1,s_2}= Z'$.
\end{lemma}

As Lemma \ref{lem: essential-sep} is not exactly stated this way in \cite{Gro16}, we briefly sketch how to obtain it. If $\Sepp\in \T'$ is such that $s'\in S'$, then by \cite[Lemma 4.25]{Gro16} there exists a (unique) connected component $C$ of $G\setminus S'_{\vee}$ such that the separations $\Seppp$ of $\T'$ such that $S''=S'$ are exactly the ones such that $C\subseteq Z''$, and for which every separation $\Sep\in \T$ such that $S^{\vee}=S'$ satisfies $C\subseteq Z$. \cite[Lemmas 4.23, 4.25]{Gro16} and the fact that $\T$ is a tangle ensure the existence of a separation $\Sep\in \T$ such that $S^{\vee}=S'$ and $C=Z\setminus \sg{s_1,s_2}=Z\setminus S$. In particular by Lemma \ref{lem: Gro-T'} the projection $\Sep^{\vee}=(Y\setminus S', S', C)$ is in $\T'$ so if we assume that $G'[Z']$ is connected, the choice of $C$ imposes $C\subseteq Z'$, thus $Z'=C$. It implies that $\Sep$ satisfies the property described in Lemma \ref{lem: essential-sep}.

\begin{lemma}\label{lem: regiontangle}
 $\T'$  is a region tangle. 
\end{lemma}

\begin{proof}
  
  Assume
  for the sake of contradiction that $\T'$ contains an infinite strictly decreasing sequence of
  separations $\Seppi{n}_{n\in \NN}$. By \cref{cor: G3conn}, $G'$ is 3-connected, so the only possible non-proper separation $\Sepi{n}$ is for $n=0$, thus we may assume that all the
  separations $\Sepi{n}$ are tight. By \cref{lem: TWcut}, there are finitely many integers $n$
  for which $s' \in S'_n$. Up to extracting an infinite subsequence, one can assume that for all $n$,
  either $s' \in Y'_n$ or $s' \in Z'_n$. If there exists $N$ such that $s' \in
  Y'_N$ then for all $n \ge N$ we must have $s'\in Y_n$ by definition of $\preceq$. Up to extracting another infinite subsequence, we can assume that either $s' \in Y'_n$ for
  all $n$ or $s' \in Z'_n$ for all $n$. As a result, and because $\mathcal T'$ contains the projection of $\mathcal T$ with respect to $\M$,
  $((Y'_n)_\wedge, S'_n,
  (Z'_n)_\wedge)_{n \in \NN}$ is an infinite decreasing sequence of
separations of order 3 in $\T$, contradicting the fact that $\T$ is well-founded.
\end{proof}

We conclude this subsection with the following result relating the degeneracy of minimal separations in $G$ and $G'$. Its proof is the same as the proof of~\cite[Lemma 4.28]{Gro16}, which directly translates to the locally finite case. To be more precise, we also need the additional assumption that $\T'$ is a region tangle to make the proof work, which is given by \cref{lem: regiontangle}.

\begin{lemma}[Lemma 4.28 and Corollary 4.29 in~\cite{Gro16}]\label{lem:region_contracted_tangle}
  Either $G'$ is 4-connected and $\T'_{\mathrm{min}} =  \{(\emptyset, \emptyset, V(G'))\}$, or 
  $$\T'_{\mathrm{min}} = \{\Sep^\vee : \Sep \in \T_{\mathrm{min}} \text{ and }S^\vee \text{ is a
    separator of } G'\}.$$

   In the latter case, for all $\Sep \in \T_{\mathrm{min}}$, $\Sep$ is non-degenerate
  if and only if $\Sep^\vee$ is non-degenerate. Moreover, $\Exnd{\T'}=\Exnd{\T}\setminus \sg{s_1s_2}$.
\end{lemma}

\subsection{Contracting all the crossedges}
\label{sec: all}

In the previous subsection we studied the consequences of contracting a single crossedge in $G$. However, in our application we will need to contract \emph{all} crossedges of $\Exnd{\T}$ (which form a matching in $G$). We now study how this affects $G$.

Before going further, we will need to introduce some notation, extending the
notation from \cite{Gro16} to the infinite case. For convenience we write
$M := \Exnd{\T}$ (and recall that $M$ is a matching in $G$). For every subset
$L\subseteq M$ of crossedges, we let $\GL$ be the graph obtained from $G$ after
contracting each edge $uv\in L$ into a new vertex $s_{u,v}$. Note that the order in which the edges are
contracted is irrelevant in the definition of $\GL$.

We denote
$\overline L = M \setminus L$. In this section we will also often use the notation $L-L'$ instead of $L\setminus L'$, to avoid any possible confusion when reading superscripts (for instance we will write $G^{\backslash L-L'/}$ instead of $G^{\backslash L\setminus L'/}$). 

For every $L \subseteq M$, for every vertex $x \in V(G)$, we let 
$$x\Contr{L}:= \left\{\begin{array}{ll}
x & \text{ if } x \in X
\setminus V(L)\\ 
s_{u,v} & \text{ if } x \text{ is the endpoint of a crossedge } uv \in L.\\ 
\end{array}\right.$$

For every subset $X\subseteq V(G)$ of vertices, we let $X\Contr{L} :=
\{x\Contr{L} :x \in X\}$ be the projection of $X$ to $\GL$.

\begin{remark}\label{rem:contract_comut}
  Note that for every disjoint subsets $K, L \subset M$ and for all $X \in V(G)$,
  $X\Contr{K\cup L} = (X\Contr{L})\Contr{K} = (X\Contr{K})\Contr{L}$.
\end{remark}

For every
$X'\subseteq V(\GM)$ and $L \subseteq M$, let $X'\Expand{L}$ denote the maximal
set $X\subseteq V(G\Contr{\overline L})$ such that $X\Contr{L}=X'$.
In other words $X'\Expand{L}$ is the set of vertices obtained after ``uncontracting'' the
edges of $L$ in $X$. Note that with the notation introduced above we have $G=
G\Contr{\emptyset}$. Given a separation $\Sep$ of $G$, we define $$\Sep\Contr{L}
:= \left(Y\Contr{L}\setminus S\Contr{L},S\Contr{L},Z\Contr{L}\setminus S\Contr{L}\right).$$


Note that when $L=\sg{s_1s_2}$ consists of a single edge, we recover the notions of the previous subsection; with our previous notation this gives: $x\Contr{L}=x^{\vee}$, $X\Contr{L}=X^{\vee}$ and $\Sep\Contr{L}=\Sep^{\vee}$.

For each finite subset of crossedges $L\subseteq M$, and every enumeration $(e_1, \ldots, e_\ell)$ of the edges of $L$, we let $\mathcal T\Contr{(e_1, \ldots, e_\ell)}$ denote the tangle of $G\Contr{L}$ obtained after iteratively applying Lemma \ref{lem: Gro-T'} to the graphs $G_0:=G, G_1, \ldots, G_\ell$ with $G_i:=G\Contr{\sg{e_1, \ldots, e_i}}$ for each $i\in [\ell]$.
\begin{lemma}[Lemma 4.30 (5) in \cite{Gro16}]
\label{lem: Gro-enum} 
For every enumeration $(e_1, \ldots, e_\ell)$ of a finite set $L\subseteq M$ of crossedges and every permutation $\sigma$ of $[\ell]$, $\mathcal T\Contr{(e_1, \ldots, e_\ell)}=\mathcal T\Contr{(e_{\sigma(1)}, \ldots, e_{\sigma(\ell)})}$.
\end{lemma}
In the remainder of the subsection, for every finite subset $L\subseteq M$, we will denote with $\TL$ the unique tangle associated to any enumeration of $L$ given by Lemma \ref{lem: Gro-enum}.

Intuitively when $G$ is finite, one of the main properties of $\TL$ is that separations of $\TL_{\mathrm{nd}}$ are in correspondence with separations of $\mathcal T_{\mathrm{nd}}$, and that the only crossing pairs between elements of $\TL$ correspond to pairs which were already crossing in $\T$. Thus after each contraction, we reduce the number of crossedges, hence when $L=M$, the family $\TL_{\mathrm{nd}}$ must be orthogonal and we can apply results from previous subsections to the graph $\GL$. We now show formally how to extend the relevant proofs of \cite{Gro16} to the locally finite case.

\smallskip

In \cite{Gro16}, the author proved that if $G$ is finite and 3-connected, for
every $L \subseteq M$, the graph $\GL$ is $3$-connected and that $\TL$ is a
tangle of order $4$ induced by $\mathcal T$ in $\GL$. Using the results from the previous subsection, this immediately extends to $\GL$ and $\TL$ when $G$ is locally finite and $L$ is finite, by
induction on the size of $L$.

\begin{theorem}[Generalization of Lemma 4.30 in~\cite{Gro16}]\label{thm:contract tangle1}
  Let $L\subseteq M$ be a finite set of crossedges. Then we have
  \begin{enumerate}
  \item $G\Contr{L}$ is 3-connected
  \item $\T\Contr{L}$ is a region tangle of order 4 of $G\Contr{L}$ such that 
  $$\T\Contr{L}_{\mathrm{min}} = \{\Sep\Contr{L} : \Sep \in \T_{\mathrm{min}} \text{ such that }S\Contr{L} \text{ is a
    separator of } G\Contr{L}\}$$ or $\T\Contr{L}_{\mathrm{min}} = \{(\emptyset,
  \emptyset, V(G'))\}$ if $L = M$ is finite and $G\Contr{L}$ is
  4-connected.
  \item $\Exnd{\T\Contr{L}} = \Exnd{\T}\setminus L$
  \item $\TL$ contains the projection $\sg{(Y,S,Z)\Contr{L}: \Sep\in \T}$ of $\T$ with respect to the model $\M=(\sg{v}\Expand{L})_{v\in V(G\Contr{L})}$.
  \end{enumerate}
\end{theorem}


We will now extend Theorem \ref{thm:contract tangle1} to the case where $L\subseteq M$ is infinite. Given a set $X \subseteq V(G\Contr{M})$, we denote by
$M(X)\subseteq M$ the subset of edges of $G$ contracted to a vertex in $X$.

\begin{lemma}\label{lem:contract connected}
  The graph $\GM$ is 3-connected.
\end{lemma}
\begin{proof}
  Assume for the sake of contradiction that $\GM$ has a  separator $S$ of order at most 2. Then the set $L:= M(S)$ has size at most 2 and $S$ is a separator of order 2
  of 
  $G\Contr{L}$. This contradicts \cref{thm:contract tangle1}.
\end{proof}

We now let $L\subseteq M$ be any (not necessarily finite) subset of crossedges and give a general definition of $\TL$ extending the previous one. For every $\Sepp\in \separ_{< 4}(G\Contr{L})$, we let $L':=M(S')$. Note that $L'$ is finite, and that $(Y'\Expand{\overline L}, S', Z'\Expand{\overline L})$ is a separation of order at most $3$ in $G\Contr{L'}$. We define $\TL$ as the family of separations $\Sepp$ of $G\Contr{L}$ such that $(Y'\Expand{\overline L}, S', Z'\Expand{\overline L})\in \T\Contr{L'}$. Note that when $L$ is finite, $(Y',S',Z')=(Y'\Expand{\overline L}, S', Z'\Expand{\overline L})\Contr{L- L'}$ and thus iterative applications of Lemma \ref{lem: Gro-T'} together with Lemma \ref{lem: Gro-enum} imply that our definition of $\TL$ coincides with the one we gave above for finite subsets  $L\subseteq M$.

Thanks to~\cref{rem:contract_comut}, for all $L\subseteq M$, $\T\Contr{M} =
(\T\Contr{L})\Contr{\overline L}$ and $\GM = (\GL)\Contr{\overline L}$.
We say that a set $X \subseteq V(G)$ \emph{hits the edges of $L$ once} if for
all $e \in L$, $|X \cap e| = 1$.

\begin{lemma}
    $\T\Contr{M}$ is a region tangle of order 4 in $\GM$ such that $\sg{\Sep\Contr{M}: \Sep \in \T}\subseteq \T\Contr{M}$.
\end{lemma}
\begin{proof}

 We first prove that $\T\Contr{M}$ is a tangle of order 4. To prove
  \ref{tangle: it1}, let $\Sepp \in \separ_{<4}(G\Contr{M})$ and $L :=
  M(S')$. Then $L$ has size at most $3$, so by \cref{thm:contract
    tangle1}, $\T\Contr{L}$ is a region tangle of order $4$ of $G\Contr{L}$.
  As  $\Sep := (Y'\Expand{\overline L}, S',Z'\Expand{\overline L})$ is
  a separation of order 3 of $G\Contr{L}$ and $\T\Contr{L}$ is a tangle of
  order 4, either $\Sep \in \T\Contr{L}$ or $(Z,S,Y) \in
  \T\Contr{L}$. By definition of $\T\Contr M$ we then have either $(Z',S',Y')\in \T\Contr M$ or $\Sepp\in \T\Contr M$, implying that $\T\Contr M$ satisfies \ref{tangle: it1}.

\medskip  

  To prove \ref{tangle: it2}, let
  $\Seppi{1}, \Seppi{2}, \Seppi{3} \in \T\Contr{M}$. Let
  $L := M(S'_1 \cup S'_2 \cup S'_3)$. Once again $L$ is finite with size at
  most $9$ and for all $i\in \sg{1,2,3}$,
  $\Sepi{i} := ((Y'_i)\Expand{\overline L}, S'_i,(Z'_i)\Expand{\overline
    L})$ is a separation of order 3 of $G\Contr{L}$. 
  
  \begin{claim}
   \label{clm: tangle-expand}
   For every $i\in \sg{1,2,3}$, $\Sepi{i}\in \T\Contr{L}$.
  \end{claim}
  \begin{proofofclaim}
   Assume that $i=1$, the other cases being symmetric. We let $L_1:=M(S_1')$. Then by definition of $\T\Contr M$, $\Sepppi{1}:=((Y'_1)\Expand{\overline{L_1}}, S_1', (Z'_1)\Expand{\overline{L_1}})\in \T\Contr{L_1}$. 
   Our goal is to show that $\Sepi{1}=\Sepppi{1}\Contr{L- L_1}$. As both $L$ and $L_1$ are finite and  $L_1\subseteq L$, iterative applications of Lemmas \ref{lem: Gro-T'} and Lemma \ref{lem: Gro-enum} imply $\T\Contr{L}$ must contain the projection of $\T\Contr{L_1}$ with respect to the model $\M=(\sg{v}\Expand{(L\setminus L_1)})_{v\in V(G)\Contr{L}}$. Thus if we succeed to prove that 
   \begin{equation}\label{eq:1}\Sepi{1}=\Sepppi{1}\Contr{L- L_1},\end{equation} we immediately obtain that $\Sepi{1}\in \TL$, which concludes the claim.

\smallskip

To prove that \eqref{eq:1} holds, note first that every edge of $M$ is contracted in $G\Contr M$ so in particular it has its endpoints in exactly one of the three sets $(Y'_1)\Expand{M}, (S'_1)\Expand{M}$ and $(Z'_1)\Expand{M}$. In particular by definition of $L_1$, the edges of $L_1$ are all disjoint from $(Y'_1)\Expand{M}$ and thus $(Y'_1)\Expand{M}=(Y'_1)\Expand{\overline{L_1}}$. This implies that $Y_1=(Y'_1)\Expand{\overline{L}}=((Y'_1)\Expand{\overline{L_1}})\Contr{L- L_1}=(Y''_1)\Contr{L-L_1}$. As $S'_1$ is disjoint from $Y'_1$ in $G\Contr{M}$, it is also disjoint from $Y_1$ in $G\Contr{L}$ so we have $Y_1=(Y''_1)\Contr{(L\setminus L_1)}\setminus S''_1$. Symmetric arguments give $Z_1=(Z''_1)\Contr{L- L_1}\setminus S''_1$, and as $S''_1=S'_1=S_1$, we get the desired equality.
  \end{proofofclaim}
  
  By Theorem \ref{thm:contract tangle1}, $\TL$ is a region tangle of order $4$ so Claim \ref{clm: tangle-expand} implies that either $Z_1 \cap Z_2 \cap Z_3 \neq \emptyset$ or there exists an edge of $\GL$ with both endpoints in $Z_1\cup Z_2\cup Z_3$.
  If
  $Z_1 \cap Z_2 \cap Z_3 \neq \emptyset$, then
  $Z'_1 \cap Z'_2 \cap Z'_3 = (Z_1 \cap Z_2 \cap Z_3)\Contr{\overline L}
  \neq \emptyset$. Otherwise, there is an edge $e \in E(G\Contr{L})$ that has an
  endpoint in each $Z_i$, in which case, either
  $Z'_1 \cap Z'_2 \cap Z'_3 \neq \emptyset$ if $e \in \overline L$, or
  $e\Contr{\overline L}$ is an edge of $G\Contr{M}$ which has an endpoint in each
  $Z'_i$. This proves~\ref{tangle: it2} 
  %
  and shows that $\T\Contr{M}$ is a tangle of order 4.  

  \smallskip
  
  We now prove the inclusion $\sg{\Sep\Contr{M}: \Sep \in \T}\subseteq \T\Contr{M}$. Let $\Sep \in \T$ and $L:=M(S)$. By Theorem \ref{thm:contract tangle1} $(4)$, $\Sep\Contr{L}\in \TL$. Write $\Sepp=\Sep\Contr{M}$ and note that $(Y'\Expand{\overline{L}}, S', Z'\Expand{\overline{L}})=\Sep\Contr{L}$, 
  thus by definition of $\T\Contr{M}$, $\Sep\Contr{M}\in \T\Contr{M}$.

  \smallskip

  We now prove that $\T\Contr{M}$ is a well-founded set. For the sake of
  contradiction, let $(\Seppi{n})_{n \in \NN}$ be an infinite decreasing
  sequence of separations of $\T\Contr{M}$. The contradiction will follow from the next claim
  \begin{claim}
   \label{clm: tangle-region}
   There exists an infinite decreasing sequence $(\Sepppi{n})_{n \in \NN}$ in $\T\Contr{M}$ such that for each $n\geq 0$, $\Sepppi{n}=\Sepi{n}\Contr{M}$ for some $\Sepi{n}\in \T$. 
  \end{claim}
  
  \begin{proofofclaim}
  By \cref{lem: TWcut} and because $G\Contr M$ is $3$-connected, for any
  $n \in \NN$, there are finitely many separations $\Seppi{m}$ such that
  $S'_n \cap S'_m \neq \emptyset$. Therefore, up to considering a subsequence,
  we can assume that for all $n$, $S'_n \subseteq Y'_{n+1}$ and
  $S'_{n+1} \subseteq Z'_{n}$. In particular, as by Lemma \ref{lem:contract connected} $G\Contr{M}$ is $3$-connected, $S'_{n+1}$ is included in some connected component $C_n$ of $G\Contr{M}[Z'_n]$. Then we have
  $$\Seppi{n+1}\preceq (V(G\Contr{M})\setminus(S'_n\cup C_n), S'_n, C_n)\preceq \Seppi{n},$$
  implying that $(V(G\Contr{M})\setminus(S'_n\cup C_n), S'_n, C_n)\in \T\Contr{M}$. Hence we may also assume up to replacing $\Seppi{n}$ with $(V(G\Contr{M})\setminus(S'_n\cup C_n), S'_n, C_n)$ that for each $n\geq 0$, $G\Contr{M}[Z'_n]$ is connected.

  For each $n\geq 0$, we let $L_n:=M(S'_n)$. Connectedness of $G\Contr{M}[Z'_n]$ then implies that $G\Contr{L_n}[(Z'_n)\Expand{\overline{L_n}}]$ is connected. Observe that $|L_n|\leq 3$ successive applications of Lemma \ref{lem: essential-sep} imply that there exists some separation $\Sepi{n}\in \T$ such that $S_n\Contr{L_n}=S'_n$ and $Z_n\setminus ((S'_n)\Expand{L_n})=(Z'_n)\Expand{\overline{L_n}}$. We let $\Sepppi{n}:=\Sepi{n}\Contr{M}$. By Theorem \ref{thm:contract tangle1} (4), $\Sepi{n}\Contr{L_n}\in \T\Contr{L_n}$. Moreover, \[\Sepi{n}\Contr{L_n}=((Y''_n)\Expand{\overline{L_n}}, S''_n, (Z''_n)\Expand{\overline{L_n}}),\]
  thus by definition of $\T\Contr{M}$, we have $\Sepppi{n}=(\Sepi{n}\Contr{L_n})\Contr{\overline{L_n}}\in \T\Contr{M}$.

  Note that $S''_n=S_n\Contr{M}=S_n\Contr{L_n}=S'_n$ and 
  $$Z''_n=Z_n\Contr{M}\setminus S'_n=(Z_n\setminus ((S'_n)\Expand{L_n}))\Contr{\overline{L_n}}=((Z'_n)\Expand{\overline{L_n}})\Contr{\overline{L_n}}=Z'_n.$$
  We thus deduce that $\Seppi{n}=\Sepppi{n}$, so $(\Seppi{n})_{n\in \mathbb N}$ is an infinite decreasing sequence in $\T\Contr{M}$ satisfying the desired properties.
\end{proofofclaim}
  
It remains to show how to derive a contradiction from Claim \ref{clm: tangle-region}. For this let $(\Sepppi{n})_{n\in \mathbb N}$ and $(\Sepi{n})_{n\in \mathbb N}$ be as in Claim \ref{clm: tangle-region}. Again, up to considering a subsequence,  we can assume that for all $n$, $S''_n \subseteq Y''_{n+1}$ and $S''_{n+1} \subseteq Z''_{n}$. As
$S''_n \cap S''_{n+1} = \emptyset$, the separators $S_n$ and $S_{n+1}$ cannot
contain two vertices of a common crossedge of $M$. Thus,
$(S''_n)\Expand{M} \subseteq Y_{n+1}$ and $(S''_{n+1})\Expand{M} \subseteq Z_n$
and hence $\Sepi{n+1} \prec \Sepi{n}$. This proves that
$(\Sepi{n})_{n \in \NN}$ is an infinite decreasing sequence of separations of
$G$ with respect to $\T$, contradicting the fact that $\T$ is a region tangle.
\end{proof}

\begin{lemma}
  Either $\GM$ is 4-connected and $\T\Contr{M}_{\mathrm{min}} = \{
  (\emptyset,\emptyset,V(G)\Contr{M})\}$ or 
   $$\T\Contr{M}_{\mathrm{min}} = \{\Sep\Contr{M} :\Sep \in \T_{\mathrm{min}} \text{ such that
   } S\Contr{M} \text{ is a separator of } G\Contr{M}\}.$$

   Finally, we have $\Exnd{G\Contr{M}} = \emptyset$.
\end{lemma}
\begin{proof}

  Assume that $G\Contr{M}$ is not 4-connected. We first prove the direct
  inclusion. Let $\Sepp \in \T\Contr{M}_{\mathrm{min}}$, $L := M(S')$ and 
  $\Sepi{0} := (Y'\Expand{\overline L}, S', Z'\Expand{\overline L})$. Then by definition of $\T\Contr{M}$, $\Sepi{0} \in \T\Contr{L}$. We prove that $\Sepi{0}$ is a minimal
  element of $\T\Contr{L}$. Assume for a contradiction that there exists
  $\Sepi{1} \prec \Sepi{0}$ in $\T\Contr{L}$. Then
  $\Sepi{1}\Contr{\overline L} \preceq \Sepi{0}\Contr{\overline L}=\Sepp$ and
  $\Sepi{1}\Contr{\overline L} \neq \Sepp$ because
  $\overline L \cap S' = \emptyset$. Moreover, by Theorem \ref{thm:contract tangle1} $(2)$, we may assume that $\Seppi{1}$ is minimal and that $\Sepi{1}=\Seppi{1}\Contr{L}$ for some $\Seppi{1}\in \T$.
  Thus as 
  $\Sepi{1}\Contr{\overline L}=(\Seppi{1}\Contr{L})\Contr{\overline{L}}=\Seppi{1}\Contr{M}$, we must have
  $\Sepi{1}\in \T\Contr{M}$, contradicting the minimality of $\Sepi{0}$
  in $\T\Contr{L}$. Hence $\Sepp = \Sepi{0}\Contr{\overline{L}}$ for some
  $\Sepi{0} \in \T\Contr{L}_{\mathrm{min}}$. Again we can apply \cref{thm:contract tangle1} and write $\Sepi{0}=\Sep\Contr{L}$ for some
  $\Sep \in \T_{\mathrm{min}}$. Thus we have:
  $$\Sepp =  (\Sep\Contr{L})\Contr{\overline{L}} = \Sep\Contr{M},$$
  so we are done with the direct inclusion.
  
  Conversely, let $\Sepi{1} \in \T_{\mathrm{min}}$ such that $S_1\Contr{M}$ is a
  separator of $G\Contr{M}$. Note that by previous inclusion and because $\T\Contr{M}$ is a region tangle, it is enough to prove that for any $\Sepi{2}\in \T$ such that $\Sepi{2}\Contr{M}\preceq\Sepi{1}\Contr{M}$, we have $\Sepi{2}\Contr{M}=\Sepi{1}\Contr{M}$.
  Let $\Sepi{2} \in \T$ such that $\Sepi{2}\Contr{M}
  \preceq \Sepi{1}\Contr{M}$, $L := M(S_1\cup S_2)$ and $\overline L
  := M \setminus L$. For $i \in \{1,2\}$, the edges in $\overline{L}$ are either
  contained in $Y_i$ or in $Z_i$. Note that for each $i$ and $L \subseteq {M}$: 
  $$S_i\Contr{L}\cup (Z_i\Contr{L}\setminus S_i\Contr{L}) = (Z_i\cup S_i)\Contr{L}$$
  and
  $$S_i\Contr{L}\cup (Y_i\Contr{L}\setminus S_i\Contr{L}) = (Y_i\cup S_i)\Contr{L}.$$
  
  Thus as $\Sepi{1} \Contr{M} \preceq
  \Sepi{2}\Contr{M}$, we have $(S_1\cup Z_1)\Contr{M}\subseteq (S_2\cup Z_2)\Contr{M}$. By previous remark that no edge of $\overline{L}$ is contained in $Z_i$, this implies that
  $(S_1 \cup Z_1)\Contr{L} \subseteq (S_2 \cup
  Z_2)\Contr{L}$. Likewise $(S_1 \cup Y_1)\Contr{L} \subseteq (S_2 \cup
  Y_2)\Contr{L}$. As a result, $\Sepi{1} \Contr{L} \preceq
  \Sepi{2}\Contr{L}$. By Theorem \ref{thm:contract tangle1} $(2)$, $\Sepi{1}\Contr{L} \in \T\Contr{L}_{\mathrm{min}}$, thus we have
  $\Sepi{1}\Contr{L} = \Sepi{2}\Contr{L}$ and $\Sepi{1}\Contr{M} =
  \Sepi{2}\Contr{M}$, showing that $\Sepi{1}\Contr{M}\in
  \T\Contr{M}_{\mathrm{min}}$.

  We now prove that $\Exnd{G\Contr{M}} = \emptyset$. Assume that there are two
  crossing non-degenerate minimal 3-separations $\Sepppi{1}$ and $\Sepppi{2}$ in
  $\T\Contr{M}$, let $L = M(S''_1 \cup S''_2)$. For $i \in \{1,2\}$, all
  crossedges of $\GL$ lie in $Y''_i$ or in $Z''_i$, hence
  $\Seppi{i} = ((Y''_i)\Expand{\overline{L}}, S''_i, (Z''_i)\Expand{\overline{L}})$ is
  the only 3-separation of $\separ_{<4}(\GL)$ such that
  $\Seppi{i}\Contr{\overline{L}} = \Sepppi{i}$. Since
  $\Sepppi{i} \in \T_{\mathrm{min}}\Contr{M}$, from what we just proved, we must
  have $\Seppi{i} \in \T_{\mathrm{min}}\Contr{M}$. Note that the separations
  $\Seppi{i}$ are non-degenerate in $\GL$. Furthermore, as $L = M(S''_1 \cup
  S''_2)$, $\Seppi{1}$ and $\Seppi{2}$ must be also crossing in $\GL$, but this
  contradicts $\Exnd{T\Contr{L}} = \Exnd \T \setminus L$ (third item
  of~\cref{thm:contract tangle1}).
\end{proof}

For each $L\subseteq M$, we let $\RL:= R_{\mathcal T}\Contr{L}$.
Note that for each $L\subseteq M$, $(R_{\TL})\Expand{L}= R_{\T}$.
Thus together with \cref{lem: ortho}, this immediately gives the following, which is the locally finite extension of one of the main results from \cite{Gro16}:

\begin{theorem}
\label{thm: GroRegion}
 Let $G$ be a locally finite $3$-connected graph, and let $\mathcal T$ be a region tangle  of order $4$ in $G$. Let $M:=\Exnd{\T}$ be the set of crossedges between non-degenerate minimal separations of $\mathcal T$. Then the graph $G^{\backslash M/}\torso{R^{\backslash M/}}$ is a quasi-$4$-connected minor of $G$.
\end{theorem}

In order to obtain a proof of \cref{thm: Grofini} in the locally finite case, one can either reuse the arguments from \cite[Section 5]{Gro16}, or equivalently adapt our proof from \cref{sec: proof}.

\subsection{Planarity after uncontracting crossedges}\label{sec: planar}

The difficulty is that in general, nice properties of $G\Contr{M}\torso{R\Contr{M}}$ are not satisfied anymore by $G\torso{R_{\T}}$. To circumvent this and find a quasi-$4$-connected region in $G$, it is proved in \cite{Gro16} that for every subset $X'$ of $R_{\T}$ obtained by deleting one endpoint of each edge of $M$, the graph $G\torso{X'}$ is isomorphic to $G\Contr{M}\torso{R\Contr{M}}$.
However we cannot choose such a subset $X'$ canonically in general, as illustrated in \cref{ex: gro}.
Despite the fact that uncontracting the edges of $M$ does not preserve the quasi-$4$-connectivity of torsos, we now prove that at least planarity is preserved by this operation.

\begin{proposition}
 \label{prop: planar}
 If $G\Contr{M}\torso{R\Contr{M}}$ is planar, then so is $G\torso{R_{\T}}$.
\end{proposition}

This is obtained by combining the following two lemmas:

\begin{lemma}
 \label{lem: planar1}
 For every subset $L\subseteq M$, we denote $\mean{L}:=M\setminus L$. Assume that for every finite subset $L\subseteq M$, $G\Contr{\mean{L}}\torso{R_{\T}\Contr{\mean{L}}}$ is planar.
 Then $G\torso{R_{\T}}$ is also planar.
\end{lemma}

\begin{proof}
Assume for the sake of contradiction that $G$ enjoys the properties described above but that $G\torso{R_{\T}}$ is not planar. Then by Wagner's theorem \cite{wagner37}, $G\torso{R_{\T}}$ admits $F$ as a minor, for some $F\in\sg{K_{5},K_{3,3}}$.
Note that we can find a model $(V_v)_{v\in V(F)}$ of $F$ such that each set
$V_v$ is finite. Then $X:=\bigcup_{v\in V(F)} V_v$ is a finite subset of $V(G)$ and as $G$ is locally finite, there are only finitely
many edges in $M\left(X\Contr{M}\right)$ (recall that for each subset
$X' \subseteq V(\GM)$,  $M(X')$ is the set of crossedges of $M$
  that contract to a vertex in $X'$).
We let $L:=M\left(X\Contr{M}\right)$ denote this finite set of  edges and note that the sets $V_v$ are also subsets of $V(G\Contr{L})$. It follows that $(V_v)_{v\in V(F)}$ is also a model of $F$ in $G\Contr{\mean{L}}\torso{R\Contr{\mean{L}}}$, a contradiction.
\end{proof}

\begin{lemma}[Planar contraction of a single crossedge]
\label{lem: planar2}
  Let $G$ be locally finite and 3-connected, and $\T$ be a region tangle of order 4 in $G$. Let
  $\Sepi{1}$ and $\Sepi{2}$ be two minimal non-degenerate crossing separations
  of $\T$. Let $s_1s_2$ be the corresponding crossedge and $G'$ be the graph obtained from $G$ after contracting
  $s_1s_2$. Let $\Seppi{i}:=\Sepi{i}^{\vee}$ be the projection of $\Sepi{i}$ to $G'$ for each $i\in\sg{1,2}$.
  Let $R := R_{\T} \subseteq V(G)$ and $R' := R^\vee$. If $G'\llbracket R' \rrbracket$ is
  planar, then so is $G\llbracket R \rrbracket$. 
\end{lemma}

\begin{proof}
  We let $H := G\llbracket R\rrbracket$ and $H' := G'\llbracket R'
  \rrbracket$ and for $i \in \{1,2\}$, we write $S_{i} = \{s_i, t_i, r_i\}$ such that $s_1s_2$ is the crossedge between $\Sepi{1}$ and $\Sepi{2}$. Since
  $\Sepi{1}$ and $\Sepi{2}$ are crossing, the edge $s_1s_2$ belongs to $E(G[R])\subseteq E(H)$. Note that in particular we have $6= |S_1\cup S_2| \leq |R|$.

  \begin{claim}
  \label{clm: s1}
    The neighborhood of $s_1$ in $H$ is:
    $$N_H(s_1)=\sg{s_2}\cup(\fc(S_2)\setminus \sg{s_1}),$$
    and $\fc(S_2)$ is a triangle in $H$.
  \end{claim}
  \begin{proofofclaim}
  Note that by definition of the torso, the projection $(Y_i\cap R, S_i \cap R, Z_i \cap R)$ of $\Sepi{i}$ to $H$ is a separation of $H$. Hence the only possible neighbors of $s_1$ in $H$ must lie in $(R\cap Z_1\cap Y_2)\cup\sg{t_2, r_2}$ (see \cref{fig:crossing}). 
  
Note that as $G$ is $3$-connected, $H$ must also
be $3$-connected: 
this comes from the fact that $|V(H)|\geq 6$ and from the observation that any
separator $S\subseteq R=V(H)$ of $H$ is also a separator of $G$. Thus in particular every vertex of $H$ has degree at
least $3$.
  
  Then, as $|\sg{s_2}\cup(\fc(S_2)\setminus \sg{s_1})|=3$, it is enough to prove that $N_H(s_1)\subseteq \sg{s_2}\cup(\fc(S_2)\setminus \sg{s_1})$ as equality will be immediately implied as $d_H(s_1)\geq 3$. For this we let $t\in N_H(s_1)\setminus \sg{s_2}$. We distinguish two cases:
  \begin{itemize}
   \item If $t\in S_2$, then without loss of generality let $t=t_2$. First note
     that if $t$ is not an endpoint of some crossedge then $t\in S_2\cap
     \fc(S_2)$ and there is nothing to prove (see \cref{fig: planar}). Thus we assume that there exists a crossedge $t_2s_3$ incident to $t_2$ for some $s_3$ and we prove that this case implies a contradiction, which will imply the desired inclusion. As $t_2\neq s_2$ and $\Exnd{\T}$ is a matching, we have $s_3\neq s_1$ and there exists $\Sepi{3}\in \T_\mathrm{nd}$ that crosses $\Sepi{2}$ via the crossedge $t_2s_3$. As $\Sepi{1}$ and $\Sepi{2}$ cross, we have $s_1\in Y_2$. As $\Sepi{2}$ and $\Sepi{3}$ cross, we have $t_2\in Y_3$ and $S_3\setminus \sg{s_3}\subseteq Z_2$. As we assumed that $s_1t_2\in E(H)$, $s_1\neq s_3$ and as $t_2\in Y_3$, we must have $s_1\in Y_3\cup (S_3\setminus \sg{s_3})$. This implies a contradiction as $Y_2\cap Y_3=\emptyset$ and $Y_2\cap (S_3\setminus \sg{s_3})=\emptyset$. 
   
   \item If $t\notin S_2$, then we must have: $t\in R\cap Y_2\cap Z_1$ and by
     definition of $R$, as $t\notin \bigcap_{\Sep\in \Tnd}Z$, this means that $t\in S_3$ for some $\Sepi{3}\in \Tnd \setminus \sg{\Sepi{1}, \Sepi{2}}$. By \cref{lem: cross}, $\Sepi{2}$ and $\Sepi{3}$ are either orthogonal or crossing. If we were in the former case, then we should have $S_3\cap Y_2=\emptyset$, which is impossible as $t\in S_3\cap Y_2$. Hence $\Sepi{2}$ and $\Sepi{3}$ are crossing, and if $s_3$ denotes the endpoint of the crossedge between $S_2$ and $S_3$, as $S_3\cap Y_2=\sg{s_3}$ we must have $t=s_3$ so we are done as $s_3\in \fc(S_2)\setminus\sg{s_1}$. 
  \end{itemize}

  The fact that $\fc(S_2)$ forms a clique follows from \cite[Lemma 4.33]{Gro16}.
  \end{proofofclaim}

Note that by symmetry, \cref{clm: s1} also implies that we have $N_H(s_2)=\sg{s_1}\cup(\fc(S_1)\setminus \sg{s_2})$ and that $\fc(S_1)$ is a clique.
  
  \begin{figure}[htb]
  \centering
    \begin{tikzpicture}[scale=0.7]
      \begin{scope}
        
        \draw[draw=none,fill=lightgray] (0,1.5) rectangle (7,5);
        \draw[draw=none,fill=lightgray] (1.5,1.5) rectangle (7,3);
        \draw[draw=none,fill=lightgray] (1.5,0) rectangle (7,1.5);
        \draw[draw=none,fill=white] (1.5,1.5) rectangle (2.7,2.5);        
        \draw[thick] (0,0) rectangle (7,5);
        \draw (1.5,0) rectangle (2.7,5);
        \draw (0,1.5) rectangle (7,2.5);
        \draw (4.5,0) rectangle (5.7,5);

        \node[draw=black, circle, fill=black,inner sep = 2pt] (s11) at (.75,2) {};
        \node[draw=black, circle, fill=black,inner sep = 2pt] (s12) at (3.5,2) {};
        \node[draw=black, circle, fill=black,inner sep = 2pt] (s13) at (6.25,2)
        {};
        \node[draw=black, circle, fill=black,inner sep = 2pt] (s21) at (2,.75) {};
        \node[draw=black, circle, fill=black,inner sep = 2pt] (s22) at (2,3.18) {};
        \node[draw=black, circle, fill=black,inner sep = 2pt] (s23) at
        (2,4.33) {};
        \node[draw=black, circle, fill=black,inner sep = 2pt] (s3) at
        (5,.75) {};
        
        \node[right=-.1cm of s11] {$s_2$};
        \node[right=-.1cm of s12] {$t_2$};
        \node[right=-.1cm of s13] {$r_2$};
        \node[below=-.1cm of s21] {$s_1$};
        \node[right=-.1cm of s22] {$t_1$};
        \node[right=-.1cm of s23] {$r_1$};
        \node[right=-.1cm of s3] {$s_3$};

        \path (s21) edge[bend right = 20, dashed] (s12);
        \path (s21) edge[bend right = 20, dashed] (s3);
        \path (s12) edge[bend right = 20, dashed] (s3);

        \draw (s11) -- (s21);
        \draw (s13) -- (s3);
        \node (Y1) at (.75,-.5) {$Y_1$};
        \node (S1) at (2,-.5) {$S_1$}; 
        \draw[decorate,decoration={brace,mirror}] (2.7,-.05) -- (7,-.05); 
        \node (Z1) at (4.75,-.5) {$Z_1$}; 
        \node (Y2) at (-.5,.75) {$Y_2$}; 
        \node (S2) at (-.5,2) {$S_2$}; 
        \node (Z2) at (-.5,3.5) {$Z_2$};
        \node (22) at (5.1,5.4) {$S_3$};
    \end{scope}

    \begin{scope}[xshift=10cm]
         
        \draw[draw=none,fill=lightgray] (0,2.5) rectangle (7,5);
        \draw[draw=none,fill=lightgray] (1.5,1.5) rectangle (7,3);
        \draw[draw=none,fill=lightgray] (2.7,0) rectangle (7,1.5);
        \draw[thick] (0,0) rectangle (7,5);
        \draw (1.5,0) rectangle (2.7,5);
        \draw (0,1.5) rectangle (7,2.5);
        \draw (4.5,0) rectangle (5.7,5);

        \node[draw=black, circle, fill=black,inner sep = 2pt] (s1) at (2,2) {};
        \node[draw=black, circle, fill=black,inner sep = 2pt] (s12) at (3.5,2) {};
        \node[draw=black, circle, fill=black,inner sep = 2pt] (s13) at (6.25,2) {};
        \node[draw=black, circle, fill=black,inner sep = 2pt] (s22) at (2,3.18) {};
        \node[draw=black, circle, fill=black,inner sep = 2pt] (s23) at
        (2,4.33) {};
        \node[draw=black, circle, fill=black,inner sep = 2pt] (s3) at
        (5,.75) {};
        
        \node[right=-.1cm of s1] {$s'$};
        \node[right=-.1cm of s12] {$t_2$};
        \node[right=-.1cm of s13] {$r_2$};
        \node[right=-.1cm of s22] {$t_1$};
        \node[right=-.1cm of s23] {$r_1$};
        \node[right=-.1cm of s3] {$s_3$};

        \path (s1) edge[bend right = 50, dashed] (s12);
        \path (s1) edge[bend right = 50, dashed] (s3);
        \path (s12) edge[bend right = 50, dashed] (s3);

        \draw (s13) -- (s3);
        \node (Y1) at (.75,-.5) {$Y'_1$};
        \node (S1) at (2,-.5) {$S'_1$};
        \draw[decorate,decoration={brace,mirror}] (2.7,-.05) -- (7,-.05); 
        \node (Z1) at (4.75,-.5) {$Z_1$}; 
        \node (Y2) at (-.5,.75) {$Y'_2$}; 
        \node (S2) at (-.5,2) {$S'_2$}; 
        \node (Z2) at (-.5,3.5) {$Z_2$};
        \node (22) at (5.1,5.4) {$S_3$};
 
    \end{scope}

      \end{tikzpicture}
      \caption{Left: The graph $G$ when $S_2$ is incident to exactly $2$ crossedges. Here $t_2$ is part of no crossedge and $S_2$
        and $S_3$ are crossing via the crossedge $r_2s_3$. Hence,
        $\fc(S_2) = \{s_1, t_2,s_3\}$.\\ Right: The graph $G'$ obtained after
        contracting the crossedge $s_1s_2$. The dashed edges are edges that
        appear in $H$ and $H'$ respectively. The situation is identical when $S_2$ is incident to $3$ crossedges, but harder to illustrate in $2$ dimensions.}
    \label{fig: planar}
  \end{figure}
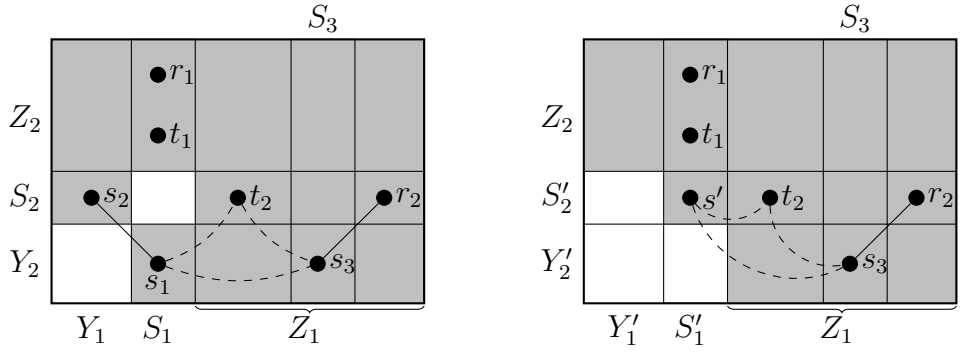

  Recall that by Wagner's theorem~\cite{wagner37}, a graph is planar if and only if it is
  $K_5$ and $K_{3,3}$-minor free. Hence, it is enough to prove that if $H$ contains $K_5$ or
  $K_{3,3}$ as a minor, then so does $H'$. We write $\fc(S_i)=\sg{s_{3-i}, u_i, v_i}$ for $i\in \sg{1,2}$, and we recall that  the vertex of $H'$ resulting from the contraction of $s_1$ and $s_2$ is denoted by $s'$.
  \begin{claim}
    If $H$ contains a $K_5$-minor, then so does $H'$.
  \end{claim}
  \begin{proofofclaim}
 Let $(V_1, \dots V_5)$
  be a model of $K_5$ in $H$. Let $V_1', \dots V_5'$ be the projection of the
  sets $V_i$ to $H'$. If $s_1$ and $s_2$ are in the same set $V_i$, then $(V_1', \dots V_5')$ is
  also a model of $K_5$ in $H'$, so we can assume that the vertices $s_1$ and $s_2$  belong to distinct sets
  $V_i$, say $s_1\in V_1$ and $s_2\in V_2$. As by \cref{clm: s1}, $s_1$ has degree 3 in $H$, we have $V_1 \neq
  \{s_1\}$, so $V_1$ contains one neighbor of $s_1$ distinct of $s_2$, say $u_2$. Since $u_2$ and $v_2$ are
  adjacent in $H$, the edge $u_2s'$ in $H'$ has an endpoint in $V_1'\setminus \{s'\}$ and an endpoint in $V'_2$. Moreover as $u_2v_2\in E(H')$, the set $V_1'\setminus \{s'\}$ is connected in $H'$. Thus $(V_1'\setminus \{s'\}, V_2', V_3', V_4', V_5')$
  is a model of $K_5$ in $H'$, as desired.
  \end{proofofclaim}
  \begin{claim}
    If $H$ contains a $K_{3,3}$-minor, then $H'$ contains a $K_{3,3}$-minor or a $K_5$-minor.
  \end{claim}
  \begin{proofofclaim}
  Let $(V_1, \dots V_6)$
  be a model of $K_{3,3}$ in $H$, such that $V_i$ is adjacent to $V_j$ if $i$
  and $j$ have different parities. Let $V_1', \dots V_6'$ be their projection to
  $H'$. If $s_1$ and $s_2$ are in the same set $V_i$, then $(V_1', \dots V_6')$ is
  also a model of $K_{3,3}$ in $H'$, so we can assume that the vertices $s_1$ and $s_2$  belong to distinct sets
  $V_i$, say $s_1\in V_1$ and $s_2\notin V_1$.

  If $u_2 \in V_1$, then the edges $s'u_2$ and $u_2v_2$ in $H'$ ensure that
  $V_1' \setminus \{s'\}$ remains connected and that $(V_1' \setminus \{s'\},
  V_2', \dots V_6')$ is a model of {$K_{3,3}$} in $H'$. Thus we can assume that $u_2 \notin V_1$ and
  similarly $v_2 \notin V_1$. Since $V_1$ is connected, we must have $V_1 =
  \{s_1\}$.

  As $s_1$ has degree three and $V_1$ is adjacent to $V_2$, $V_4$ and $V_6$,
  this implies that $s_2$, $u_2$ and $v_2$ must belong to different sets $V_{2i}$,
  say $s_2 \in V_2$, $u_2 \in V_4$ and $v_2 \in V_6$. By applying the same
  reasoning as for $s_2$, we obtain $V_2= \sg{s_2}$, $u_1 \in V_3$ and $v_1 \in V_5$. But then
  $(\{s'\}, V_3', V_4', V_5', V_6')$ is a model of $K_5$ in $H'$.
  \end{proofofclaim}
  
  This concludes the proof of Lemma \ref{lem: planar2}.
\end{proof}

\begin{proof}[Proof of \cref{prop: planar}]
 Assume that $G\Contr{M}\torso{R\Contr{M}}$ is planar and for every  $L\subseteq M$,  set $\mean{L}:=M\setminus L$. Then, using \cref{lem: planar2}, we can easily prove by induction
 on $|L|\in \mathbb{N}$ that for any \emph{finite} set $L\subseteq M$, $G\Contr{\mean{L}}\torso{R_{\T}\Contr{\mean{L}}}$
 is planar. In order to be able to use  induction, we also
 need to observe that for the contraction of a single crossedge, the equality
 $R_{\T'}= R_{\T}^{\vee}$ holds. This is proved in \cite[Section 4.5]{Gro16} and
 can be deduced from item (2) of \cref{thm:contract tangle1}. We thus
 conclude by \cref{lem: planar1} that $G\torso{R_{\T}}$ is also planar.
\end{proof}

\section{The structure of quasi-transitive graphs avoiding a minor}\label{sec: struct}

\subsection{Main results}

Our main result in this section is the following more precise version of \cref{intro:main}.
\begin{theorem}
 \label{thm: mainCTTD}
 Let $G$ be a locally finite graph excluding $\Kinf$ as a minor and let $\Gamma$ be a group with a quasi-transitive action on $G$. Then there is an integer $k$ such that $G$ admits a $\Gamma$-canonical tree-decomposition $(T,\mathcal V)$, with $\mathcal V=(V_t)_{t\in V(T)}$,  whose torsos $G\llbracket V_t\rrbracket$  either have size at most $k$ or are $\Gamma_t$-quasi-transitive $3$-connected planar minors of $G$.
 Moreover, the edge-separations of $(T,\mathcal V)$ are tight.
\end{theorem}

\begin{remark}\label{rem:4.2}
A natural question is whether we can bound the maximum size $k$ of a finite bag in  \cref{thm: mainCTTD} by a function of the forbidden minor, when $G$ excludes some finite minor instead of the countable clique $K_\infty$. By taking the free product of the cyclic groups $\mathbb{Z}_k$ and $\mathbb{Z}$, with their natural sets of generators, we obtain a 4-regular Cayley graph consisting of cycles of length $k$ arranged in a tree-like way. This graph has no $K_4$ minor, but in any canonical tree-decomposition, each cycle $C_k$ has to be entirely contained in a bag, and thus there is no bound on the size of a bag as a function of the forbidden minor in \cref{thm: mainCTTD}. We can replace $\mathbb{Z}_k$ in this construction by the toroidal grid $\mathbb{Z}_k\times \mathbb{Z}_k$, and obtain a Cayley graph with no $K_8$-minor, such that the bags in any (non-necessarily canonical) tree-decomposition of finitely bounded adhesion are arbitrarily large.
\end{remark}

We will also prove the following version of \cref{intro:main2} at the same time. 
\begin{theorem}
\label{thm: main2}
 Let $G$ be a locally finite graph excluding $\Kinf$ as a minor and let $\Gamma$ be a group with a quasi-transitive action on $G$. Then there is an integer $k$ such that $G$ admits a $\Gamma$-canonical tree-decomposition $(T,\mathcal V)$, with $\mathcal V=(V_t)_{t\in V(T)}$, of adhesion at most $3$, and whose torsos $G\torso{V_t}$ are $\Gamma_t$-quasi-transitive minors of $G$ which are either planar or have treewidth at most $k$. The edge-separations of $(T,\mathcal V)$ are all non-degenerate.
\end{theorem}

\begin{remark}
 \label{rem: one-end}
If we carefully consider the proof of \cref{thm: main2}, we can check that if $G$ has only one end (as in the example of \cref{fig: Tangles}), then the tree-decomposition we obtain has adhesion $3$ and consists of a star with one infinite bag associated to its central vertex $z_0$, and finite bags on its branches.
In particular, $G\torso{V_{z_0}}$ cannot have bounded treewidth, as otherwise it would have more than one end, hence it must be planar.
Thus, \cref{thm: main2} implies that every one-ended locally finite quasi-transitive graph that excludes a minor can be obtained from a one-ended quasi-transitive planar graph by attaching some finite graphs on it along separators of order at most $3$. 
\end{remark}

\subsection{Tools}

Our proof of \cref{thm: mainCTTD,thm: main2} mainly consists in an application of \cref{thm: GroRegion} together with the following  result of Thomassen:

\begin{theorem}[Theorem 4.1 in \cite{Thomassen92}]
\label{thm: Thomassenq4c}
Let $G$ be a locally finite, quasi-transitive, quasi-4-connected graph $G$. If $G$ has a thick end, then $G$ is either planar or admits the countable clique $K_{\infty}$ as a minor.
\end{theorem}

A direct consequence of \cref{thm: Thomassenq4c} is the following, which will be our base case in what follows.

\begin{corollary}\label{cor:quasiplanartw}
 Let $G$ be a quasi-transitive, quasi-4-connected, locally finite graph which excludes the countable clique $K_{\infty}$ as a minor.  Then $G$ is planar or has finite treewidth.
\end{corollary}

\begin{proof}
Assume that $G$ is non-planar. As $G$ is $K_{\infty}$-minor free, by \cref{thm: Thomassenq4c}, all its ends have have finite degree. Then by \cref{thm: tw-ends}, $G$ has finite treewidth.
\end{proof}

Thomassen proved that if a quasi-transitive graph has only one end, then this end must be thick \cite[Proposition 5.6]{Thomassen92}. We prove the following generalization, which might be of independent interest.

\begin{proposition}
\label{prop: thin-ends}
Let $k\ge 1$ be an integer, and let $G$ be a locally finite quasi-transitive graph. Then $G$ cannot have exactly one end of degree exactly $k$. 
\end{proposition}

\begin{proof}
Assume without loss of generality that $G$ is connected, since otherwise each component of $G$ is also quasi-transitive locally finite, and we can restrict ourselves to a single component containing an end of degree exactly $k$.
Let $\Gamma$ be a group acting quasi-transitively on $G$.
Assume that $G$ has an end $\omega$ of degree exactly $k$ for some integer $k\ge 1$. As explained in \cite[Section 4]{TW}, there exists an infinite sequence of sets $S_0, S_1, \ldots$ of size $k$ such that for each $i\geq 0$, $S_{i+1}$ belongs to the component $G_i$ of $G_{i-1}- S_i$ where $\omega$ lives (where we set $G_0:=G$), and such that there exist $k$ vertex-disjoint paths $P_{1,i},\ldots, P_{k,i}$ from the $k$ vertices of $S_i$ to the $k$ vertices of $S_{i+1}$. By concatenating these paths, we obtain $k$ vertex-disjoint rays in $G$ living in $\omega$. 
As $G$ is connected and locally finite, note that up to extracting a subsequence of $(S_i)_{i\geq 0}$, we may assume that the $k$ paths $P_{1,i},\ldots, P_{k,i}$ are in the same component of $G- (S_i\cup S_{i+1})$.
Hence if we set $\Sepi{i}:=(G- (G_i\cup S_i),S_i, G_i)$ for each $i\geq 1$, $\Sepi{i}$ is a tight separation such that for each $i\geq 1$, $\Sepi{i+1}\prec\Sepi{i}$. Hence by \cref{lem: TWcut}, as there are only finitely many $\Gamma$-orbits of tight separations of size $k$, there exist $i<j$ and $g\in \Gamma$ such that $\Sepi{i}\cdot g= \Sepi{j}$. Assume without loss of generality that $(i,j)=(0,1)$. Note that by definition of $\prec$, the action of $g$ preserves the order $\prec$, i.e.\ for each $\Sep\prec \Sepp$, we must have $\Sep \cdot g\prec \Sepp \cdot g$. 
We now consider the sequence of separations $\Seppi{i}_{i\geq 0}$ defined for each $i\geq 0$ by: 
$\Seppi{i}:=\Sepi{0}\cdot g^{i}$. Then the sequence $\Seppi{i}_{i\geq 0}$ is strictly decreasing according to $\prec$.
Recall that there exist $k$ vertex-disjoint paths from $S_0$ to $S_1$ that extend to $k$ disjoint rays belonging to $\omega$. Then for each $i\geq 0$, there exist $k$ vertex-disjoint paths from $S'_i$ to $S'_{i+1}$ such that their concatenations consists in $k$ vertex-disjoint rays that belong to some end $\omega'$ of degree exactly $k$ (the fact that the end has degree at most $k$ follows from the fact that all the sets $S_i'$ are separators of size $k$ in $G$). If $\omega'\neq \omega$ then we are done, so we assume that $\omega'=\omega$. Now, observe that the sequence $\Sepppi{i}_{i\geq 0}$ defined for each $i\geq 0$ by $\Sepppi{i}:=\Sepi{0}\cdot g^{-i}$ also satisfies that for each $i\geq 0$, there exists $k$ vertex-disjoint paths $P''_{j,i}:=P_{j,0}\cdot g^{-i}$ for $j\in [k]$ from $S''_{i+1}$ to $S''_i$. If we consider the $k$ vertex-disjoint rays obtained from the concatenation of the paths $P''_{j,k}$, these rays must belong to the same end $\omega''$ as for each $i$, the paths $P''_{j,k}$ are in the same component of $G-(S''_i \cup S''_{i+1})$. The end $\omega''$ must have degree exactly $k$ as each $\Sepppi{i}$ is a separation of order $k$.
Moreover the sequence $\Sepppi{i}_{i\geq 0}$ is strictly increasing according to $\prec$, hence $\omega$ and $\omega''$ cannot live in the same component of $G- S_0$. Thus we found an end $\omega''$ distinct from $\omega$ of degree $k$.
\end{proof}

\cref{prop: thin-ends} and its proof are reminiscent of Halin's classification of the different types of action an automorphism of a quasi-transitive locally finite graph $G$ can have  on the ends of $G$ \cite[Theorem 9]{Halin73}. However it is not clear for us whether \cref{prop: thin-ends} can be seen as an immediate corollary of Halin's work.

\subsection{Proof of \cref{thm: mainCTTD,thm: main2}}
\label{sec: proof}
Let $G$ be a locally-finite
quasi-transitive graph excluding $\Kinf$ as a minor and let $\Gamma$ be a group inducing a quasi-transitive action on $G$. Let $(T,\mathcal V)$,  with $\mathcal V=(V_t)_{t\in V(T)}$, be a $\Gamma$-canonical tree-decomposition  of adhesion at most 2 obtained by applying \cref{thm: Tutte} to $G$. By \cref{lem: quasitrans}, for each $t\in V_t$, $\Gamma_t$ acts quasi-transitively on $G_t:=G\llbracket V_t\rrbracket$. Moreover, as $G_t$ is a minor of $G$, it must also exclude $\Kinf$ as a minor.

We let $t_1,\ldots, t_m$ be representatives of the orbits of $V(T)/\Gamma$. 
For each finite torso $G_{t_i}$ of $(T, \mathcal V)$, we define $(\widetilde{T}_{t_i},\widetilde{\mathcal{V}}_{t_i})$ as the trivial tree-decomposition of $G_{t_i}$ (in which the tree $\widetilde{T}_{t_i}$ contains a single node). For each infinite, $3$-connected torso $G_{t_i}$ of $(T, \mathcal V)$, we let $(\widetilde{T}_{t_i}, \mathcal{\widetilde{V}}_{t_i})$ be a 
$\Gamma_{t_i}$-canonical tree-decomposition of $G_{t_i}$ obtained by applying \cref{thm: tree-tangles} to $G_{t_i}$, i.e.\ $(\widetilde{T}_{t_i}, \mathcal{\widetilde{V}}_{t_i})$
distinguishes efficiently all the tangles of $G_{t_i}$ of order $4$. By \cref{rem: relevant}, the edge-separations of $(\widetilde{T}_{t_i}, \mathcal{\widetilde{V}}_{t_i})$ in $G_{t_i}$ are all distinct. By \cref{rem: degenerate-tangle}, the edge-separations of $(\widetilde{T}_{t_i}, \mathcal{\widetilde{V}}_{t_i})$ in $G_{t_i}$ are non-degenerate. 
Hence by \cref{lem: torso-minor},  the torsos of $(\widetilde{T}_{t_i},\mathcal{\widetilde{V}}_{t_i})$ are minors of $G_{t_i}$.  We now use \cref{cor: refinement} and find a refinement $(T_1,\mathcal V_1)$ of $(T,\mathcal V)$ with respect to some family $(T_t,\mathcal V_t)_{t\in V(T)}$ of $\Gamma_t$-canonical tree-decompositions of $G_t$
such that the construction $t\mapsto (T_t,\mathcal V_t)_{t\in V(T)}$ is $\Gamma$-canonical and  such that for each $i\in I$, $(T_{t_i}, \mathcal V_{t_i})$ is a  subdivision of $(\widetilde{T}_{t_i}, \mathcal{\widetilde{V}}_{t_i})$. 
Since the construction $t\mapsto (T_t, \mathcal V_t)_{t\in V(T)}$ is $\Gamma$-canonical, for each $t\in V(T_1)$ the decomposition $(T_t, \mathcal V_t)$ is $\Gamma_t$-canonical and efficiently distinguishes  the tangles of order $4$ of $G_t$ (by a slight abuse of notation, we keep denoting by $G_t$ the torso of the tree-decomposition $(T_1,\mathcal V_1)$ associated to the node $t\in V(T_1)$).
Note that by construction, the adhesion sets of $(T_1,\mathcal V_1)$ have size at most $3$ and all the edge-separations are tight. Moreover, the torsos of each tree-decomposition $(T_t, \mathcal V_t)$ are minors of $G_t$ for each $t\in V(T)$, and as the torsos of $(T,\mathcal V)$ are minors of $G$, we also have that the torsos of $(T_1,\mathcal V_1)$ are minors of $G$. In particular, they also exclude $\Kinf$ as a minor. Moreover, by \cref{lem: quasitrans}, for each $t\in V(T_1)$, $\Gamma_t$ acts quasi-transitively on  $G_{t}$. By \cref{lem: TWcut}, since all edge-separations of $(T_1,\mathcal V_1)$ are tight and have order at most 3, the graph $G_t$ is locally finite for each $t\in V(T_1)$.

\begin{claim}\label{clth}For each $t\in V(T_1)$ such that $G_t$ is infinite, $G_t$ is 3-connected and has a unique tangle $\mathcal{T}_t$ of order 4. Moreover $\mathcal{T}_t$ is a $\Gamma_t$-invariant region tangle and every end of $G_t$ has degree at least 4.
\end{claim}

\begin{proofofclaim} Consider a node $t\in V(T_1)$ such that $G_t$ is infinite. As all torsos are cycles, subgraphs of complete graphs of size at most 3, or 3-connected, $G_t$ itself is 3-connected. Since $G_t$ is connected and infinite, it contains some end $\omega$. Let $\mathcal T_{t}:=\sg{\Sep, |S|\leq 3 \text{ and } \omega \text{ lives in }Z}$ be defined in $G_t$. Note that $\mathcal T_{t}$ is a tangle of order 4 in $G_t$. As $G_t$ is a minor of $G$, by \cref{lem: RS} every tangle $\T'$ of order $4$ in $G_t$ induces a tangle $\T$ of order $4$ in $G$, and by \cref{rem: proj-inj} this mapping is injective. Moreover, note that if $\Sep$ is an edge-separation of $(T_1,\mathcal V_1)$ such that $V_t\subseteq Z\cup S$, then if $\mathcal M$ is any faithful model of $G_t$ in $G$, the projection $\Sepp:= \pi_{\mathcal M}\Sep$ is such that $Y'=\emptyset$. Thus $\Sepp\in \T'$, hence $\Sep\in \T$. This means that every edge-separation of $(T_1,\mathcal V_1)$ is oriented toward $t$ by $\T$. Hence if $G_t$ admits two distinct tangles $\T'_1, \T'_2$ of order $4$, the two associated tangles $\T_1, \T_2$ given by \cref{lem: RS} must be distinct and not distinguished by $(T_1, \mathcal V_1)$, a contradiction. This proves the existence and uniqueness of a tangle $\T_t$ of order $4$ in $G_t$.

Note that as $\Gamma_t$ acts on $G_t$ and $\mathcal T_{t}$ is the unique tangle of order 4 in $G_t$, the tangle $\mathcal T_{t}$ is $\Gamma_t$-invariant (as a family of separations). 

We can also observe that if the end $\omega$ in $G_t$ has degree at most 3, then by \cref{prop: thin-ends}, $G_t$ has another end $\omega'$ of degree at most 3 and the construction of $\mathcal T_{t}$ using the end $\omega'$ instead of $\omega$ yields a different tangle of order 4, which contradicts the uniqueness of $\mathcal T_{t}$. So every end of $G_t$ has degree at least 4.

It remains to prove that $\mathcal T_{t}$ is a region tangle. If not we can find an infinite decreasing sequence of separations of order 3 in $G_t$, and this sequence defines an end of degree 3 in $G_t$, which contradicts the fact that every end of $G_t$ has degree at least 4.
\end{proofofclaim}

We will need to decompose further the infinite torsos of the tree-decomposition $(T_1,\mathcal{V}_1)$.
Let $t\in V(T_1)$ be such that $G_t$ is infinite, and  let $\mathcal T_t$ be the region tangle of order $4$ in $G_t$ given by \cref{clth}.  We let $M_t:=\Exnd{\T_t}$ denote the set of crossedges  of $\mathcal T_t$, $(T'_t,\mathcal V'_t)$ be the $\Gamma_t$-canonical tree-decomposition of $G_t$ given by \cref{lem: cross-star}, and $z_0\in V(T_t)$ be the center of the star $T'_t$. By \cref{lem: quasitrans,lem: Grominor}, the graph $H:=G_t\torso{V'_{z_0}}$ is a $\Gamma_t$-quasi-transitive faithful minor of $G_t$, thus it must also exclude $\Kinf$ as a minor. 

Now we observe that $\Gamma_t$ induces a quasi-transitive group action on $H\Contr{M_t}$: for each $w\in V(H\Contr{M_t})$ and every $\gamma \in \Gamma_t$, we set:
$$ w \cdot \gamma:= \begin{cases}
	s_{u \cdot \gamma, v\cdot \gamma}&\text{if $w=s_{u,v}$, for some $\sg{u,v}\in M_t$, and}\\
	w\cdot \gamma &\text{otherwise},\\
\end{cases}$$
where we recall that the notation $s_{u,v}$, for $\sg{u,v}\in M_t$, is
introduced at the beginning of \cref{sec: all}. As $M_t$ is $\Gamma_t$-invariant, we easily see that the mapping $\gamma$ defines a bijection over $V(H\Contr{M_t})$. We let the reader check that it gives a graph isomorphism of $H\Contr{M_t}$. Note that the number of $\Gamma_t$-orbits of $V(H\Contr{M_t})$ is at most the number of $\Gamma_t$-orbits of $V(H)$, hence it must be finite.

\smallskip

As $H\Contr{M_t}$ is a minor of $H$, it also excludes the countable clique $K_\infty$ as a minor. It follows from \cref{thm: GroRegion}  that $H\Contr{M_t}$ is quasi-$4$-connected.
Hence, by \cref{cor:quasiplanartw}, $H\Contr{M_t}$  either has finite treewidth or it is planar. 
It is not hard to observe that the treewidth of $H$ is at most twice the treewidth of $H\Contr{M_t}$ so in particular if we are in the first case, $H$ has also bounded treewidth.
In the second case, \cref{prop: planar} implies that $H$ is also planar. In both cases, we obtain that $(T_t', \mathcal V_t')$ is a $\Gamma_t$-canonical tree-decomposition of $G_t$ with non-degenerate edge-separations, adhesion $3$ and where each torso is a minor of $G_t$ and has either bounded treewidth or is planar.
Eventually we can use \cref{prop: refinement} together with \cref{lem: canonical} as we did before to find a tree-decomposition $(T^{*}, \mathcal V^{*})$ of $G$ with the properties of \cref{thm: main2}.

We now explain how to derive \cref{thm: mainCTTD}: every torso $G\torso{V_t}$ of $(T^{*}, \mathcal V^{*})$ which is neither finite nor planar must have bounded treewidth, hence by \cref{thm: tw-ends} it must admit a $\Gamma_t$-canonical tree-decomposition where each torso has bounded width. Exactly as before we can apply  \cref{cor: refinement} to find a refinement of $(T^{*}, \mathcal V^{*})$ with the properties of \cref{thm: mainCTTD}.
\qed

\section{Applications}\label{sec: appli}

\subsection{The Hadwiger number of quasi-transitive graphs}\label{sec: hadwiger}


We say that a graph $H$ is \emph{singly-crossing} if $H$ can be embedded in the plane with a single edge-crossing.
    It was observed by Paul Seymour that Theorem \ref{thm: main2} bears striking similarities with a structure theorem of Robertson and Seymour \cite{RS93-1} related to the exclusion of a singly-crossing graph as a minor. Their theorem states that if $H$ is singly-crossing, then there is a constant $k_H$ such that any graph excluding $H$ as a minor has a tree-decomposition with adhesion at most 3 in which all torsos are planar or have treewidth at most $k_H$. On the other hand, for any integer $k$ there is a finite singly-crossing graph $H_k$ such that any graph with a tree-decomposition with adhesion at most 3 in which all torsos are planar or have treewidth at most $k$ must exclude  $H_k$ as a minor (this can be seen by taking $H_k$ to be a 4-connected triangulation of a sufficiently large grid, and adding an edge between two non-adjacent vertices lying on incident faces).
Using this observation, the following strengthening of \cref{intro:had} is now an immediate  consequence of \cref{thm: main2}.

\begin{theorem}
\label{thm: had}
For every locally finite quasi-transitive graph $G$ avoiding the countable clique $K_\infty$ as a minor, there is a finite singly-crossing  graph $H$ such that $G$ is $H$-minor-free. In particular there is an integer $k$ such that $G$ is $K_k$-minor-free.
\end{theorem}


Note that in this application we have not used explicitly the property that the underlying tree-decomposition was canonical, but it is used implicitly in the sense that this is what garantees that the treewidth of the torsos is uniformly bounded in \cref{thm: main2}.

\subsection{Accessibility in quasi-transitive graphs avoiding a minor}\label{sec: vaccess}

We now prove \cref{intro:vacc}, which we restate here for convenience.

\begin{theorem}
\label{thm: access}
    Every locally finite quasi-transitive graph avoiding the countable clique $K_\infty$ as a minor is vertex-accessible.
\end{theorem}

\begin{proof}
    Let $G$ be a locally finite graph avoiding the countable clique $K_\infty$ as a minor, with a group $\Gamma$ acting quasi-transitively on $G$. Let $(T,\mathcal{V})$, with $\mathcal{V}=(V_t)_{t\in V(T)}$, be a $\Gamma$-canonical tree-decomposition of $G$ of adhesion at most 3 obtained by applying \cref{thm: main2} to $G$. In particular all torsos are quasi-transitive minors of $G$, and all non-planar torsos  have bounded treewidth. Observe that any two ends living in different parts of the tree-decomposition are separated  by the separator of size at most 3 of an edge-separation of the tree-decomposition. Consider any node $t\in V(T)$. If $G\torso{V_t}$ is planar then since it is quasi-transitive (by \cref{lem: quasitrans}) and locally finite, $G\torso{V_t}$ is vertex-accessible \cite[Theorem 3.8]{Dunwoody07} and thus there is an integer $k_t$ such that all pairs of ends lying in $G\torso{V_t}$  can be separated by a set of at most $k_t$ vertices.  If $G\torso{V_t}$ has bounded treewidth then by \cref{thm: tw-ends} there is a integer $k_t$ such that all ends of $G\torso{V_t}$  have degree at most $k_t$, and thus all pairs of ends lying in $G\torso{V_t}$ can be separated by a set of at most $k_t$ vertices. As $V(T)/\Gamma$ is finite; there is only a finite number of possible values for the integers $k_t$, $t\in V(T)$, and thus their maximum $k$ is well-defined. We have proved that every pair of ends in $G$ can be separated by a set of at most $\max\{k,3\}$ vertices, which concludes the proof.
\end{proof}

\subsection{Finite presentability of minor-excluded groups}\label{sec: access}

A \emph{walk} in a graph $G$ is a finite sequence of vertices $W=(v_1,\ldots, v_k)$ where for each $i\in [k-1]$, $v_i=v_{i+1}$ or $v_iv_{i+1}\in E(G)$. We call $W$ a \emph{closed walk} when $v_1= v_k$ or $v_1v_k \in E(G)$. We let $\mathcal W(G)$ denote the set of closed walks of $G$. If $W$ is a closed walk that contains a \emph{spur}, i.e.\ if there is some $i$ such that $v_{i-1}=v_{i+1}$, then we say that  $W'=(v_1,\ldots, v_{i-1},v_{i+1},\ldots, v_k)$ is obtained from $W$ by  \emph{deleting the spur}. The inverse operation of \emph{adding a spur} consists in adding a neighbor of $v$ between $v_i$ and $v_{i+1}$ in the walk $W$ if $v_i=v_{i+1}=v$. Similarly, by \emph{deleting a repetition} we mean replacing $W$ by $(v_1,\ldots,v_{i-1},v_{i+1},\ldots, v_k)$ if $v_i=v_{i+1}$, and the inverse operation of \emph{adding a repetition} consists in replacing $W$ by $(v_1,\ldots,v_{i-1},v_i,v_i,v_{i+1},\ldots, v_k)$, for some $1\le i \le k$.
The \emph{rotation} of $W$ is the walk $(v_2,v_3\ldots, v_k,v_1)$, and
the \emph{reflection} of $W$ is the walk $(v_k,v_{k-1}\ldots,
v_2,v_1)$.

\smallskip

If $W=(v_1,\ldots, v_k)$ and $W'=(v'_1,\ldots,v'_\ell)$ are two walks such that $v_k=v'_1$, then their \emph{sum} is the walk $W\cdot W':=(v_1,\ldots,v_k=v'_1, \ldots, v'_\ell)$. We will say that a set of closed walks $\mathcal W$ \emph{generates} another set of closed walks $\mathcal{W}'$ if every element of $\mathcal W'$ can be obtained from elements of $\mathcal W$ by adding and deleting spurs and repetitions, and performing sums, reflections and rotations. 

\medskip

The following result was proved in \cite[Theorem 5.12]{HamannPlanar} when $G$ is a quasi-transitive locally finite planar graph. We reuse some of the arguments of the proof of \cite[Proposition 5.9]{HamannPlanar} and combine them with our structure theorem to extend the result to graphs excluding the countable clique $K_\infty$ as a minor.

\begin{theorem}
 \label{thm: finpres}
 Let $G$ be a locally finite graph excluding the countable clique $K_{\infty}$ as a minor and let $\Gamma$ be a group acting quasi-transitively on $G$. Then the set of closed walks of $G$ admits a $\Gamma$-invariant generating set with finitely many $\Gamma$-orbits.
\end{theorem}

\begin{proof}
 We consider a $\Gamma$-canonical tree-decomposition $(T,\mathcal V)$, with $\mathcal V=(V_t)_{t\in V(T)}$,  given by \cref{thm: mainCTTD}. 
We let
$A$ denote the set of pairs $\sg{x,y}$ of vertices of $G$ for which there exists an edge-separation $\Sep$ of $(T,\mathcal V)$ such that $x,y\in S$ and $xy\notin E(G)$. By \cref{rem: nbaretes}, as the edge-separations associated to $(T,\mathcal V)$ are tight, $E(T)/\Gamma$ is finite.
As $(T,\mathcal V)$ has finitely bounded adhesion, this implies that there is a finite number of
$\Gamma$-orbits of $A$. We let $\sg{x_1,y_1}, \ldots, \sg{x_\ell, y_\ell}$ 
be representatives of these orbits. For each $j \in [\ell]$ we let $P_{j}$ be a path from $x_j$ to $y_j$ (which always exists, since the edge-separations are tight). For each $\{x,y\}\in A$, we consider the representative $\sg{x_j, y_j}$ in the $\Gamma$-orbit of $\sg{x,y}$, and we define $f(x,y)$ as the image of  the path $P_{j}$ under an automorphism that maps $\sg{x_j, y_j}$ to $\sg{x, y}$. Note that $f(x,y)$ is an $(x,y)$-path in $G$.

We let $G^+$ be the graph obtained from $G$ by adding all possible edges $xy$ such that $xy\in E(G\torso{V_t})$ for some $t\in V(T)$. In other words the edge-set of $G^+$ is exactly $E(G)\uplus A$. For each walk $W$ in $G^+$, we define the walk $f(W)$ in $G$ as the walk obtained from $W$ by replacing every edge $(x,y)$ of $W$ such that $\sg{x,y}\in A$ by $f(x,y)$ (this definition extends the definition of $f$ above, which applied to walks $(x,y)$ of length 1 in $G^+$). For each set of walks $\mathcal S\subseteq \mathcal W(G^+)$, we let $f(\mathcal S):=\sg{f(W)\in \mathcal W(G), W\in \mathcal S}$.

\medskip

\begin{claim}
 \label{clm: fin-presaux}
 For every $W\in \mathcal W(G^+)$, 
 if $W_1, \ldots, W_k\in \mathcal W(G^+)$ generate $W$ in $\mathcal W(G^+)$,
 then $f(W)$ is generated by $f(W_1), \ldots, f(W_k)$ in $\mathcal W(G)$.
\end{claim}
\begin{proofofclaim}
 Let $W\in \mathcal W(G^+)$ be generated by $W_1, \ldots, W_k\in \mathcal W(G^+)$. We prove by induction on the number of operations needed to generate $W$ from $W_1,\ldots, W_k$ that $f(W)$ is generated by the closed walks $f(W_1), \ldots, f(W_k)$.
 
 If $W = W_i$ for some $i\in [k]$, then the result is immediate. 
 Assume that $W$ is obtained from some closed walk $W'$ after performing a rotation on $W'$, and that $W'$ is generated by $W_1, \ldots, W_k$. We write $W=(v_1,\ldots,v_r)$. If $v_1=v_2$ or $v_1v_2\in E(G)$, then $f(W)$ is obtained after performing a single rotation on $f(W')$. If $v_1v_2\in E(G^+)\setminus E(G)= A$, then  $f(W)$ is obtained after performing $|f(v_1,v_2)|$ rotations to $f(W')$.
 In any case if we assume by the induction hypothesis that $f(W')$ is generated by $f(W_1), \ldots, f(W_k)$, we are immediately done. The case where $W$ is obtained after performing a reflection on  $W'$ or adding/removing a repetition on a walk $W'$ is even simpler.

 \smallskip
 
 Now assume that $W$ is the concatenation of two walks $W',W''\in \mathcal W(G^+)$ for which the induction hypothesis holds. Then we observe by definition of $f$ that
 $f(W) = f(W')\cdot f(W'')$.
 Then again we conclude by the induction hypothesis that $f(W)$ is generated by $f(W_1), \ldots, f(W_k).$
 
 Assume now that $W= (v_1, \ldots, v_{i-1}, v_i, v_{i+1}, \ldots,  v_r)$ is obtained from 
 $W'=(v_1, \ldots, \allowbreak v_{i-1}, v_{i+1}, \ldots, v_r)$ 
 after adding the spur $v$ between $v_{i-1}$ and $v_{i+1}$ with $v_{i-1}=v_{i+1}$. We let $x:=v_{i-1}$ and $y:=v_{i}$ and distinguish two cases:
 \begin{itemize}
  \item If $xy\in E(G)$, then we observe that by definition of $f$, $f(W)$ is obtained from $f(W')$ after adding the same spur so we are done using the induction hypothesis on $W'$.
  \item If $xy\in A$, then $f(W)$ must be of the form $U_1\cdot f(x,y)\cdot f(x,y)^{-1} \cdot U_2$, where $U_1\cdot U_2= f(W')$. This means that $f(W)$ can be obtained from $f(W')=U_1\cdot U_2$ by adding $|f(x,y)|$ spurs, hence the induction hypothesis on $W'$ implies that $f(W)$ is generated by $f(W_1), \ldots, f(W_k)$.
 \end{itemize}
 
 Finally assume that $W=(v_1, \ldots, v_{i-1}, v_{i+1}, \ldots, v_r)$ is obtained from $W'=(v_1, \ldots, \allowbreak v_{i-1}, v_i, v_{i+1}, \ldots,  v_r)$ after deleting the spur $v_i$. Again we let $x:=v_{i-1}=v_{i+1}$ and $y:=v_{i}$ and distinguish two cases:
 \begin{itemize}
  \item If $xy\in E(G)$, then as above, $f(W)$ is obtained from $f(W')$ after the removal of a spur and we are immediately done by applying the induction hypothesis on $W'$.
  \item If $xy\in A$, then we claim that $f(W)$ is generated by $f(W')$ as $f(W)=U_1\cdot U_2$, with $U_1:=f((v_1,\ldots,x))$ and $U_2:=f((x,\ldots,v_r))$, and $f(W')=U_1\cdot f(x,y)\cdot f(x,y)^{-1}\cdot U_2$. This shows that $f(W)$ is obtained from $f(W')$ after deleting $|f(x,y)|$ spurs and we can conclude by the induction hypothesis applied to $W'$ that $f(W)$ is generated by $f(W_1), \ldots, f(W_k)$.
 \end{itemize}
This concludes the proof of \cref{clm: fin-presaux}.
\end{proofofclaim}

\begin{claim}
\label{clm: Gen-torsos}
 $\mathcal W(G)$ is generated by $\bigcup_{t\in V(T)} f(\mathcal W(G\torso{V_t}))$.
\end{claim}

\begin{proofofclaim}
 Let $W\in \mathcal W(G)$. First, note that $W$ can be generated in $G^+$ by closed walks of $\bigcup_{t\in V(T)} \mathcal W(G\torso{V_t})$. This comes from the following observation: fix any edge $t_1t_2$ in $T$, with associated separation $\Sep$ in $G$. Then any closed walk $W$ in $G$ can be written as the sum of closed walks in $G^+[Y\cup S]$ and $G^+[S\cup Z]$, followed by the removal of spurs corresponding to the edges of the adhesion $S$. 
 Thus we proved that $\mathcal W(G)$ is generated by $\bigcup_{t\in V(T)} \mathcal W(G\torso{V_t})$ in $G^+$.
 
 Now observe that as for each $W\in \mathcal W(G)$, $f(W)=W$, \cref{clm: fin-presaux} implies that $\mathcal W(G)$ is generated by $\bigcup_{t\in V(T)} f(\mathcal W(G\torso{V_t}))$ in $G$.
\end{proofofclaim}

As $G$ is locally finite, note that for every pair $\sg{x,y}\in A$, there are only finitely many paths of the form $P_j\cdot \gamma$ for some $(j,\gamma)\in [\ell]\times \Gamma$ having $x$ and $y$ as endpoints. For each $\sg{x,y}\in A$, we let $\mathcal P_{x,y}$ denote the set of all such paths and $\mathcal C_{x,y}$ denote the set of all closed walks of the form $P\cdot P'^{-1}$  with $P,P'\in \mathcal P_{x,y}$. Then $\mathcal C_{x,y}$ is finite for each $\sg{x,y}\in A$ and the set
$$\mathcal C:=\bigcup_{\sg{x,y}\in A}\mathcal C_{x,y}$$
is a $\Gamma$-invariant subset of $\mathcal W(G)$ with a finite number of $\Gamma$-orbits. We also consider the set of closed walks $\mathcal C'$ of $\mathcal W(G^+)$ of the form $xPy$ for each $\sg{x,y}\in A$ and $P\in \mathcal P_{x,y}$. Note that $f(\mathcal C')\subseteq \mathcal C$.

By \cref{rem: nbaretes}, $V(T)/\Gamma$ is finite. 
As for every $t\in V(T)$, $G\llbracket V_t\rrbracket$ is either finite or $\Gamma_t$-quasi-transitive planar, by \cite[Theorem 25]{HamannPlanar} the set $\mathcal W(G\torso{V_{t}})$ of closed walks of $G\llbracket V_{t}\rrbracket=G^+[V_t]$ has a generating set of cycles  with finitely many $\Gamma_t$-orbits.
We consider representatives $t_1,\ldots, t_m$ of each of the finitely many orbits $V(T)/\Gamma$, and for each $i\in [m]$, we let $\mathcal{W}_i$ be a finite set of closed walks of $G\llbracket V_{t_i}\rrbracket=G^+[V_{t_i}]$ such that $\mathcal W_i \cdot \Gamma_{t_i}$ generates $\mathcal W(G\torso{V_{t_i}})$.

\begin{claim}
 \label{clm: Gen}
 The set 
 $$\left(\bigcup_{i=1}^m f(\mathcal W_i) \right)\cdot \Gamma \cup \mathcal C$$ 
 generates $\bigcup_{t\in V(T)} f(\mathcal W(G\torso{V_t}))$ in $\mathcal W(G)$.
\end{claim}
\begin{proofofclaim}
 First, note that for each $i\in[m]$, \cref{clm: fin-presaux} implies that $f(\mathcal W_i \cdot \Gamma_{t_i})$ generates $f(\mathcal W(G\torso{V_{t_i}}))$.

 We first let $i\in [m]$ and show that $f(\mathcal W_i)\cdot \Gamma_{t_i}\cup \mathcal C$ generates $f(\mathcal W_i \cdot \Gamma_{t_i})$. We let $W_i\in \mathcal W_i$ and $\gamma\in\Gamma_{t_i}$. One has to be careful as in general the walks $f(W_i\cdot \gamma)$ and $f(W_i)\cdot \gamma$ are not the same. Nevertheless we show that $f(W_i\cdot \gamma)$ is generated by $f(W_i)\cdot \gamma$ and by the walks of $\mathcal C$, which is enough to conclude. Indeed, it is not hard to see that $f(W_i\cdot \gamma)$ is generated in $\mathcal W(G^+)$ by $f(W_i)\cdot \gamma$ and by the walks $xPy\in \mathcal C'$ for each pair $\sg{x,y}\in A$ of consecutive vertices of $W_i$, where $P\in \mathcal P_{x,y}$. Thus by \cref{clm: fin-presaux}, $f(W_i \cdot \gamma)=f(f(W_i\cdot \gamma))$ is generated by $f(W_i)\cdot \gamma=f(f(W_i)\cdot \gamma)$ and by the walks of $f(\mathcal C')\subseteq \mathcal C$.

 To conclude with the proof of the claim we let $t\in V(T)$, and $(i, \gamma)\in [m]\times \Gamma$ be such that $t=t_i\cdot \gamma$. We let $W\in \mathcal W(G\torso{V_t})$. Then there exist $\gamma\in \Gamma$ and $W'\in \mathcal W(G\torso{V_{t_i}})$ such that $W=W'\cdot \gamma$. The exact same arguments as in previous paragraph also apply to prove that $f(W)$ is generated by $f(W')\cdot \gamma$ together with the walks of $\mathcal C$. As we just proved above that $f(W')$ can be generated by finitely elements from $f(\mathcal W_i)\cdot \Gamma_{t_i}\cup \mathcal C$, we conclude that $f(W)$ can be generated by finitely elements from 
 $$(f(\mathcal W_i)\cdot \Gamma_{t_i}\cup \mathcal C)\cdot \gamma \cup \mathcal C\subseteq (f(\mathcal W_i)\cdot \Gamma)\cup \mathcal C,$$
 as desired.
\end{proofofclaim}

Combining \cref{clm: Gen-torsos,clm: Gen}, we obtain that $\mathcal W(G)$ is generated by $\left(\bigcup_{i=1}^m f(\mathcal W_i) \right)\cdot \Gamma \cup \mathcal C$ in $\mathcal W(G)$.
As $\mathcal C$ has a finite number of $\Gamma$-orbits, this  concludes the proof of \cref{thm: finpres}.
\end{proof}

We obtain the following consequence of \cref{thm: finpres} in the group setting, which is a restatement of  \cref{intro:fp}.

\begin{corollary}
 \label{cor: finpres}
 Let $\Gamma$ be a finitely generated $K_\infty$-minor-free group. Then $\Gamma$ is finitely presented.
\end{corollary}

\begin{proof}
 We let $G:=\Cay(\Gamma,S)$ be a locally finite Cayley graph of $\Gamma$ which is $K_\infty$-minor-free, and consider the right action of $\Gamma$ on $G$. Let $\mathcal W$ denote a finite set of representative of the $\Gamma$-orbits of the $\Gamma$-invariant generating set of closed walks of $G$ obtained by applying \cref{thm: finpres} to $G$. Note that for every $W\in \mathcal W$, if $r_W$ denotes the sequence of labels of $W$  with respect to $S$, the set of closed walks of $G$ labeled by $r_W$ is exactly the orbit $W\cdot \Gamma$ of $W$. Thus it is not hard to check that $\langle S\,|\, r_W, W\in \mathcal W\rangle$ is a finite presentation of $\Gamma$.
\end{proof}

 Note that \cref{intro:acc} can now be deduced either from \cref{cor: finpres}, combined with the result of Dunwoody \cite{Dunwoody1985} stating that every finitely presented group is accessible, or from \cref{intro:vacc}, combined with the result of Thomassen and Woess \cite{TW} stating that groups admitting a vertex-accessible Cayley graphs are accessible.

\subsection{The domino problem in minor-excluded groups}\label{sec: Domino}

In this section we  prove the following version of \cref{intro:domino}.

\begin{theorem}
\label{thm: Domino}
Let $\Gamma$ be a finitely generated group excluding the countable clique $K_\infty$ as a minor. Then the domino problem is undecidable on $\Gamma$ if and only if 
\begin{itemize}
    \item $\Gamma$ is one-ended, or
    \item $\Gamma$ has an infinite number of ends and has a  one-ended planar subgroup which is finitely generated.
\end{itemize}
In particular, these situations correspond exactly to the cases where $\Gamma$ is not virtually free.
\end{theorem}

We will need the following results on one-ended planar groups. One-ended planar groups are exactly infinite \emph{planar discontinuous groups}, i.e.\ groups that act properly discontinuously on simplicial complexes homeomorphic to $\mathbb{R}^2$ where every nontrivial automorphism fixes neither a face nor maps an edge $(u,v)$ to its inverse $(v,u)$ (see \cite[Section 4]{Zieschang80} for a complete survey on these groups).

\begin{theorem}[\cite{BN, Fox}]
\label{thm: Zieschang}
 Every planar discontinuous group contains the fundamental group of a closed orientable surface as a subgroup of finite index.
\end{theorem}

It follows that every one-ended planar group contains the fundamental group of a closed orientable surface of genus $g\ge 1$ as a subgroup (if $g=0$ the fundamental group is trivial and cannot be a subgroup of finite index).
We will  use \cref{thm: Zieschang} in combination with the following result of Aubrun et al.\ \cite{Aubrun18surface} on the domino problem in surface groups (note that the case $g=1$ corresponds to the undecidability of the domino problem in $\mathbb{Z}^2$).

\begin{theorem}[\cite{Aubrun18surface}]
\label{thm: aubrun}
    For any $g\ge 1$, the domino problem in the fundamental group of the closed orientable surface of genus $g$ is undecidable.
\end{theorem}

For groups with infinitely many ends we will use the existence of finite graphs of groups (which follows from the accessibility of the groups we consider). The following result of Babai \cite{Babai} will be crucial. 

\begin{theorem}[\cite{Babai}]
\label{thm: Babai}
If $\Gamma'$ is a finitely generated subgroup of $\Gamma$, then there exists some finite generating set $S'$ of $\Gamma'$ such that $\Cay(\Gamma',S')$ is a minor of $\Cay(\Gamma,S)$.
\end{theorem}

Finally we will need the following result of Thomassen, which is tightly related to \cref{thm: Thomassenq4c}, and which will allow us to reduce the $K_\infty$-minor-free case to the planar case.

\begin{theorem}[Theorem 5.7 in \cite{Thomassen92}]
\label{thm: Th92}
 Let $G$ be a locally finite, transitive, non-planar, one-ended graph. Then $G$ contains the countable clique $K_\infty$ as a minor. 
\end{theorem}

We are now ready to prove \cref{thm: Domino}.

\begin{proof}
Assume that $\Gamma$ has a Cayley graph $G=\Cay(\Gamma,S)$ excluding $\Kinf$ as a minor. If $\Gamma$ has 0 end, then it is finite and the domino problem is decidable. If $\Gamma$ has 2 ends then $\Gamma$ is virtually $\mathbb{Z}$ and the domino problem is also decidable. If $\Gamma$ is one-ended, then since $G$ is transitive and excludes the countable clique $K_\infty$ as a minor, it follows from \cref{thm: Th92} that $G$ is (one-ended) planar.  By \cref{thm: Zieschang}, $\Gamma$ contains the fundamental group $\Gamma'$ of a closed surface of genus $g\ge 1$ as a subgroup. By \cref{thm: aubrun}, the domino problem is undecidable for $\Gamma'$. It was proved that for every finitely generated subgroup $\Gamma'$ of a finitely generated group $\Gamma$, if the domino problem is undecidable for $\Gamma'$ then it is also undecidable for $\Gamma$ \cite[Proposition 9.3.30]{Aubrun18subshift}. This implies that if $\Gamma$ is one-ended, the domino problem is undecidable for $\Gamma$.

Assume now that $\Gamma$ has an infinite number of ends. By \cref{intro:acc}, $\Gamma$ is accessible and by \cref{thm: multi-ended}, if $\Gamma$ is not virtually free then one of the (finitely presented) vertex-groups $\Gamma_u$ (which is a subgroup of $\Gamma$) in its associated finite graph of groups is one-ended. By \cref{thm: Babai}, $\Gamma_u$ must also have a locally finite Cayley graph that excludes $\Kinf$ as a minor. By the paragraph above, the domino problem is undecidable for $\Gamma_u$, and since $\Gamma_u$ is a subgroup of $\Gamma$, the domino problem is also undecidable for $\Gamma$. 
\end{proof}
        
 \section{Open problems}\label{sec:ccl}

The first problem we consider is related to \cref{conj: DP}, which states that a group has a decidable domino problem if and only if it is virtually free (or equivalently, it has bounded treewidth). Using the planar case of \cref{thm: Domino}, this conjecture would be a direct consequence of a positive answer to the following problem.

\begin{problem}\label{pro:2}
    Is it true that every finitely generated group which is not virtually free has a  finitely generated one-ended planar subgroup?
\end{problem}

It was pointed out to us by Agelos Georgakopoulos (personal communication) that  the lamplighter group $L=\mathbb{Z}_2\wr \mathbb{Z}$ provides a negative answer to \cref{pro:2}. This follows from a characterization of all subgroups of $L$ by Grigorchuk and Kravchenko \cite{GK14}, which implies that every finitely generated subgroup of $L$ is either finite or isomorphic to another lamplighter group (and the latter cannot be planar).

\medskip

A related problem on the graph theory side is the following. 

\begin{problem}\label{pro:3}
    Is it true that every locally finite quasi-transitive graph of unbounded treewidth contains  a quasi-transitive planar graph of unbounded treewidth as a subgraph?
\end{problem}

It could be the case that Cayley graphs of lamplighter groups also provide a negative answer to \cref{pro:3}, but this is not clear (the assumption of being quasi-transitive and subgraph containment are both much weaker than being a Cayley graph and subgroup containment).

\medskip

A crucial component of our result is the strategy of Grohe \cite{Gro16} to obtain a canonical tree-decomposition up to the contraction of a matching (see also \cref{thm: GroRegion}). At one point we tried to figure out whether the following was true (in the end it turned out that we did not need to answer these questions, but we believe they might be of independent interest).

\begin{problem}\label{pro:4}
    Let $G$ be a locally finite quasi-transitive graph. Is there a proper coloring of $G$ with a finite number of colors such that the colored graph $G$ itself is quasi-transitive (where automorphisms have to preserve the colors of the vertices)?
\end{problem}

\begin{problem}\label{pro:5}
    Let $G$ be a locally finite quasi-transitive graph. Is there an orientation of the edges of $G$ such that the oriented graph $G$ itself is quasi-transitive (where automorphisms have to preserve the orientation of the edges)?
\end{problem}

Note that \cref{pro:5} has a positive answer for Cayley graphs  whose generating set does not contain any element of order 2. 
An example showing that Problem \ref{pro:4} has a negative answer
  was recently constructed by Hamann and M\"oller. It was then observed
  by Abrishami, Esperet and Giocanti, and independently by Norin and
  Przytycki, that a variant of this example could be used to provide a negative
  answer to Problem \ref{pro:5} as well. 

\bigskip

Finally, we started the introduction of the paper by mentioning the following natural question: do graphs avoiding a fixed minor have a tree-decomposition in the spirit of the Graph Minor Structure Theorem of Robertson and Seymour \cite{RS-XVI}, but with the additional constraint that the decomposition is canonical ? We gave a positive answer to this question for quasi-transitive graphs but the general case is still open (although Example \ref{ex: cycle} and Remark \ref{rem:4.2} give natural limitations to the properties of such a canonical tree-decomposition).

 \bigskip

 \paragraph{\bf Acknowledgement} The authors would like to thank Marthe
 Bonamy and Vincent Delecroix for their involvement in the early stages
of the project and all the subsequent discussions. The second author would like to thank Reinhard Diestel, Raphael Jacobs, Paul Knappe and Matthias Hamann for the helpful discussions and for pointing out a simplification of our initial proof of \cref{thm: finpres}. We also thank Johannes Carmesin and Jan Kurkofka for sharing their recent results on canonical decompositions of 3-connected graphs. Finally, we would like to express our gratitude to the reviewers of the journal version of the paper for their careful reading and the large number of corrections and suggestions. The quality of the paper benefited greatly from their feedback.

\bibliographystyle{alpha}
\bibliography{coolnew}

\end{document}